\documentclass[a4paper,11pt,reqno]{amsart}

\usepackage[utf8x]{inputenc}
\usepackage{hyperref}
\usepackage{xcolor}
\hypersetup{
    colorlinks,
    linkcolor={red!50!black},
    citecolor={blue!50!black},
    urlcolor={blue!80!black}
}

\usepackage{microtype}
\usepackage{MnSymbol}
\usepackage{textcomp}
\usepackage{stmaryrd}

\newtheoremstyle{plain}
  {.5\baselineskip\@plus.2\baselineskip\@minus.2\baselineskip}
  {.5\baselineskip\@plus.2\baselineskip\@minus.2\baselineskip\@plus.5em}
  {\slshape}
  {}
  {\bfseries}
  {.}
  { }
  {}

\newtheoremstyle{definition}
  {.5\baselineskip\@plus.2\baselineskip\@minus.2\baselineskip}
  {0.8\baselineskip\@plus.2\baselineskip\@minus.2\baselineskip\@plus.5em}
  {}
  {}
  {\bfseries}
  {.}
  { }
  {}

\makeatletter
\newcommand{\eqnum}{\refstepcounter{equation}\textup{\tagform@{\theequation}}}
\makeatother

\makeatletter \@addtoreset{equation}{section} \makeatother
\renewcommand{\theequation}{\thesection.\arabic{equation}}

\newtheorem{thm}[equation]{Theorem}
\newtheorem{thmX}{Theorem}
\newtheorem{corX}[thmX]{Corollary}

\newtheorem*{thm*}{Theorem}
\newtheorem{lem}[equation]{Lemma}
\newtheorem{cor}[equation]{Corollary}
\newtheorem{prop}[equation]{Proposition}
\newtheorem*{defthm*}{Definition/Theorem}

\theoremstyle{definition}
\newtheorem{defn}[equation]{Definition}
\newtheorem{rem}[equation]{Remark}
\newtheorem{exam}[equation]{Example}
\newtheorem{constr}[equation]{Construction}

\newtheorem{notat}[equation]{Notation}
\newtheorem{warn}[equation]{Warning}
\newtheorem*{exam*}{Example}

\newcommand\arXiv[1]{\href{http://arxiv.org/abs/#1}{arXiv:#1}}

\usepackage{xparse}

\usepackage{xspace}

\newcommand{\changelocaltocdepth}[1]{%
  \addtocontents{toc}{\protect\setcounter{tocdepth}{#1}}%
  \setcounter{tocdepth}{#1}}

\newcommand{\nc}{\newcommand}
\nc{\renc}{\renewcommand}
\nc{\ssec}{\subsection}
\nc{\sssec}{\subsubsection}
\nc{\on}{\operatorname}
\nc{\term}[1]{#1\xspace}

\usepackage{tikz}
\usetikzlibrary{matrix}
\usepackage{tikz-cd}

\tikzset{
  commutative diagrams/.cd,
  arrow style=tikz,
  diagrams={>=latex}}
\makeatletter
\tikzset{
  column sep/.code=\def\pgfmatrixcolumnsep{\pgf@matrix@xscale*(#1)},
  row sep/.code   =\def\pgfmatrixrowsep{\pgf@matrix@yscale*(#1)},
  matrix xscale/.code=%
    \pgfmathsetmacro\pgf@matrix@xscale{\pgf@matrix@xscale*(#1)},
  matrix yscale/.code=%
    \pgfmathsetmacro\pgf@matrix@yscale{\pgf@matrix@yscale*(#1)},
  matrix scale/.style={/tikz/matrix xscale={#1},/tikz/matrix yscale={#1}}}
\def\pgf@matrix@xscale{1}
\def\pgf@matrix@yscale{1}
\makeatother

\usepackage[cmtip, all]{xy}

\usepackage[inline]{enumitem}
\setlist[enumerate,1]{label={(\alph*)},itemsep=\parskip}

\newlist{thmlist}{enumerate}{2}
\setlist[thmlist,1]{
  label={\em(\roman*)}, ref={(\roman*)},
  itemsep=0.5em,
  align=right,widest=vi)}
\setlist[thmlist,2]{
  label={\em(\alph*)}, ref={(\alph*)},
  itemsep=0.75em,
  labelsep=0em,labelindent=0em,leftmargin=*,align=left,widest=vi),
  topsep=0.75em}

\newlist{thmlistbis}{enumerate}{1}
\setlist[thmlistbis,1]{
  label={\em(\roman*~\textit{bis})},
  ref={(\roman*}~\textit{bis}\upshape{)},
  itemsep=0.5em,
  leftmargin=0pt, align=right, widest=vi)}

\newlist{defnlist}{enumerate}{2}
\setlist[defnlist,1]{
  label={(\roman*)}, ref={(\roman*)},
  itemsep=0.5em,
  align=right, widest=vi)}

\setlist[defnlist,2]{
  label={(\alph*)}, ref={(\alph*)},
  itemsep=0.75em,
  labelsep=0em,labelindent=0em,leftmargin=*,align=left,widest=vi),
  topsep=0.75em}

\newlist{inlinelist}{enumerate*}{1}
\setlist[inlinelist,1]{label={(\alph*)}}

\newlist{inlinedefnlist}{enumerate*}{1}
\definecolor{green}{HTML}{38550C}
\setlist[inlinedefnlist,1]{label={\color{green}(\roman*)}}

\nc{\sA}{\ensuremath{\mathcal{A}}\xspace}
\nc{\sB}{\ensuremath{\mathcal{B}}\xspace}
\nc{\sC}{\ensuremath{\mathcal{C}}\xspace}
\nc{\sD}{\ensuremath{\mathcal{D}}\xspace}
\nc{\sE}{\ensuremath{\mathcal{E}}\xspace}
\nc{\sF}{\ensuremath{\mathcal{F}}\xspace}
\nc{\sG}{\ensuremath{\mathcal{G}}\xspace}
\nc{\sH}{\ensuremath{\mathcal{H}}\xspace}
\nc{\sI}{\ensuremath{\mathcal{I}}\xspace}
\nc{\sJ}{\ensuremath{\mathcal{J}}\xspace}
\nc{\sK}{\ensuremath{\mathcal{K}}\xspace}
\nc{\sL}{\ensuremath{\mathcal{L}}\xspace}
\nc{\sM}{\ensuremath{\mathcal{M}}\xspace}
\nc{\sN}{\ensuremath{\mathcal{N}}\xspace}
\nc{\sO}{\ensuremath{\mathcal{O}}\xspace}
\nc{\sP}{\ensuremath{\mathcal{P}}\xspace}
\nc{\sQ}{\ensuremath{\mathcal{Q}}\xspace}
\nc{\sR}{\ensuremath{\mathcal{R}}\xspace}
\nc{\sS}{\ensuremath{\mathcal{S}}\xspace}
\nc{\sT}{\ensuremath{\mathcal{T}}\xspace}
\nc{\sU}{\ensuremath{\mathcal{U}}\xspace}
\nc{\sV}{\ensuremath{\mathcal{V}}\xspace}
\nc{\sW}{\ensuremath{\mathcal{W}}\xspace}
\nc{\sX}{\ensuremath{\mathcal{X}}\xspace}
\nc{\sY}{\ensuremath{\mathcal{Y}}\xspace}
\nc{\sZ}{\ensuremath{\mathcal{Z}}\xspace}

\nc{\bA}{\ensuremath{\mathbf{A}}\xspace}
\nc{\bB}{\ensuremath{\mathbf{B}}\xspace}
\nc{\bC}{\ensuremath{\mathbf{C}}\xspace}
\nc{\bD}{\ensuremath{\mathbf{D}}\xspace}
\nc{\bE}{\ensuremath{\mathbf{E}}\xspace}
\nc{\bF}{\ensuremath{\mathbf{F}}\xspace}
\nc{\bG}{\ensuremath{\mathbf{G}}\xspace}
\nc{\bH}{\ensuremath{\mathbf{H}}\xspace}
\nc{\bI}{\ensuremath{\mathbf{I}}\xspace}
\nc{\bJ}{\ensuremath{\mathbf{J}}\xspace}
\nc{\bK}{\ensuremath{\mathbf{K}}\xspace}
\nc{\bL}{\ensuremath{\mathbf{L}}\xspace}
\nc{\bM}{\ensuremath{\mathbf{M}}\xspace}
\nc{\bN}{\ensuremath{\mathbf{N}}\xspace}
\nc{\bO}{\ensuremath{\mathbf{O}}\xspace}
\nc{\bP}{\ensuremath{\mathbf{P}}\xspace}
\nc{\bQ}{\ensuremath{\mathbf{Q}}\xspace}
\nc{\bR}{\ensuremath{\mathbf{R}}\xspace}
\nc{\bS}{\ensuremath{\mathbf{S}}\xspace}
\nc{\bT}{\ensuremath{\mathbf{T}}\xspace}
\nc{\bU}{\ensuremath{\mathbf{U}}\xspace}
\nc{\bV}{\ensuremath{\mathbf{V}}\xspace}
\nc{\bW}{\ensuremath{\mathbf{W}}\xspace}
\nc{\bX}{\ensuremath{\mathbf{X}}\xspace}
\nc{\bY}{\ensuremath{\mathbf{Y}}\xspace}
\nc{\bZ}{\ensuremath{\mathbf{Z}}\xspace}

\nc{\bbA}{\ensuremath{\mathbb{A}}\xspace}
\nc{\bbB}{\ensuremath{\mathbb{B}}\xspace}
\nc{\bbC}{\ensuremath{\mathbb{C}}\xspace}
\nc{\bbD}{\ensuremath{\mathbb{D}}\xspace}
\nc{\bbE}{\ensuremath{\mathbb{E}}\xspace}
\nc{\bbF}{\ensuremath{\mathbb{F}}\xspace}
\nc{\bbG}{\ensuremath{\mathbb{G}}\xspace}
\nc{\bbH}{\ensuremath{\mathbb{H}}\xspace}
\nc{\bbI}{\ensuremath{\mathbb{I}}\xspace}
\nc{\bbJ}{\ensuremath{\mathbb{J}}\xspace}
\nc{\bbK}{\ensuremath{\mathbb{K}}\xspace}
\nc{\bbL}{\ensuremath{\mathbb{L}}\xspace}
\nc{\bbM}{\ensuremath{\mathbb{M}}\xspace}
\nc{\bbN}{\ensuremath{\mathbb{N}}\xspace}
\nc{\bbO}{\ensuremath{\mathbb{O}}\xspace}
\nc{\bbP}{\ensuremath{\mathbb{P}}\xspace}
\nc{\bbQ}{\ensuremath{\mathbb{Q}}\xspace}
\nc{\bbR}{\ensuremath{\mathbb{R}}\xspace}
\nc{\bbS}{\ensuremath{\mathbb{S}}\xspace}
\nc{\bbT}{\ensuremath{\mathbb{T}}\xspace}
\nc{\bbU}{\ensuremath{\mathbb{U}}\xspace}
\nc{\bbV}{\ensuremath{\mathbb{V}}\xspace}
\nc{\bbW}{\ensuremath{\mathbb{W}}\xspace}
\nc{\bbX}{\ensuremath{\mathbb{X}}\xspace}
\nc{\bbY}{\ensuremath{\mathbb{Y}}\xspace}
\nc{\bbZ}{\ensuremath{\mathbb{Z}}\xspace}

\nc{\mrm}[1]{\ensuremath{\mathrm{#1}}\xspace}
\nc{\mit}[1]{\ensuremath{\mathit{#1}}\xspace}
\nc{\mbf}[1]{\ensuremath{\mathbf{#1}}\xspace}
\nc{\mcal}[1]{\ensuremath{\mathcal{#1}}\xspace}
\nc{\msc}[1]{\ensuremath{\mathscr{#1}}\xspace}

\nc{\sub}{\subseteq}
\nc{\too}{\longrightarrow}
\nc{\hook}{\hookrightarrow}
\nc{\hooklongrightarrow}{\lhook\joinrel\longrightarrow}
\nc{\hooklong}{\hooklongrightarrow}
\nc{\hooklongleftarrow}{\longleftarrow\joinrel\rhook}
\nc{\twoheadlongrightarrow}{\relbar\joinrel\twoheadrightarrow}
\nc{\longrightleftarrows}{\ \raisebox{0.3ex}{\(\mathrel{\substack{\xrightarrow{\rule{1em}{0em}} \\[-1ex] \xleftarrow{\rule{1em}{0em}}}}\)}\ }

\renc{\ge}{\geqslant}
\renc{\le}{\leqslant}

\nc{\id}{\mathrm{id}}

\DeclareMathOperator{\rk}{\mathrm{rk}}
\DeclareMathOperator{\Hom}{\on{Hom}}
\nc{\uHom}{\underline{\smash{\Hom}}}
\DeclareMathOperator{\Maps}{\on{Maps}}
\DeclareMathOperator{\Aut}{\on{Aut}}
\DeclareMathOperator{\End}{\on{End}}

\nc{\uEnd}{\underline{\smash{\End}}}
\DeclareMathOperator{\codim}{\on{codim}}
\nc{\colim}{\varinjlim}
\renc{\lim}{\varprojlim}
\nc{\Cofib}{\on{Cofib}}
\nc{\Fib}{\on{Fib}}
\nc{\initial}{\varnothing}
\nc{\op}{\mathrm{op}}

\DeclareMathOperator*{\fibprod}{\times}
\DeclareMathOperator*{\fibcoprod}{\operatorname{\sqcup}}

\renc{\setminus}{\smallsetminus}

\usepackage{mathtools}
\DeclarePairedDelimiter\abs{\lvert}{\rvert}%

\newcommand{\thmref}[1]{Theorem~\ref{#1}}

\newcommand{\secref}[1]{Sect.~\ref{#1}}
\newcommand{\ssecref}[1]{Subsect. ~\ref{#1}}
\newcommand{\sssecref}[1]{(\ref{#1})}
\newcommand{\lemref}[1]{Lemma~\ref{#1}}
\newcommand{\propref}[1]{Proposition~\ref{#1}}
\newcommand{\corref}[1]{Corollary~\ref{#1}}

\newcommand{\remref}[1]{Remark~\ref{#1}}
\newcommand{\defnref}[1]{Definition~\ref{#1}}

\renewcommand{\eqref}[1]{(\ref{#1})}
\newcommand{\constrref}[1]{Construction~\ref{#1}}

\newcommand{\examref}[1]{Example~\ref{#1}}

\newcommand{\notatref}[1]{Notation~\ref{#1}}
\newcommand{\itemref}[1]{\ref{#1}}

\nc{\h}{\mrm{h}}
\nc{\pt}{{\mrm{pt}}}
\nc{\un}{\mbf{1}}
\nc{\Pic}{\on{Pic}}
\nc{\pr}{{\mrm{pr}}}
\nc{\pur}{\mrm{pur}}

\nc{\A}{\bA}
\renc{\P}{\bP}
\nc{\Spec}{\on{Spec}}
\nc{\Asp}{\mrm{Rep}}
\nc{\Sm}{\mrm{Sm}}
\nc{\Th}{\on{Th}}
\nc{\vb}[1]{\langle #1\rangle}

\nc{\D}{\on{\mbf{D}}}
\nc{\Qcoh}{\on{Qcoh}}
\nc{\Dqc}{\on{\mbf{D}}_{\mrm{qc}}}
\nc{\bDelta}{\mathbf{\Delta}}
\nc{\Cech}{\textnormal{\v{C}}}
\nc{\Dperf}{\on{\mbf{D}}_{\mrm{perf}}}
\nc{\cl}{{\mrm{cl}}}
\nc{\CH}{\on{A}}
\nc{\et}{\mrm{\acute{e}t}}
\renc{\H}{\on{A}}
\nc{\Z}{\bZ}
\nc{\Q}{\bQ}
\nc{\K}{{\on{K}}}
\nc{\KB}{\K^{\mrm{B}}}
\nc{\G}{{\on{G}}}
\nc{\KH}{\mrm{KH}}
\nc{\Einfty}{{\sE_\infty}}
\renc{\sp}{\mrm{sp}}
\nc{\red}{\mrm{red}}
\nc{\RGamma}{\on{\mrm{R}\Gamma}}
\renc{\L}{\mrm{L}}
\nc{\otimesL}{\otimes^\bL}
\nc{\fibprodR}{\fibprod^\bR}
\nc{\Ex}{\mrm{Ex}}
\nc{\GL}{\mrm{GL}}

\nc{\Nis}{\mathrm{Nis}}
\nc{\LNis}{\on{\mrm{L}_\Nis}}
\nc{\htp}{{\A^1}}
\nc{\Lhtp}{\on{\mrm{L}_\htp}}
\nc{\MotSpc}{{\mbf{H}}}
\nc{\SH}{{\mbf{SH}}}
\nc{\DM}{{\mbf{DM}}}

\nc{\KGL}{{\mrm{KGL}}}
\nc{\MGL}{{\mrm{MGL}}}

\nc{\dash}{{\textnormal{-}}}
\nc{\dashMod}[1]{#1\dash\textrm{mod}}
\nc{\MGLmod}{\MGL\dash\textbf{mod}}

\nc{\fr}{{\mrm{fr}}}
\nc{\SHfr}{\SH^\fr}

\nc{\eul}{\mrm{eul}}
\nc{\gys}{\mrm{gys}}

\nc{\Bor}{{\triangleleft}}
\nc{\Lis}{\mrm{Lis}}
\nc{\RKE}{\mrm{RKE}}
\nc{\Shv}{\on{Shv}}

\nc{\rep}{\mrm{rep}}

\nc{\piz}{\on{\pi}_0}

\nc{\rightrightrightarrows}
  {\mathrel{\substack{\textstyle\rightarrow\\[-0.6ex]
   \textstyle\rightarrow \\[-0.6ex]
   \textstyle\rightarrow}}}
\nc{\rightrightrightrightarrows}
  {\mathrel{\substack{\textstyle\rightarrow\\[-0.6ex]
   \textstyle\rightarrow \\[-0.6ex]
   \textstyle\rightarrow \\[-0.6ex]
   \textstyle\rightarrow}}}

\nc{\Ccoh}{\mrm{C}^\bullet}
\nc{\BM}{\mrm{BM}}
\nc{\CBM}{\mrm{C}^{\BM}_\bullet}

\nc{\inftyCat}{\term{$\infty$-category}}
\nc{\inftyCats}{\term{$\infty$-categories}}

\nc{\inftyGrpd}{\term{$\infty$-groupoid}}
\nc{\inftyGrpds}{\term{$\infty$-groupoids}}

\nc{\ani}{\term{animum}}
\nc{\anis}{\term{anima}}

\nc{\fund}{\term{fundamental}}
\nc{\funds}{\term{fundamentals}}

\nc{\qfund}{\term{quasi-fundamental}}

\title{Generalized cohomology theories for algebraic stacks}

\author[A.\,A. Khan]{Adeel~A.~Khan}
\author[C. Ravi]{Charanya Ravi}

\date{2024-10-06}

\makeatletter
\def\l@subsection{\@tocline{2}{0pt}{4pc}{6pc}{}}
\makeatother

\begin{document}

\begin{abstract}
  We extend the stable motivic homotopy category of Voevodsky to the class of scalloped algebraic stacks, and show that it admits the formalism of Grothendieck's six operations.
  Objects in this category represent generalized cohomology theories for stacks like algebraic K-theory, as well as new examples like genuine motivic cohomology and algebraic cobordism.
  These cohomology theories admit Gysin maps and satisfy homotopy invariance, localization, and Mayer--Vietoris.
  For example, we deduce that homotopy K-theory satisfies cdh descent on scalloped stacks.
  We also prove a fixed point localization formula for torus actions.
  
  Finally, the construction is contrasted with a ``lisse-extended'' stable motivic homotopy category, defined for arbitrary stacks: we show for example that lisse-extended motivic cohomology of quotient stacks is computed by the equivariant higher Chow groups of Edidin--Graham, and we also get a good new theory of Borel-equivariant algebraic cobordism.
  Moreover, the lisse-extended motivic homotopy type is shown to recover all previous constructions of motives of stacks.
\end{abstract}

\maketitle

\renewcommand\contentsname{\vspace{-1cm}}
\tableofcontents

\setlength{\parindent}{0em}
\parskip 0.75em


\section{Introduction}

  \subsection{Motivic homotopy theory of stacks}

    Motivic homotopy theory provides a framework for the study of generalized or extraordinary cohomology theories in algebraic geometry, such as motivic cohomology, algebraic K-theory, and algebraic cobordism.
    Objects of the motivic stable homotopy category $\SH(X)$, over a scheme $X$, are ``generalized sheaves'' that can be taken as coefficients for cohomology.

    In this paper we are interested in generalized cohomology theories on algebraic stacks.
    To that end, we introduce an extension of the motivic stable homotopy category to a large class of algebraic stacks, called \emph{scalloped} stacks (see \ssecref{ssec:intro/pt1} below), which includes for instance tame Deligne--Mumford or Artin stacks with separated diagonal as well as quotients of qcqs algebraic spaces by nice linear algebraic groups.
    Our first main result is as follows (see \thmref{thm:exc} and \examref{exam:SH}):

    \begin{thmX}\label{thm:intro/six}
      The assignment $\sX \mapsto \SH(\sX)$, together with the formalism of six operations, extends from qcqs\footnote{quasi-compact and quasi-separated} algebraic spaces to scalloped algebraic stacks.
      More precisely, we have the following operations:
      \begin{thmlist}
        \item
        For every scalloped stack $\sX$, a pair of adjoint bifunctors $(\otimes, \uHom)$.

        \item
        For every morphism of scalloped stacks $f : \sX \to \sY$, an adjoint pair
        \[
          f^* : \SH(\sY) \to \SH(\sX),
          \quad
          f_* : \SH(\sX) \to \SH(\sY).
        \]
        
        \item
        For every representable morphism of finite type $f : \sX \to \sY$ between scalloped stacks $\sX$ and $\sY$, an adjoint pair
        \[
          f_! : \SH(\sX) \to \SH(\sY),
          \quad
          f^! : \SH(\sY) \to \SH(\sX).
        \]
      \end{thmlist}
      Moreover, these satisfy various identities including the base change and projection formulas, homotopy invariance, purity isomorphism, and localization triangle.
    \end{thmX}

    In the case of noetherian schemes, the six functor formalism on $\SH$ was constructed in the work of Voevodsky, Ayoub, and Cisinski--Déglise (see \cite{VoevodskyCross,AyoubThesis,CisinskiDegliseBook}).
    For a self-contained account in the generality of qcqs algebraic spaces, see \cite{KhanSix}.
    Our proof of \thmref{thm:intro/six} is logically independent of these earlier works.

  \subsection{Genuine cohomology theories}
  \label{ssec:intro/coh}

    Given a motivic spectrum $\sF \in \SH(\sX)$ over a scalloped stack $\sX$, we define the cohomology spectra of $\sX$ with coefficients in $\sF$ as the mapping spectra
    \[ \Ccoh(\sX, \sF) = \Maps_{\SH(\sX)}\big(\un_\sX, \sF\big). \]
    Given any K-theory class $\alpha \in \K(\sX)$, we write $\sF\vb{\alpha}$ for the Thom twist\footnote{%
      For oriented examples (such as motivic cohomology, algebraic K-theory, and algebraic cobordism), a choice of orientation determines isomorphisms $\sF\vb{\alpha} \simeq \sF\vb{r} \simeq \sF(r)[2r]$, where $r$ is the virtual rank of $\alpha$ (when viewed as a K-theory class, $r = [\sO_\sX] + \cdots + [\sO_\sX] = [\sO^r_\sX]$).
    } by $\alpha$.
    \thmref{thm:intro/six} yields (see \ssecref{ssec:coh/op}):
    
    \begin{corX}\label{cor:intro/cohomology}
      Cohomology with coefficients in $\sF \in \SH(\sX)$ has the following properties:

      \begin{thmlist}
        \item\emph{Functoriality}.\label{item:iyzbho}
        For every representable morphism $f : \sX' \to \sX$, there are inverse image maps
        \[
          f^* : \Ccoh(\sX, \sF) \to \Ccoh(\sX', \sF).
        \]
        If $f$ is smooth and proper, then there is also a Gysin map
        \[
          f_! : \Ccoh(\sX', \sF) \to \Ccoh(\sX, \sF)\vb{-\Omega_{\sX'/\sX}}
        \]
        where $\Omega_{\sX'/\sX}$ denotes the relative cotangent sheaf.
        These are functorial and satisfy base change and projection formulas.
        
        \item\emph{Homotopy invariance}.
        For every scalloped stack $\sX$ and every vector bundle $p : \sE \to \sX$, the inverse image map
        \[ p^* : \Ccoh(\sX, \sF) \to \Ccoh(\sE, \sF) \]
        is invertible.
        
        \item\emph{Localization}.\label{item:0a7sfd0p1}
        For every scalloped stack $\sX$ and every closed immersion $i : \sZ \to \sX$ with quasi-compact open complement $j : \sU \to \sX$, there is an exact triangle
        \begin{equation*}
          \Ccoh_\sZ(\sX, \sF)
          \to \Ccoh(\sX, \sF)
          \xrightarrow{j^*} \Ccoh(\sU, \sF),
        \end{equation*}
        where
        \[
          \Ccoh_\sZ(\sX, \sF)
          = \Maps_{\SH(\sX)}\big(i_*(\un_\sZ), \sF\big)
        \]
        is cohomology with support in $\sZ$.
        
        \item\emph{Mayer--Vietoris/cdh descent}.\label{item:a07gsf01}
        Let $f : \sX' \to \sX$ be a representable étale morphism (resp. a representable proper morphism) of scalloped stacks which induces an isomorphism away from a quasi-compact open substack $\sY \subseteq \sX$ (resp. a closed substack $\sY \subseteq \sX$).
        Then there is a homotopy cartesian square
        \[ \begin{tikzcd}
          \Ccoh(\sX, \sF) \ar{r}\ar{d}{f^*}
          & \Ccoh(\sY, \sF) \ar{d}
          \\
          \Ccoh(\sX', \sF) \ar{r}
          & \Ccoh(f^{-1}(\sY), \sF).
        \end{tikzcd} \]
      \end{thmlist}
    \end{corX}

    For a quotient stack $\sX = [X/G]$, where $G$ is a nice linear algebraic group acting on a qcqs algebraic space $X$, the genuine cohomology of $\sX$ in particular defines a notion of genuine $G$-equivariant cohomology of $X$.
    Moreover, $X$ can itself be a scalloped stack (e.g. a tame Deligne--Mumford stack with separated diagonal) so that we can also make sense of genuine equivariant cohomology of scalloped stacks.

    We are primarily interested in the cohomology theories arising from three examples of motivic spectra over $\sX$: the algebraic K-theory spectrum $\KGL_\sX$, the (integral) motivic cohomology spectrum $\bZ_\sX$, and the algebraic cobordism spectrum $\MGL_\sX$.
    
    Cohomology with coefficients in $\KGL_\sX$ recovers (the $\A^1$-invariant version of) the well-known algebraic K-theory of stacks: for every scalloped stack $\sX$, there is a canonical isomorphism
    \begin{equation}\label{eq:intro/KGL}
      \Ccoh(\sX, \KGL) \simeq \KH(\sX),
    \end{equation}
    where the right-hand side is the homotopy invariant K-theory spectrum $\KH(\sX)$ as studied in \cite{KrishnaRavi,HoyoisKrishna} (and in \cite{WeibelKH} for schemes).
    In particular, if $\sX$ is nonsingular, then we have
    \[ \Ccoh(\sX, \KGL) \simeq \K(\sX) \simeq \G(\sX) \]
    where $\K(\sX)$ is the Thomason--Trobaugh K-theory spectrum of perfect complexes on $\sX$ and $\G(\sX)$ is the Quillen K-theory spectrum of coherent sheaves on $\sX$.
    See \ssecref{ssec:KH}.

    In the case of $\bZ_\sX$ and $\MGL_\sX$, we get new extensions of motivic cohomology and cobordism to stacks, which are ``genuine'' refinements of previously known cohomology theories even in the case of quotient stacks.
    The word ``genuine'' here is used in the sense of \emph{genuine} equivariant homotopy theory (see \cite{Segal,HillHopkinsRavenel,NikolausScholze}), as opposed to \emph{Borel}-equivariant homotopy theory; compare \ssecref{ssec:intro/limit} below and see also the discussion in \ssecref{ssec:related}.

    There are also ``quadratic'', $\mrm{SL}$-oriented, refinements of these three theories: hermitian K-theory $\mrm{KQ}_\sX$ (\remref{rem:KQ}), Milnor--Witt motivic cohomology $\widetilde{\bZ}_\sX$ (\remref{rem:Ztilde}), and special linear algebraic cobordism $\mrm{MSL}_\sX$ (\remref{rem:MSL}).

  \subsection{Fixed point localization}

    Let $k$ be a field and let $T = \bG_{m,k}$ be the multiplicative group over $\Spec(k)$.
    Given a motivic spectrum $\sF \in \SH(BT)$, we can define $T$-equivariant Borel--Moore homology with coefficients in $\sF$ by the formula:
    \[
      \CBM([X/T]_{/BT}, \sF)
      := \Maps_{\SH([X/T])}\big( \un, f^!(\sF) \big),
    \]
    where $X$ is a qcqs algebraic space with $T$-action and $f : [X/T] \to BT$ is the projection.
    We prove an analogue of Thomason's concentration theorem (see \cite[Thm.~2.1]{ThomasonLefschetz}) in this setting.
    It relates the equivariant Borel--Moore homology of an algebraic space with that of its fixed locus:

    \begin{thmX}[Concentration]\label{thm:intro/conc}
      Let $X$ be a $T$-equivariant algebraic space of finite type over $k$ and denote by $i : X^{T} \to X$ the inclusion of the locus of fixed points.
      Then for every motivic spectrum $\sF \in \SH(BT)$, the map
      \[ i_* : \CBM((X^T \times BT)_{/BT}, \sF)\vb{\ast} \to \CBM([X/T]_{/BT}, \sF)\vb{\ast} \]
      becomes invertible after inverting the Euler classes of powers of the tautological line bundle $[\A^1/\bG_m]$.
    \end{thmX}

    See \corref{cor:conc}.
    The notation $C\vb{\ast}$ denotes the direct sum of $C\vb{\alpha}$ over all K-theory classes $\alpha$ (see \remref{rem:graded}).

    As an application of \thmref{thm:intro/conc}, one can prove (virtual) Atiyah--Bott localization and wall-crossing formulas in this context, following the arguments of \cite{virloc}.
    This suggests that there are genuine-equivariant counterparts to Joyce's machine producing enumerative invariants out of abelian categories \cite{JoyceWall} (compare \cite{LiuJoyce}), which we plan to investigate in the future.

  \subsection{Lisse-extended cohomology theories}
  \label{ssec:intro/limit}

    A different way to extend $\SH$ and generalized cohomology theories to stacks is via the following formal procedure.

    Given any qcqs algebraic stack $\sX$ (not necessarily scalloped), denote by $\Lis_\sX$ the \inftyCat of pairs $(U, u : U \to \sX)$, where $U$ is a qcqs algebraic space and $u : U \to \sX$ is a smooth morphism.
    The \emph{lisse extension} $\SH_\Bor(\sX)$ is the homotopy limit of \inftyCats
    \[
      \SH_\Bor(\sX) = \lim_{(U,u) \in \Lis_\sX} \SH(U),
    \]
    over the $*$-inverse image functors.
    The motivic cohomology, algebraic K-theory, and algebraic cobordism spectra immediately give rise to lisse-extended motivic spectra
    \[ \bZ^\Bor_\sX, ~\KGL^\Bor_\sX, ~\MGL^\Bor_\sX \in \SH_\Bor(\sX) \]
    over $\sX$, simply by virtue of stability under $*$-inverse image.
    If $\sX$ is scalloped, then these are moreover the images of their genuine counterparts by a canonical functor (which commutes with colimits and $*$-inverse image)
    \[
      \SH(\sX) \to \SH_\Bor(\sX).
    \]
    However, this functor is far from being an equivalence, so that the corresponding cohomology theories are very different (as the example below of K-theory shows).

    In fact, we show that for quotient stacks, cohomology with coefficients in any of the lisse-extended cohomology theories above can be computed via Totaro and Morel--Voevodsky's algebraic approximation of the Borel construction (see \thmref{thm:Borel}).
    For example, in the case of motivic cohomology we have:

    \begin{thmX}\label{thm:intro/Bor}
      Let $G$ be a linear algebraic group over a field $k$ of characteristic zero.
      Let $X$ be a smooth $G$-quasi-projective $k$-scheme.
      Then for every $n,s\in\bZ$ there are canonical isomorphisms
      \begin{equation*}
        \pi_s \Ccoh_\Bor([X/G], \bZ)\vb{n}
        \simeq \CH^n_G(X, s)
      \end{equation*}
      where on the right-hand side are the Edidin--Graham equivariant higher Chow groups \cite{EdidinGraham}.
    \end{thmX}

    See Example~\ref{exam:equiv Chow}.
    The result also holds for fields of characteristic $p>0$, up to inverting $p$.

    \begin{exam}\label{exam:0pa8hfds}
      In the case of $\KGL^\Bor_{[X/G]}$, the canonical map
      \[
        \piz \K([X/G]) \to \piz \Ccoh_\Bor([X/G], \KGL),
      \]
      induced by the functor $\SH({[X/G]}) \to \SH_\Bor({[X/G]})$ (for $G$ nice), is not bijective if $G$ is nontrivial.
      In fact, it exhibits the target as a completion of the source.
      See \examref{exam:equiv k}.
    \end{exam}

    \begin{rem}
      For general coefficients $\sF$, \thmref{thm:Borel} gives isomorphisms
      \[
        \Ccoh_\Bor([X/G], \sF)\vb{n}
        \simeq \lim_i \Ccoh([X/G] \fibprod_{BG} U_i, \sF)\vb{n}
      \]
      where $(U_i)_i$ is a sequence of algebraic approximations to the Borel construction.
      On $\pi_0$ we have surjections (see \eqref{eq:anfp1n0} for notation)
      \[
        \H^n_\Bor([X/G], \sF)
        \twoheadrightarrow \lim_i \H^n([X/G] \fibprod_{BG} U_i, \sF)
      \]
      which however we do not know to be injective for general $\sF$.\footnote{%
        Note added in final revision: In the subsequent paper \cite{equilisse} we show that this holds when \begin{inlinelist}
          \item $\sF$ is eventually coconnective with respect to the cohomological t-structure, or
          \item $\sF$ is the \emph{rationalized} algebraic cobordism or algebraic K-theory spectrum.
        \end{inlinelist}
      }
      For example, our proof that this holds in the case of motivic cohomology relies on a vanishing property which does not hold e.g. in algebraic cobordism.

      In particular, although the right-hand side has been considered in the case of algebraic cobordism in \cite{HellerMalagonLopez,KrishnaCobordism} (and in an abstract setting in \cite[Cor.~3.8]{KiemPark}), it is not known to admit right-exact localization sequences (see Footnote~\ref{fn:1390gy1}).
      In contrast, lisse-extended cobordism does have right-exact localization sequences which in fact even extend to long-exact sequences using the higher groups (see \propref{prop:07ug71}).\footnote{%
        Note added in final revision: In the subsequent paper \cite{equilisse} we show that lisse-extended bordism (of possibly singular schemes) agrees with the bordism theory of \cite{HellerMalagonLopez} with rational coefficients.
        In particular, the latter also admits right-exact localization sequences.
      }
      In general, we regard the lisse extension as a good way to define ``Borel-type'' extensions of arbitrary generalized cohomology theories.
    \end{rem}

    \begin{rem}
      \thmref{thm:Borel} is deduced from a stronger comparison of motivic stable homotopy types.
      For example, the lisse-extended motivic stable homotopy type of a classifying stack $BG$ is compared with the Morel--Voevodsky construction (see \thmref{thm:Borel M}).
      Moreover, we show that the lisse-extended motivic homotopy type recovers (and generalizes) all previous constructions of motives of stacks found in the literature \cite{ToenMotive,Choudhury,HoskinsPepinLehalleur,RicharzScholbach,ChoudhuryDeshmukhHogadi} (see Remarks~\ref{rem:CDH comparison}, \ref{rem:ysbobql} and \ref{rem:UO1b2hov}).
    \end{rem}

  \subsection{Derived stacks and virtual functoriality}

    Throughout the paper, we work in the setting of derived algebraic geometry.
    In particular, \thmref{thm:intro/six} and \corref{cor:intro/cohomology} remain valid for scalloped \emph{derived} stacks.
    Working in this generality allows us to construct an enhanced functoriality for our cohomology theories (see \thmref{thm:gys}):

    \begin{thmX}
      Let $f : \sX \to \sY$ be a quasi-smooth, representably smoothable\footnote{
        Here representably smoothable means that $f$ admits a global factorization through an unramified representable morphism followed by a smooth representable morphism (see \defnref{defn:smoothable}).
      } morphism of scalloped derived stacks with affine diagonal.
      Then there exists a natural transformation
      \[ \gys_{\sX/\sY} : f^*\vb{\sL_{\sX/\sY}} \to f^! \]
      of functors $\SH(\sY) \to \SH(\sX)$, where $\sL_{\sX/\sY}$ is the relative cotangent complex.
      If $f$ is smooth, then $\gys_{\sX/\sY}$ is the purity isomorphism (\thmref{thm:purity}).
      Moreover, for every motivic spectrum $\sF \in \SH(\sY)$, for $f : \sX \to \sY$ proper quasi-smooth and representably smoothable, the Gysin transformation yields Gysin maps
      \[
        f_! : \Ccoh(\sX, \sF) \to \Ccoh(\sY, \sF)\vb{-\sL_{\sX/\sY}}
      \]
      extending those of \corref{cor:intro/cohomology}.
    \end{thmX}

    \emph{Quasi-smoothness} is a derived version of the notion of local complete intersection morphism, which in particular gives rise on classical truncations to a relative perfect obstruction theory in the sense of \cite{BehrendFantechi}.
    For every scalloped derived stack $\sX$, the inclusion of the classical truncation $i : \sX_\cl \to \sX$ induces canonical isomorphisms
    \begin{equation}\label{eq:0ap-s-h}
      i^* : \Ccoh(\sX, \sF) \simeq \Ccoh(\sX_\cl, \sF),
    \end{equation}
    under which the Gysin functoriality for quasi-smooth morphisms can be interpreted as ``virtual'' functoriality at the level of classical truncations (cf. \cite{Manolache}).

    Note that the Gysin transformation $\gys_{\sX/\sY}$ for any quasi-smooth, representably smoothable, morphism $f : \sX \to \sY$, gives rise to a canonical ``bivariant'' \emph{virtual fundamental class} which lives in the relative Borel--Moore homology
    \[ \Ccoh(\sX, f^!(\sF))\vb{-\sL_{\sX/\sY}} \]
    for any motivic ring spectrum $\sF \in \SH(\sY)$.
    These groups define a bivariant theory in the sense of Fulton--MacPherson \cite{FultonMacPherson}, and the relative virtual fundamental class can be viewed as an ``orientation'' for $f$ in the sense of \emph{op. cit}.

  \subsection{Homotopy invariant K-theory}

    When applied to cohomology with coefficients in $\KGL$ \eqref{eq:intro/KGL}, the isomorphisms \eqref{eq:0ap-s-h} yield the following corollary (see also \remref{rem:0afsgdh013}), which was obtained in 2018 in the early stages of this project.
    Since then, it has also been independently reproven\footnote{%
      Although the result in \emph{op. cit.} is not stated in this generality, the same proof applies to scalloped derived stacks.
      In fact, \cite[Thm.~5.2.2]{ElmantoSosnilo} generalizes to show that any ``truncating'' invariant of stable \inftyCats satisfies derived nilpotent invariance for scalloped derived stacks.
    } via categorical methods by Elmanto and Sosnilo \cite{ElmantoSosnilo}.

    \begin{corX}\label{cor:intro/KH}
      For every scalloped derived stack $\sX$ with classical truncation $\sX_\cl$, there is a canonical isomorphism of spectra
      \[ \KH(\sX) \simeq \KH(\sX_\cl). \]
    \end{corX}

    The original modest motivation for this work was to give a proof of the following theorem, a slight generalization of a result of Hoyois--Krishna (see \cite{HoyoisKH,HoyoisKrishna}, and \cite{CisinskiKH} for the case of schemes).

    \begin{corX}\label{cor:intro/KH cdh}
      The presheaf of spectra $\sX \mapsto \KH(\sX)$ satisfies cdh descent on the site of scalloped stacks.
    \end{corX}

    This result is an immediate consequence of the formula \eqref{eq:intro/KGL} and \corref{cor:intro/cohomology}\itemref{item:a07gsf01}.
    We also give a more direct argument, which only requires the unstable theory developed in \secref{sec:H}, by using the cdh descent criterion in \cite{KhanKblow} (see \remref{rem:KH cdh}).
  
  \subsection{Outline of Part I}
  \label{ssec:intro/pt1}

    The first part of the paper begins in \secref{sec:scallop} by introducing the class of \emph{scalloped} (derived) stacks.
    As a first approximation, scalloped stacks are those that are built, locally in some sense, out of quotient stacks of the form $[X/G]$ where $G$ is a nice embeddable group scheme over an affine scheme $S$ and $X$ is a \emph{quasi-affine} $S$-scheme with $G$-action.
    We recall (see \ssecref{ssec:nice}):
    \begin{itemize}
      \item
      Nice groups are a certain class of affine fppf group schemes which are linearly reductive (meaning that the functor of $G$-invariants is exact) which is stable under passage to closed subgroups and extensions.
      For example, algebraic tori are nice, as are finite étale groups of order invertible on $S$.
      
      \item 
      Embeddability means that $G$ is a closed subgroup of the general linear group of some vector bundle on $S$.
      For $G$ nice, this always holds Nisnevich-locally on $S$ (see \cite[Cor.~13.2]{AlperHallRydh}), and globally if $S$ is the spectrum of a field.
    \end{itemize}
    Such quotients will be called \emph{\qfund}.
    The class of scalloped stacks is then built as the closure of the class of \qfund stacks under the following property: given a quasi-compact algebraic stack $\sX$ with separated diagonal, and a representable étale neighbourhood $p : \sV \to \sX$ of a closed substack $\sZ\sub\sX$ with $\sV$ \qfund, we require that if $\sX\setminus\sZ$ is scalloped, then so is $\sX$.
    For example, this class includes quotients $[X/G]$ as above where now $X$ is a qcqs \emph{algebraic space}.
    It also includes tame Deligne--Mumford stacks, as well as tame Artin stacks in the sense of \cite{AbramovichOlssonVistoli}, with separated diagonal.
    In fact, the main result of \cite{AlperHallHalpernLeistnerRydh} implies that the class of scalloped stacks is precisely the class of qcqs stacks with separated diagonal and nice stabilizers.
    See Sects.~\ref{ssec:scallop} and \ref{ssec:scallop/tame}.
    In Appendix~\ref{sec:linearly} we also define \emph{linearly scalloped} stacks, which roughly speaking is a variant of the above definitions where arbitrary linearly reductive groups are allowed instead of only nice ones, and we explain how the theory developed in this paper can be extended to that setting.

    We conclude \secref{sec:scallop} by extending compact generation results for the derived category of quasi-coherent complexes on a qcqs scheme \cite{ThomasonTrobaugh,BondalVdB} to the class of scalloped derived stacks (see \thmref{thm:perfect}).
    This was previously known in the cases of tame quasi-Deligne--Mumford classical stacks (see \cite[Thm.~A]{HallRydhPerfect}) and quasi-compact derived stacks with affine diagonal and nice stabilizers (see \cite[Thm.~A.3.2]{BKRSMilnor}).\footnote{%
      It was also known by \cite[Thm.~0.3.4]{DrinfeldGaitsgoryFiniteness} that the derived $\infty$-category of quasi-coherent complexes on a scalloped stack is \emph{dualizable}, i.e., a \emph{retract} of a compactly generated stable \inftyCat (see \cite[Prop.~D.7.3.1]{LurieSAG}).
    }
    These proofs use the étale-local compact generation criterion of \cite[Thm.~C]{HallRydhPerfect}, which does not apply to general scalloped stacks.
    Our proof is in fact much more elementary and follows the same lines as the case of schemes or algebraic spaces (as in e.g. \cite[Thm.~3.1.1]{BondalVdB} and \cite[Thm.~9.6.1.1]{LurieSAG}), using étale neighbourhoods inductively to reduce to the case of \qfund stacks.

    In the rest of Part~I, which consists of Sects.~\ref{sec:H} and \ref{sec:SH}, we begin working towards the proof of \thmref{thm:intro/six} (which will be completed in Part~II) by giving the construction of the unstable and stable motivic homotopy categories over a scalloped derived stack.
    Recall that if $X$ is a qcqs algebraic space, the motivic homotopy category $\MotSpc(X)$ is the \inftyCat of $\A^1$-homotopy invariant Nisnevich sheaves on the site $\Sm_{/X}$ of smooth algebraic spaces of finite presentation over $X$.
    The stabilization $\SH(X)$ is defined by adjoining a $\otimes$-inverse of the Thom space $\Th_X(\sO_X) \simeq X\times (\P^1,\infty)$ of the trivial line bundle.
    See \cite{KhanSix}.

    Over a scalloped stack $\sX$, the unstable category $\MotSpc(\sX)$ is defined similarly as the \inftyCat of $\A^1$-invariant Nisnevich sheaves on $\Sm_{/\sX}$, where $\Sm_{/\sX}$ is the site of smooth representable morphisms of finite presentation over $\sX$ (see \ssecref{ssec:H/constr}).
    However, the correct definition of $\SH(\sX)$ is considerably more involved due to the fact that vector bundles on a stack need not be Nisnevich-locally trivial.
    When $\sX$ is the quotient of a qcqs algebraic space by a nice embeddable group, $\SH(\sX)$ is defined by adjoining $\otimes$-inverses of the Thom spaces of all vector bundles on $\sX$ (\remref{rem:invert quot}).
    For general $\sX$, the construction is less explicit, but is determined by the requirement that the assignment $\sX \mapsto \SH(\sX)$ satisfies Nisnevich descent.
    See \thmref{thm:SH}.

    For the quotient of a $G$-quasi-projective scheme $X$ by a nice group $G$, $\SH([X/G])$ recovers Hoyois's equivariant stable motivic homotopy category $\SH^G(X)$ (\remref{rem:Hoyois stable}).
    Thus in this case, our construction in particular removes the quasi-projectivity hypotheses in \cite{HoyoisEquivariant}.
    (Using the alternative theory in Appendix~\ref{sec:linearly} we can also allow $G$ to be linearly reductive, just as in \emph{op. cit.}, but in that case the quasi-projectivity hypotheses are necessary for us as well.)
    Due to the more involved construction of $\SH(\sX)$ for general $\sX$, the standard functoriality results such as the smooth base change formula (see \thmref{thm:funct SH}) require more work than in the quotient stack case.

  \subsection{Outline of Part II}

    In Part II we complete the proof of \thmref{thm:intro/six} by constructing the formalism of Grothendieck's six operations on the stable motivic homotopy category.

    We begin in \secref{sec:six} with an axiomatization of the system of categories $\SH(\sX)$ with its basic operations, via a structure called a $(*,\sharp,\otimes)$-formalism.
    We impose a basic set of axioms on such a structure, adapting Voevodsky's original conditions in the case of schemes \cite{VoevodskyCross}.
    The construction $\SH(-)$ is, by design, the universal one satisfying these axioms (see \propref{prop:univ}).

    In \secref{sec:prop} we prove the base change formula for proper representable morphisms (see \thmref{thm:ipqnosdf}).
    The general proof roughly follows \cite[\S 2.4]{CisinskiDegliseBook} in the case of schemes.
    A key new input necessary to reduce the case of a proper representable morphism $f : \sX \to \sY$ to the projective case is \thmref{thm:Chow}, which asserts that we can always find a covering $\widetilde{\sX} \to \sX$ in the proper cdh topology such that the composite $\widetilde{\sX} \to \sY$ becomes projective.
    This relies on a variant of Chow's lemma for stacks proven by Rydh.

    With proper base change in hand, we proceed in \secref{sec:shriek} to the construction of the $!$-operations.
    Following Deligne, the construction for \emph{compactifiable} morphisms is relatively straightforward.
    Using the $\infty$-categorical machinery of \cite{LiuZhengGluing}, we further extend this to morphisms that are only Nisnevich-locally compactifiable, such as arbitrary representable morphisms of finite type between scalloped stacks.
    We also prove purity for smooth representable morphisms between scalloped stacks with affine diagonal (\thmref{thm:purity}).

    Finally in \secref{sec:gys} we adapt the constructions of \cite{DegliseJinKhan} and \cite{KhanVirtual} to construct Euler and Gysin transformations in our setting.
    These will give rise in Part~III to Euler classes and Gysin maps in cohomology.

  \subsection{Outline of Part III}

    In Part~III we finally turn our attention towards generalized cohomology theories represented by motivic spectra.
    
    In \secref{sec:coh} we give the definitions and record the various operations and properties asserted in \corref{cor:intro/cohomology}.
    We also define relative Borel--Moore homology, for a representable morphism of finite type $f : \sX \to \sY$, via cohomology with coefficients in $f^!(\sF)$.
    When $f$ is smooth, this is related to cohomology via a Poincaré duality statement (\propref{prop:poincare}).

    \secref{sec:ex} contains the constructions of our main examples of motivic spectra mentioned in \ssecref{ssec:intro/coh}.
    For the construction of the algebraic K-theory spectrum $\KGL$ (\ssecref{ssec:KH}), we first use our compact generation result (\thmref{thm:perfect}) to show that algebraic K-theory satisfies Nisnevich descent on scalloped derived stacks (\thmref{thm:Thomason}).
    This already yields unstable representability of $\KH$ (see \constrref{constr:07g301}), which is enough to deduce Corollaries~\ref{cor:intro/KH} and \ref{cor:intro/KH cdh} (see Remarks~\ref{rem:0afsgdh013} and \ref{rem:KH cdh}).
    The stable representability (\thmref{thm:KGL}) is not much more difficult than in the case of schemes or algebraic spaces \cite{CisinskiKH,KhanKblow} and quotient stacks \cite{HoyoisKH}.

    The construction of the algebraic cobordism spectrum $\MGL$ is more subtle because the original definition of Voevodsky \cite[\S 6.3]{VoevodskyICM} in the case of schemes is reasonable for quotient stacks but not in general (see \remref{rem:apfsdin}).
    Instead, in \ssecref{ssec:coh/MGL} we give a different definition following \cite[\S 16]{BachmannHoyoisNorms}.

    The definition of the motivic cohomology spectrum is even less obvious.
    Indeed, even for schemes the construction is highly nontrivial as Voevodsky's theory of finite correspondences, for example, hinges on the intersection theory of relative cycles, which is delicate over general bases \cite{CisinskiDegliseBook}.
    Recently, Hoyois \cite{HoyoisFramedLoc} has given a definition of the Spitzweck motivic cohomology spectrum \cite{Spitzweck} that relies on the much more robust theory of \emph{framed} correspondences introduced in \cite{EHKSY}.
    In \ssecref{ssec:coh/Z}, we sketch an extension of this construction to stacks, although we do not undertake a full investigation of the theory of framed correspondences between stacks here.

    In \secref{sec:fix} we prove \thmref{thm:intro/conc}.
    In fact, we prove the statement more generally for $T$ any split torus over a connected noetherian affine base.
    Although our formulation (and proof) of this result is closest to \cite{ThomasonLefschetz}, such localization theorems are ubiquitous in the setting of equivariant cohomology: see also \cite{EdidinGrahamLocalization,Borel,SegalK}.

    \secref{sec:lim} deals with the lisse-extended variant of the stable motivic homotopy category, and explains how this recovers many previous constructions in the literature.
    We first show in \ssecref{ssec:Borel} that lisse-extended cohomology theories of quotient stacks can be computed by the Borel construction.\footnote{%
      Note added in final revision: A much more thorough version of this comparison is carried out in our subsequent paper \cite{equilisse}.
    }
    The comparison of \thmref{thm:intro/Bor} is then deduced in \ssecref{ssec:lim/equiv}.
    In \ssecref{ssec:lim/kan} we show that in many cases the lisse-extended stable motivic homotopy category can be computed by a Kan extension (\corref{cor:Chowdhury}).
    Finally in Subsect.~\ref{ssec:lim/mhtp} and \ref{ssec:lim/exh} we show that the lisse-extended motivic homotopy type recovers previous constructions of motives of stacks due to To\"en \cite{ToenMotive}, Choudhury \cite{Choudhury}, Richarz--Scholbach \cite{RicharzScholbach}, Hoskins--Pepin Lehalleur \cite{HoskinsPepinLehalleur}, and Choudhury--Deshmukh--Hogadi \cite{ChoudhuryDeshmukhHogadi}, whenever the latter are defined.

  \subsection{Related work}
  \label{ssec:related}

    Traditionally, equivariant cohomology theories in algebraic geometry are defined via algebraic versions of the Borel construction.
    For example, equivariant Chow groups \cite{EdidinGraham} and equivariant algebraic cobordism \cite{KrishnaCobordism,HellerMalagonLopez} are all defined this way.
    Kresch \cite{Kresch} has defined extensions of the Chow groups for algebraic stacks with affine stabilizers, which agree with the Edidin--Graham Chow groups for quotient stacks.
    We expect Kresch's Chow groups to agree with lisse-extended motivic Borel--Moore homology over a field.\footnote{%
      One can define a cycle class map from the former to the latter, which induces an isomorphism on quotient stacks (at least assuming resolution of singularities).
      The claim should then follow by stratifying the stack by quotient stacks and using the localization sequence.
      However, we have not checked that the (explicitly defined) boundary map in Kresch's localization sequence agrees with the one in motivic Borel--Moore homology.
      \emph{Note added in revision:} We have been informed that this will appear in forthcoming work of Younghan Bae and Hyeonjun Park.
    }
    
    The primary example of a non-Borel-type cohomology theory in algebraic geometry is algebraic K-theory, which is ``genuine'' by nature.
    Candidates for genuine (or ``Bredon-type'') equivariant Chow or motivic cohomology theories, for finite discrete group actions, have been given by Levine--Serpé \cite{LevineSerpe} and Heller--Voineagu--Østvær \cite{HellerVoineaguOstvaer}.
    In the forthcoming work \cite{HellerOstvaerSlice} it is claimed that the latter theory is equivalent to the zero slice of the equivariant motivic sphere spectrum, in the case of actions by the cyclic group of order two.
    On the other hand, while one would hope that the equivariant homotopy coniveau filtration of \cite{LevineSerpe} also computes the slice filtration on the equivariant motivic stable homotopy category (as in \cite{LevineConiveau}, non-equivariantly), this cannot be true because the Levine--Serpé theory fails to be homotopy invariant for nontrivial vector bundles (see \cite[Cor.~5.6]{LevineSerpe}).
    In particular, the Levine--Serpé theory cannot be representable by a genuine motivic spectrum.
    Some other Bredon-type cohomology theories for stacks have been constructed by Joshua \cite{Joshua}.

    A well-known heuristic is that in the presence of étale descent, there is no difference between genuine and Borel-type cohomology.
    For example, in étale cohomology, the ``genuine'' theory already satisfies étale descent, hence is Kan-extended from schemes, and is Borel-type (see \cite{LiuZheng,IllusieZheng}).
    Étale motives of stacks have also been compared with the Borel construction previously in \cite[Thm.~2.2.10]{RicharzScholbach} and \cite[App.~A]{HoskinsPepinLehalleur}.
    However, \thmref{thm:intro/Bor} (or rather \thmref{thm:Borel M}) shows that in general it is lisse extension, rather than étale sheafification, that computes the Borel construction.
    In the étale-local case the lisse extension already appeared in \cite{KhanVirtual} (cf. \examref{exam:0--81h}), where it was used to define Borel-type variants of \emph{rational} motivic cohomology, K-theory, and cobordism.

    The recent thesis of Chirantan Chowdhury \cite{Chowdhury} contains another construction of a stable motivic homotopy category over a certain class of algebraic stacks.
    This is a Borel-type construction and we show that it agrees with the lisse extension (see \ssecref{ssec:lim/kan}).

    There are several constructions of \emph{motives} of stacks in the literature \cite{ToenMotive,Choudhury,RicharzScholbach,HoskinsPepinLehalleur,ChoudhuryDeshmukhHogadi}.
    These are also all Borel-type and can be recovered as the image of the lisse-extended motivic stable homotopy type (\notatref{notat:as0fg1}) by the ``linearization'' functor
    \[
      \SH(k) \to \on{\mathbf{DM}}(k)
    \]
    to Voevodsky motives.
    See Remarks~\ref{rem:CDH comparison}, \ref{rem:ysbobql} and \ref{rem:UO1b2hov}.

    The two primary inspirations for our work are the genuine equivariant theories of cohomology and stable homotopy theory in algebraic topology (see e.g. \cite{HillHopkinsRavenel} or \cite[\S\S 2-4]{HillHopkinsRavenelSketch}), and the genuine equivariant motivic homotopy theory of quasi-projective schemes constructed by Hoyois \cite{HoyoisEquivariant}.
    In the topological setting, the distinction between Borel and genuine cohomology theories, and in particular the advantages of the latter, have long been well-understood: see for instance \cite{LewisMayMcClure} and \cite{CostenobleWaner} (as well as \cite[Thm.~II.2.7]{NikolausScholze}).
    Moreover, the Atiyah--Segal completion theorem \cite{AtiyahSegalCompletion} and related results like the Segal conjecture (proven by Carlsson \cite{Carlsson}) explain the precise manner in which the genuine theory is a refinement of the Borel one.
    
    In the algebraic setting, an analogue of the Atiyah--Segal theorem, describing Borel-equivariant K-theory as a completion of (genuine) equivariant K-theory, has been proven by Krishna \cite{KrishnaCompletion}.
    Conversely, in the case of actions with finite stabilizers, (genuine) equivariant K-theory can be described as the Borel-equivariant K-theory of the \emph{inertia} (a.k.a. ``twisted sectors''), up to rationalizing both sides and tensoring further with the maximal abelian extension $\bQ(\mu_\infty)$ (see \cite{BaumConnes,AtiyahSegalEuler,Vistoli,ToenGRR}).
    It would be interesting to know if these results have analogues in other genuine theories like motivic cohomology and algebraic cobordism.

    The six operations have been constructed in genuine equivariant motivic homotopy theory for quasi-projective $G$-schemes by Hoyois \cite{HoyoisEquivariant}.
    When $G$ is nice, our \thmref{thm:intro/six} extends his formalism in two orthogonal directions: we can take qcqs algebraic spaces with $G$-action (and arbitrary $G$-equivariant morphisms between them), and we also allow them to be derived.
    Note that some of our results, namely localization and Atiyah duality (Theorems~\ref{thm:loc unstable} and \ref{thm:ipqnosdf}\itemref{item:ipqnosdf/Atiyah}), are proven by dévissage arguments that eventually allow us to reduce to cases already considered by Hoyois.
    Further aspects of genuine equivariant motivic theory have been developed by Gepner--Heller \cite{GepnerHeller} and Bachmann \cite{BachmannMackey}.

    The six operational viewpoint on Euler classes and Gysin maps that we follow was introduced in \cite{DegliseJinKhan}, and extended to derived algebraic spaces in \cite{KhanVirtual}.
    In Appendix~A of \emph{loc. cit.}, the author also extended this construction to stacks by working in the lisse extension of the étale-local stable motivic homotopy category $\SH_{\et}$.
    Working in the lisse extension of $\SH$ itself, one can use the same approach to define virtual fundamental classes in generalized cohomology theories not satisfying étale descent (details will appear elsewhere).
    In \cite{LevineIntrinsic}, Levine has used Hoyois's equivariant motivic homotopy category to define virtual fundamental classes for quasi-projective $G$-schemes equipped with an equivariant perfect obstruction theory.
    When the latter arises from the cotangent complex of an equivariant quasi-smooth structure on the scheme, one can adapt the argument of \cite[\S 3.3]{KhanVirtual} to show that his construction agrees with \defnref{defn:fund smooth}.

  \subsection{Conventions}

    \begin{itemize}
      \item
      We freely use the language of \inftyCats as in \cite{LurieHTT}.
      
      \item
      We use the term \emph{\ani} (plural: \anis) as an abbreviation for animated sets in the sense of \cite[\S 5.1.4]{Anima}.\footnote{%
        This is a minor modification of the terminology introduced in \emph{loc. cit.}, where ``anima'' is both the singular and plural form.
        Our proposal of ``animum/anima'' is intended to match ``spectrum/spectra''.
      }
      The \inftyCat of \anis is equivalent to the \inftyCat of \inftyGrpds, and can be modelled by the homotopy theory of Kan complexes.

      \item
      We adopt the convention that a symmetric monoidal presentable \inftyCat is a presentable \inftyCat with a symmetric monoidal structure for which the monoidal product $\otimes$ commutes with colimits in both arguments.

      \item
      A colimit-preserving functor between presentable \inftyCats is called \emph{compact} if its right adjoint preserves filtered colimits (see e.g. \cite[Defn.~C.3.4.2]{LurieSAG}).

      \item
      A derived algebraic stack is a derived 1-Artin stack as in \cite[Chap.~2, 4.1]{GaitsgoryRozenblyum}.

      \item
      We often use the abbreviation ``qcqs'' for a (derived) algebraic space or stack to mean quasi-compact and quasi-separated.

      \item
      Given a derived algebraic stack $\sX$, we write $\K(\sX)$ for the K-theory \ani of $\sX$, defined by the Waldhausen $S_\bullet$-construction on the stable \inftyCat of perfect complexes on $\sX$.
      See e.g. \cite{Barwick} or \cite[\S 2.1]{KhanKstack}.
      By abuse of language, we will refer to points in $\K(\sX)$, rather than elements of $\pi_0\K(\sX)$, as ``K-theory classes''.
    \end{itemize}

  \subsection{Acknowledgments}

    We thank Tom Bachmann, Marc Hoyois, Marc Levine, and David Rydh for their interest in this work, for many helpful discussions, and for their valuable comments on previous drafts.
    We are especially grateful to David Rydh for explaining to us his stacky generalization of Chow's lemma, which is used in the proof of \thmref{thm:Chow}.
    We also thank Martin Gallauer for a comment on a previous version.

    The authors acknowledge partial support from SFB 1085 Higher Invariants, Universit\"at Regensburg.
    The first-named author was also supported by the Simons Collaboration on Homological Mirror Symmetry, Academia Sinica Grant AS-CDA-112-M01, and NSTC Grant 112-2628-M-001-0062030.


\section{Scalloped algebraic stacks}
\label{sec:scallop}

  \subsection{Nice group schemes}
  \label{ssec:nice}

    We recall the definition of nice and embeddable group schemes.
    We refer to \cite{AlperHallRydh} for more details.
    Fix an affine scheme $S$.

    \begin{defn}
      Let $G$ be an fppf affine group scheme over $S$.
      \begin{defnlist}
        \item
        We say $G$ is \emph{nice} if it is an extension of a finite étale group scheme, of order prime to the residue characteristics of $S$, by a group scheme of multiplicative type.

        \item
        We say $G$ is \emph{embeddable} if it is a closed subgroup of $\GL_S(\sE)$ for some locally free sheaf $\sE$ on $S$.
      \end{defnlist}
    \end{defn}

    \begin{exam}
      Any torus is nice and embeddable.
      If $S$ is the spectrum of a field $k$, then every finite étale group scheme over $S$ of order invertible in $k$ is nice and embeddable.
    \end{exam}

    \begin{rem}
      Any nice group scheme $G$ over $S$ is linearly reductive (see \cite[Rem.~2.2]{AlperHallRydh}), i.e., formation of derived global sections on $BG$ is t-exact.
    \end{rem}

  \subsection{The Nisnevich topology}

    For a qcqs derived algebraic stack $\sX$, denote by $\Asp_{/\sX}$ the \inftyCat of derived stacks $\sX'$ over $\sX$ for which the structural morphism $f : \sX' \to \sX$ is representable and of finite presentation.
    
    \begin{defn}\label{defn:Nis}
      A \emph{Nisnevich square} over $\sX$ is a cartesian square in $\Asp_{/\sX}$
      \[ \begin{tikzcd}
        \sW \ar{r}\ar{d}
        & \sV\ar{d}{p}
        \\
        \sU \ar{r}{j}
        & \sX'
      \end{tikzcd} \]
      where $j$ is an open immersion and $p$ is an étale morphism\footnote{%
        necessarily representable
      } which induces an isomorphism away from $\sU$.
      The \emph{Nisnevich topology} is the Grothendieck topology on $\Asp_{/\sX}$ generated by the following covering families:
      \begin{inlinelist}
        \item the empty family, covering the empty stack $\initial$;
        \item for every Nisnevich square as above, the family $\{j,p\}$ covering $\sX$.
      \end{inlinelist}
    \end{defn}

    \begin{prop}\label{prop:Nis}
      Let $\sF$ be a presheaf on the site $\Asp_{/\sX}$, with values in an \inftyCat $\sC$ that admits limits.
      Then $\sF$ satisfies \v{C}ech descent for the Nisnevich topology if and only if the following two conditions hold:
      \begin{thmlist}
        \item
        The object $\RGamma(\initial, \sF) \in \sC$ is terminal.

        \item
        For every Nisnevich square in $\Asp_{/\sX}$ as in \defnref{defn:Nis}, the induced square in $\sC$
        \[ \begin{tikzcd}
          \RGamma(\sX', \sF) \ar{r}{j^*}\ar{d}{p^*}
          & \RGamma(\sU, \sF) \ar{d}
          \\
          \RGamma(\sV, \sF) \ar{r}
          & \RGamma(\sW, \sF)
        \end{tikzcd} \]
        is cartesian.
      \end{thmlist}
    \end{prop}
    \begin{proof}
      Follows from \cite[Thm.~2.2.7]{KhanLocalization}.
    \end{proof}

    \begin{rem}
      It follows from \propref{prop:Nis} that the Nisnevich topology on $\Asp_{/\sX}$ (as defined in \defnref{defn:Nis}) coincides with the topology defined in \cite[\S 2C]{HoyoisKrishna}, and the one in \cite{KrishnaOstvaer} in the case of Deligne--Mumford stacks.
    \end{rem}

  \subsection{Scalloped stacks}
  \label{ssec:scallop}

    \begin{defn}\label{defn:scallop1}
      Let $\sX$ be a derived algebraic stack.
      A \emph{scallop decomposition} $(\sU_i, \sV_i, u_i)_i$ of $\sX$ is a finite filtration by quasi-compact open substacks
      \[
        \initial = \sU_0
        \hook \sU_1
        \hook \cdots
        \hook \sU_n = \sX,
      \]
      together with Nisnevich squares
      \[
        \begin{tikzcd}
          \sW_i \ar{r}\ar{d}
            & \sV_i \ar{d}{u_i}
          \\
          \sU_{i-1} \ar{r}
            & \sU_i
        \end{tikzcd}
      \]
      where $u_i$ are representable étale morphisms of finite presentation.
    \end{defn}

    \begin{rem}
      \defnref{defn:scallop1} is a variant of a notion introduced by Lurie (see \cite[Def.~2.5.3.1]{LurieSAG}).
      In the terminology of \emph{loc. cit.}, a scallop decomposition is as above but where the $\sV_i$ are required to be affine schemes.
      Thus if $\sX$ admits a scallop decomposition in Lurie's sense, then it must be a qcqs algebraic space (moreover, this turns out to be a sufficient condition, see \cite[Thm.~3.4.2.1]{LurieSAG}, \cite[I, Prop.~5.7.6]{RaynaudGruson}).
      For our purposes, namely in \defnref{defn:scallop} below, it is important to allow $\sV_i$ to be a stack.
    \end{rem}

    \begin{defn}\label{defn:scallop}
      Let $\sX$ be a quasi-compact quasi-separated derived algebraic stack.
      \begin{defnlist}
        \item
        We say that $\sX$ is \emph{nicely \fund} if it admits an \emph{affine} morphism $\sX \to BG$ for some nice embeddable group scheme $G$ over an affine scheme $S$.
        That is, $\sX$ is the quotient $[X/G]$ of an affine derived scheme $X$ over $S$ with $G$-action.

        \item
        We say that $\sX$ is \emph{nicely \qfund} if it admits a \emph{quasi-affine} morphism $\sX \to BG$ for some nice embeddable group scheme $G$ over an affine scheme $S$.
        That is, $\sX$ is the quotient $[X/G]$ of a quasi-affine derived scheme $X$ over $S$ with $G$-action.

        \item\label{item:70asdf}
        We say that $\sX$ is \emph{nicely scalloped} if it has separated diagonal and admits a scallop decomposition $(\sU_i, \sV_i, u_i)_i$ where $\sV_i$ are \qfund.
      \end{defnlist}
    \end{defn}

    \begin{rem}
      In Appendix~\ref{sec:linearly} we consider a notion of \emph{linearly scalloped} stack, which is roughly a variant of the above definition where nice groups are replaced by arbitrary linearly reductive groups.
      Throughout the main text, we will just drop the word ``nicely'' from \defnref{defn:scallop}.
    \end{rem}

    \begin{rem}
      Although we restrict ourselves to derived $1$-Artin stacks in this paper for simplicity\footnote{%
        Interested readers will find that many of our arguments, which proceed inductively to reduce to the \qfund case, would extend without modification.
      }, we note that the above definition admits a natural extension to the world of higher derived stacks.
      Namely, if \emph{$0$-scalloped} stacks are the class defined above, then a higher derived Artin stack $\sX$ is \emph{$n$-scalloped} if it has a scallop decomposition $(\sU_i,\sV_i,u_i)_i$ where $\sV_i$ are $(n-1)$-scalloped and $u_i$ are only required to be $n$-representable morphisms.
    \end{rem}

  \subsection{Properties of scalloped stacks}

    The next few results collect some of the main properties enjoyed by the class of scalloped stacks.

    \begin{thm}\label{thm:scallop}\leavevmode
      \begin{thmlist}
        \item\label{item:0a7dgfs}
        The classes of \fund, \qfund, and scalloped derived stacks are each stable under finite disjoint unions.

        \item\label{item:scallop/cover}
        Let $\sX$ be a qcqs derived algebraic stack.
        If $\sX$ has separated diagonal, then the following conditions are equivalent:
        \begin{thmlist}
          \item\label{item:scallop/cover/3}
          $\sX$ admits a Nisnevich cover $u : \sU \twoheadrightarrow \sX$ where $\sU$ is \fund.
          \item\label{item:scallop/cover/0}
          $\sX$ is scalloped.
          \item\label{item:scallop/cover/1}
          $\sX$ admits a scallop decomposition $(\sU_i, \sV_i, u_i)_i$ where the $\sV_i$ are scalloped.
          \item\label{item:scallop/cover/SNS}
          $\sX$ has nice stabilizers.
        \end{thmlist}

        \item\label{item:scallop/cover2}
        Let $\sX$ be a qcqs derived algebraic stack.
        If $\sX$ has affine diagonal, then the following conditions are equivalent:
        \begin{thmlist}
          \item\label{item:scallop/cover2/1}
          $\sX$ is scalloped.
          \item\label{item:scallop/cover2/3}
          $\sX$ admits an affine Nisnevich cover $u : \sU \twoheadrightarrow \sX$ where $\sU$ is \fund.
          \item\label{item:scallop/cover2/2}
          $\sX$ admits a scallop decomposition $(\sU_i, \sV_i, u_i)_i$ where the $\sV_i$ are \qfund and $u_i$ are quasi-affine.
        \end{thmlist}
      \end{thmlist}
    \end{thm}

    \begin{cor}\label{cor:scallop more}
      Let $f : \sX' \to \sX$ be a morphism of qcqs derived algebraic stacks.
      \begin{thmlist}
        \item
        If $\sX$ is \qfund and $f$ is quasi-affine, then $\sX'$ is \qfund.

        \item
        If $\sX$ is scalloped and $f$ is representable, then $\sX'$ is scalloped.
      \end{thmlist}
    \end{cor}
    \begin{proof}
      The first claim is obvious from the definition.
      The second follows immediately from the characterization in terms of nice stabilizers (\thmref{thm:scallop}), since the stabilizers of $\sX'$ will be subgroups of those of $\sX$, and the class of nice groups is stable under closed subgroups.
    \end{proof}

    The following can be regarded as a generalized version of Sumihiro's theorem \cite{Sumihiro}.

    \begin{thm}\label{thm:sumihiro}
      Let $\sX = [X/G]$ be the quotient of a qcqs derived algebraic space $X$ with $G$-action, where $G$ is a nice group scheme over an affine scheme $S$.
      Then we have:
      \begin{thmlist}
        \item\label{item:sumihiro/cover}
        $\sX$ admits a Nisnevich cover $u : \sU \twoheadrightarrow \sX$ where $\sU$ is of the form $[U/G]$ with $U$ an affine derived scheme over $S$ with $G$-action.
        Moreover, if $X$ has affine diagonal, then $u$ is affine.
        
        \item\label{item:sumihiro/scallop}
        $\sX$ admits a scallop decomposition $(\sU_i, \sV_i, u_i)_i$ where the $\sV_i$ are of the form $[V_i/G]$, for some quasi-affine derived schemes $V_i$ over $S$ with $G$-action; in particular, $\sX$ is scalloped.
      \end{thmlist}
    \end{thm}

    Theorems~\ref{thm:scallop} and \ref{thm:sumihiro} will be proven in \ref{ssec:scallop/proof} below.

  \subsection{Tame stacks}
  \label{ssec:scallop/tame}

    Let us recall a few classes of examples of scalloped stacks:

    \begin{exam}[Tame DM stacks]
      If $\sX$ is a tame derived Deligne--Mumford stack with separated diagonal, then its stabilizers are linearly reductive finite étale group schemes (in particular, they are nice).
      Hence $\sX$ is scalloped by \thmref{thm:scallop}\itemref{item:scallop/cover}.
    \end{exam}

    \begin{exam}[Tame Artin stacks]
      If $\sX$ is an algebraic stack with separated diagonal which is tame in the sense of \cite{AbramovichOlssonVistoli}, then it has nice stabilizers.
      Hence it is scalloped by \thmref{thm:scallop}\itemref{item:scallop/cover}.
    \end{exam}

  \subsection{Proof of \thmref{thm:scallop} and \thmref{thm:sumihiro}}
  \label{ssec:scallop/proof}

    The proof of \thmref{thm:scallop} will make use of the following result of Alper--Hall--Rydh.

    \begin{thm}\label{thm:locally affine}
      Let $f : \sX \to \sY$ be a representable morphism of derived algebraic stacks such that $\sX$ is \fund and $\sY$ has affine diagonal.
      Then there exists a Zariski cover $u : \sU \twoheadrightarrow \sX$ such that $\sU$ is \fund and the composite $f_0 = f \circ u: \sU \to \sY$ is affine.
    \end{thm}
    \begin{proof}
      Applying \cite[Prop.~12.15]{AlperHallRydh}, we get around every closed point $x\in\abs{\sX}$ a Zariski open neighbourhood $\sU_x\sub\sX$ such that $\sU_x$ is \fund and the restriction $f|_{\sU_x} : \sU_x \to \sY$ is affine.
      Since $\sX$ is quasi-compact, the induced Zariski cover $\coprod_x \sU_x \twoheadrightarrow \sX$ admits a finite subcover, which we take to be $u : \sU \twoheadrightarrow \sX$.
    \end{proof}

    \begin{proof}[Proof of \thmref{thm:sumihiro}]
      Since $\sX$ has nice stabilizers, \thmref{thm:scallop}\itemref{item:scallop/cover} implies that there exists a Nisnevich cover $u : \sU \twoheadrightarrow \sX$ with $\sU$ \fund.
      By \thmref{thm:locally affine}, we can refine the cover $u$ such that $\sU$ is \emph{affine} over $BG$, i.e., of the form $[U/G]$ where $U$ is an affine derived scheme over $S$ with $G$-action.
      If $X$ has affine diagonal, then $U \to X$ is automatically affine, hence so is $u : \sU \to \sX$.
      This shows \itemref{item:sumihiro/cover}, and \itemref{item:sumihiro/scallop} now follows from \cite[Prop.~2.9]{HoyoisKrishna}.
    \end{proof}

    \begin{proof}[Proof of \thmref{thm:scallop}]\leavevmode

      \emph{Claim~\itemref{item:0a7dgfs}:}
      Follows immediately from the definitions.

      \emph{Claim~\itemref{item:scallop/cover}, \itemref{item:scallop/cover/3} $\Rightarrow$ \itemref{item:scallop/cover/0}:}
      Follows from \cite[Prop.~2.9]{HoyoisKrishna} and the fact that open substacks of \fund stacks are \qfund.

      \emph{Claim~\itemref{item:scallop/cover}, \itemref{item:scallop/cover/0} $\Rightarrow$ \itemref{item:scallop/cover/1}:}
      Obvious.

      \emph{Claim~\itemref{item:scallop/cover}, \itemref{item:scallop/cover/1} $\Rightarrow$ \itemref{item:scallop/cover/SNS}:}
      Follows from the fact that fundamental stacks have nice stabilizers (as subgroups of nice groups are nice).

      \emph{Claim~\itemref{item:scallop/cover}, \itemref{item:scallop/cover/SNS} $\Rightarrow$ \itemref{item:scallop/cover/3}:}
      This is \cite[Thm.~1.9]{AlperHallHalpernLeistnerRydh} (see also \cite[Thm.~A.1.8]{BKRSMilnor}).

      \emph{Claim~\itemref{item:scallop/cover2}, \itemref{item:scallop/cover2/1} $\Rightarrow$ \itemref{item:scallop/cover2/3}:}
      By claim~\itemref{item:scallop/cover} we get a Nisnevich cover $u : \sU \twoheadrightarrow \sX$ with $\sU$ \fund.
      Since $\sX$ has affine diagonal, we may apply \thmref{thm:locally affine} to $u$ to get a Nisnevich cover $\sU' \twoheadrightarrow \sU$ with the composite $\sU' \twoheadrightarrow \sX$ affine and $\sU'$ \fund.

      \emph{Claim~\itemref{item:scallop/cover2}, \itemref{item:scallop/cover2/3} $\Rightarrow$ \itemref{item:scallop/cover2/2}:}
      Apply \cite[Prop.~2.9]{HoyoisKrishna}.

      \emph{Claim~\itemref{item:scallop/cover2}, \itemref{item:scallop/cover2/2} $\Rightarrow$ \itemref{item:scallop/cover2/1}:}
      Obvious.
    \end{proof}

  \subsection{The resolution property}
  \label{ssec:resprop}

    For a derived algebraic stack $\sX$, we write $\Dqc(\sX)$ for the stable \inftyCat of quasi-coherent complexes on $\sX$ (see e.g. \cite[Chap.~3, 1.1.4]{GaitsgoryRozenblyum}).
    When $\sX$ is classical, this is equivalent to the derived \inftyCat of $\sO_\sX$-modules with quasi-coherent cohomology (see \cite[Prop.~1.3]{HallRydhPerfect}).
    We write $\Qcoh(\sX) \simeq \Dqc(\sX)^\heartsuit$ for the full subcategory of discrete quasi-coherent sheaves, i.e., the heart of the t-structure.

    \begin{defn}
      Let $\sX$ be a derived algebraic stack.
      \begin{defnlist}
        \item
        We say that $\sX$ has the \emph{resolution property} if, for every quasi-coherent sheaf $\sF \in \Qcoh(\sX)$, there exists a small collection $\{\sE_\alpha\}_\alpha$ of finite locally free sheaves and a surjective (on $\piz$) morphism
        \[ \bigoplus_\alpha \sE_\alpha \twoheadrightarrow \sF \]
        in $\Dqc(\sX)$.
        
        \item
        We say that $\sX$ has the \emph{derived resolution property} if the above holds moreover for every connective quasi-coherent \emph{complex} $\sF \in \Dqc(\sX)_{\ge 0}$.
      \end{defnlist}
    \end{defn}

    \begin{exam}\label{exam:0g11up0}
      Let $G$ be an embeddable linearly reductive group scheme over an affine scheme $S$.
      Then the classifying stack $BG$ has the derived resolution property.
      In fact, a resolving set is given by finite locally free $G$-modules on $S$.
      See \cite[Ex.~1.35]{KhanKstack}.
    \end{exam}

    \begin{prop}\label{prop:2rgoy10}
      Let $\sX$ be a \qfund derived stack.
      Then $\sX$ satisfies the derived resolution property.
    \end{prop}

    \propref{prop:2rgoy10} follows from \examref{exam:0g11up0} and the following lemmas.

    \begin{lem}\label{lem:intriguess}
      Let $f : X' \to X$ be a quasi-affine morphism of derived algebraic stacks.
      Then for every connective quasi-coherent complex $\sF \in \Dqc(X')_{\ge 0}$, the canonical morphism
      \[
        f^* (f_*(\sF)_{\ge 0})
        \to f^* f_*(\sF)
        \xrightarrow{\mrm{counit}} \sF
      \]
      is surjective on $\pi_0$.
    \end{lem}
    \begin{proof}
      First note that if $f$ is affine, then the induced map on $\pi_0$ is the surjection
      \[ \pi_0 f^* (f_*(\sF)_{\ge 0}) \simeq f^*_\heartsuit f^\heartsuit_*(\pi_0 \sF) \to \pi_0 \sF \]
      since $\pi_0 f^* = f^*_\heartsuit \pi_0$ on connective objects, where $\heartsuit$ denotes the corresponding functors on the hearts of the t-structures (i.e. on abelian categories of quasi-coherent sheaves).

      In general, write $f = g \circ j$ where $j : X' \to Y$ is a quasi-compact open immersion and $g : Y \to X$ is affine.
      Since $g_*$ is t-exact, it commutes with connective covers and thus the morphism in question factors as
      \[
        j^*g^*(g_*j_*(\sF)_{\ge 0})
        \simeq j^*g^*g_*(j_*(\sF)_{\ge 0})
        \xrightarrow{\mrm{counit}} j^*(j_*(\sF)_{\ge 0})
        \to j^*j_*(\sF)
        \xrightarrow{\mrm{counit}} \sF.
      \]
      The first arrow is surjective on $\pi_0$ (since $g$ is affine), the second arrow is bijective on $\pi_0$ (since $j^*$ is t-exact), and the last arrow is invertible (by the base change formula).
    \end{proof}

    \begin{lem}\label{lem:8phi112}
      Let $f : \sX \to \sY$ be a quasi-affine morphism of derived algebraic stacks.
      If $\sY$ has the derived resolution property, then so does $\sX$.
      In fact, if $\{\sE_\alpha\}_\alpha$ is a resolving set for $\sY$, then $\{f^*(\sE_\alpha)\}_\alpha$ is a resolving set for $\sX$.
    \end{lem}
    \begin{proof}
      For any connective complex $\sF \in \Dqc(\sX)_{\ge 0}$, the derived resolution property for $\sY$ implies that there exists a surjection $\phi : \bigoplus_\alpha \sE_\alpha \twoheadrightarrow f_*(\sF)_{\ge 0}$, where $(-)_{\ge 0}$ denotes the connective cover.
      This induces a surjection
      \[
        f^*(\phi) : \bigoplus_\alpha f^*(\sE_\alpha)
        \twoheadrightarrow f^*(f_*(\sF)_{\ge 0})
      \]
      Now consider the composite
      \[
        \bigoplus_\alpha f^*(\sE_\alpha)
        \twoheadrightarrow f^*(f_*(\sF)_{\ge0})
        \twoheadrightarrow \sF,
      \]
      where the second arrow is the surjection of \lemref{lem:intriguess}.
    \end{proof}

  \subsection{Compact generation of the derived category}

    Given a derived algebraic stack $\sX$, let $\Dqc(\sX)$ be the stable \inftyCat of quasi-coherent complexes on $\sX$ as in \ssecref{ssec:resprop}.

    \begin{defn}
      A derived algebraic stack $\sX$ is \emph{perfect} if $\Dqc(\sX)$ is compactly generated by the full subcategory $\Dperf(\sX)$ of perfect complexes.
    \end{defn}

    In this subsection we prove:

    \begin{thm}\label{thm:perfect}
      Every scalloped derived stack is perfect.
    \end{thm}

    \begin{rem}
      In case $\sX$ has affine diagonal, \thmref{thm:perfect} can be proven using the criterion of \cite[Thm.~C]{HallRydhPerfect}.
      See \cite[Thm.~A.3.2]{BKRSMilnor}.
    \end{rem}

    \begin{lem}\label{lem:spdufg0p1}
      Let $\sX$ be a \qfund derived stack.
      Then for every cocompact closed subset $Z \sub \abs{\sX}$, the stable \inftyCat $\Dqc(\sX~\mrm{on}~Z)$, of quasi-coherent complexes supported on $Z$, is compactly generated by the full subcategory $\Dperf(\sX~\mrm{on}~Z)$ of perfect complexes supported on $Z$.
    \end{lem}
    \begin{proof}
      By \cite[Cor.~1.44, Prop.~1.47]{KhanKstack}, it suffices to check that $\sX$ is universally of finite cohomological dimension (as in \cite[Def.~1.18]{KhanKstack}) and has the derived resolution property.
      The former holds because $\sX$ is representable over $BG$, for some linearly reductive $G$ (see \cite[Ex.~1.23, 1.24]{KhanKstack}), and the latter holds by \propref{prop:2rgoy10}.
    \end{proof}

    \begin{lem}\label{lem:pa0sdhg1}
      Let $\sX$ be a quasi-separated derived algebraic stack.
      Then every compact object $\sF \in \Dqc(\sX)$ is perfect.
      Conversely, if $\sX$ is scalloped, then every perfect complex on $\sX$ is compact.
    \end{lem}
    \begin{proof}
      For the first claim, let $\sF \in \Dqc(\sX)$ be a compact object.
      It is enough to show that $u^*(\sF) \in \Dqc(X)$ is compact for every smooth morphism $u : X \to \sX$ with $X$ affine.
      Since $\sX$ is quasi-separated, $u$ is qcqs and representable, so $u^*$ is a compact functor (see e.g. \cite[Thm.~1.20, Ex.~1.23]{KhanKstack}).
      In particular, $u^*(\sF)$ is compact and hence also perfect, since $X$ is perfect.

      Now suppose $\sX$ is scalloped.
      Let us first note that, given any Nisnevich square
      \[ \begin{tikzcd}
        \sW \ar{r}\ar{d}
        & \sV \ar{d}{p}
        \\
        \sU \ar{r}{j}
        & \sX,
      \end{tikzcd} \]
      we have that $\sF \in \Dqc(\sX)$ is compact if and only if $j^*(\sF) \in \Dqc(\sU)$ and $p^*(\sF) \in \Dqc(\sV)$ are compact.
      The condition is necessary since $j^*$ and $p^*$ are compact functors, and sufficiency follows immediately from the fact that the square of stable \inftyCats
      \[ \begin{tikzcd}
        \Dqc(\sX) \ar{r}{j^*}\ar{d}{p^*}
        & \Dqc(\sU) \ar{d}
        \\
        \Dqc(\sV) \ar{r}
        & \Dqc(\sW)
      \end{tikzcd} \]
      is cartesian (see e.g. \cite[Thm.~2.2.3]{BKRSMilnor}), in view of the definition of compact objects, the fact that filtered colimits of spectra are exact (i.e., commute with finite limits), and that formation of mapping spectra in a stable \inftyCat commutes with limits.

      Now by induction on the length of a scallop decomposition of $\sX$, we may assume that $\sX$ is \qfund.
      In that case, the claim follows from \lemref{lem:spdufg0p1}.
    \end{proof}

    \begin{proof}[Proof of \thmref{thm:perfect}]
      Let $\sX$ be a scalloped derived stack.
      By induction, it will suffice to show the following: suppose given a Nisnevich square
      \[ \begin{tikzcd}
        \sW \ar{r}{j'}\ar{d}{p'}
        & \sV \ar{d}{p}
        \\
        \sU \ar{r}{j}
        & \sX,
      \end{tikzcd} \]
      where $j$ is a quasi-compact open immersion and $p$ is a representable étale morphism of finite presentation inducing an isomorphism away from $\sU$.
      We claim that if $\sU$ is perfect and $\sV$ is \qfund, then $\sX$ is perfect.

      Let $\{\sF_\alpha\}_\alpha$ be a small set of compact perfect complexes on $\sU$ which generate $\Dqc(\sU)$.
      By \lemref{lem:spdufg0p1} applied to $\sV$ and $Z = \abs{\sV}\setminus\abs{\sW} \sub \abs{\sV}$, there also exists a small set of compact perfect complexes $\{\sG_\beta\}_\beta$ on $\sV$ which are supported on $Z$ and which generate $\Dqc(\sV~\mrm{on}~Z)$.
      By \lemref{lem:spdufg0p1} and the Thomason--Neeman localization theorem (see e.g. \cite[Thm.~A.3.11]{HermitianII}), the restriction functor
      \[ \Dperf(\sV) \to \Dperf(\sW) \]
      is a Karoubi projection in the sense of \cite[Def.~A.3.5]{HermitianII}.
      Therefore, by replacing every $\sF_\alpha$ by $\sF_\alpha\oplus\sF_\alpha[1]$ if necessary, we can assume that each $\sF_\alpha|_{\sW}$ lifts to a perfect complex $\sF_{\sV,\alpha} \in \Dperf(\sV)$.
      Using the cartesian square of stable \inftyCats (see e.g. \cite[Thm.~2.2.3]{BKRSMilnor})
      \begin{equation}\label{eq:0pg0pu1}
        \begin{tikzcd}
          \Dperf(\sX) \ar{r}{j^*}\ar{d}{p^*}
          & \Dperf(\sU) \ar{d}
          \\
          \Dperf(\sV) \ar{r}
          & \Dperf(\sW),
        \end{tikzcd}
      \end{equation}
      we construct objects $\sF_{\sX,\alpha} \in \Dperf(\sX)$ by gluing
      \[ \sF_{\alpha} \in \Dperf(\sU), \quad \sF_{\sV,\alpha} \in \Dperf(\sV) \]
      along some choice of isomorphisms $\sF_{\sV,\alpha}|_\sW \simeq \sF_{\alpha}|_\sW$.
      Similarly, we construct objects $\sG_{\sX,\beta} \in \Dperf(\sX)$ by gluing
      \[ 0 \in \Dperf(\sU), \quad \sG_\beta \in \Dperf(\sV) \]
      along the (unique) isomorphism $\sG_\beta|_\sW \simeq 0$.
      By \lemref{lem:pa0sdhg1}, $\sF_{\sX,\alpha}$ and $\sG_{\sX,\beta}$ are all compact objects of $\Dqc(\sX)$.

      It remains to show that the union of the sets $\{\sF_{\sX,\alpha}\}_\alpha$ and $\{\sG_{\sX,\beta}\}_\beta$ generate $\Dqc(\sX)$.
      That is, for every object $\sR \in \Dqc(\sX)$ which is right orthogonal to the $\sF_{\sX,\alpha}$ and $\sG_{\sX,\beta}$'s, we have $\sR \simeq 0$, or equivalently that $\sR|_{\sU} \simeq 0$ and $\sR|_{\sV} \simeq 0$.

      We begin with the following preliminary observation:
      \begin{enumerate}
        \item[$(\ast)$]
        If $\sR \in \Dqc(\sX)$ is right orthogonal to $\sG_{\sX,\beta}$ for every $\beta$, then its restriction $\sR|_{\sV} \in \Dqc(\sV)$ belongs to the essential image of $j'_* : \Dqc(\sW) \to \Dqc(\sV)$.
      \end{enumerate}
      Indeed, it follows from the cartesian square \eqref{eq:0pg0pu1} and the fact that $\sG_{\sX,\beta}|_\sU \simeq 0$ that there is an isomorphism of mapping spectra
      \[
        \Maps_{\Dqc(\sV)}(\sG_\beta, \sR|_\sV)
        \simeq \Maps_{\Dqc(\sX)}(\sG_{\sX,\beta}, \sR)
        \simeq 0.
      \]
      Therefore, $\sR|_\sV$ is right orthogonal to $\Dqc(\sV~\mrm{on}~Z)$.
      Since the right orthogonal subcategory to the latter is precisely the essential image of $j'_* : \Dqc(\sW) \to \Dqc(\sV)$, $(\ast)$ follows.

      Let us now show that $\sR|_\sU \simeq 0$.
      We use the cartesian square \eqref{eq:0pg0pu1} again to compute $\Maps(\sF_{\sX,\alpha}, \sR)$.
      Since $\sR|_\sV \simeq j'_*(\sR|_\sW)$ by claim~$(\ast)$, it follows that the restriction map
      \begin{multline*}
        \Maps_{\Dqc(\sV)}(\sF_{\sV,\beta}, \sR|_\sV)\\
        \simeq \Maps_{\Dqc(\sV)}(\sF_{\sV,\beta}, j'_*(\sR|_\sW))
        \xrightarrow{j'^*} \Maps_{\Dqc(\sW)}(\sF_{\beta}|_\sW, \sR|_\sW)
      \end{multline*}
      is invertible (by adjunction).
      Thus we get for every $\alpha$ an isomorphism of mapping spectra
      \[
        \Maps_{\Dqc(\sU)}(\sF_\alpha, \sR|_\sU)
        \simeq \Maps_{\Dqc(\sX)}(\sF_{\sX,\alpha}, \sR)
        \simeq 0.
      \]
      Since $\{\sF_\alpha\}_\alpha$ generate $\Dqc(\sU)$, it follows that $\sR|_\sU \simeq 0$.

      Finally, we deduce $\sR|_\sV \simeq j'_*(\sR|_\sW) \simeq j'_*(\sR|_{\sU}|_\sW) \simeq 0$ by claim~$(\ast)$.
      This concludes the proof that $\sR \simeq 0$.
    \end{proof}

\section{The unstable homotopy category}
\label{sec:H}

  \subsection{Construction}
  \label{ssec:H/constr}

    Let $\sX$ be a scalloped derived stack.
    Write $\Sm_{/\sX}$ for the full subcategory of $\Asp_{/\sX}$ spanned by $\sX' \in \Asp_{/\sX}$ for which $f : \sX' \to \sX$ is \emph{smooth} representable of finite presentation.

    \begin{rem}
      If $\sX$ is \emph{classical}, then every derived stack $\sX' \in \Sm_{/\sX}$ is also automatically classical.
      The reader only interested in classical stacks may therefore ignore the word ``derived'' throughout the text, for the most part.
    \end{rem}

    \begin{defn}\label{defn:H}
      Let $\sX$ be a scalloped derived stack.
      A \emph{$\Sm$-fibred \ani} over $\sX$ is a presheaf of \anis on $\Sm_{/\sX}$.
      A \emph{motivic $\Sm$-fibred \ani} over $\sX$ is a $\Sm$-fibred \ani satisfying Nisnevich descent and homotopy invariance.
      The latter condition means that for every $\sX' \in \Sm_{/\sX}$ and every vector bundle $\pi : \sV \to \sX'$, the map of \anis
      \[ \pi^* : \RGamma(\sX', \sF) \to \RGamma(\sV, \sF) \]
      is invertible.
      We write $\MotSpc(\sX)$ for the \inftyCat of motivic $\Sm$-fibred \anis over $\sX$.
    \end{defn}

    \begin{lem}
      Let $\sF$ be a $\Sm$-fibred \ani over $\sX$.
      If the canonical map
      \[ \RGamma(\sX', \sF) \to \RGamma(\sX' \times \A^1, \sF) \]
      is invertible for all $\sX' \in \Sm_{/\sX}$, then $\sF$ satisfies homotopy invariance.
      In particular, $\sF$ is motivic.
    \end{lem}
    \begin{proof}
      Recall that there is a canonical strict $\A^1$-homotopy contracting any vector bundle $\sV \to \sX$ to the zero section, hence $\Lhtp(\h_\sX(\sV)) \simeq \pt$.
      Alternatively, see the argument in the proof of \cite[Thm.~5.2]{KrishnaRavi}.
    \end{proof}

    \begin{rem}\label{rem:loc}
      The \inftyCat $\MotSpc(\sX)$ is an accessible left Bousfield localization of the \inftyCat of $\Sm$-fibred \anis and in particular is presentable.
      The localization functor can be computed as the transfinite composite
      \begin{equation}\label{eq:description of Lmot}
        \L(\sF) \simeq \colim_{n \ge 0} (\Lhtp \circ \LNis)^{\circ n}(\sF),
      \end{equation}
      where $\LNis$ and $\Lhtp$ are the Nisnevich and $\A^1$-localization functors, respectively.
      Recall that $\Lhtp$ can be computed by the formula
      \[
        \RGamma(\sX', \Lhtp(\sF)) = \colim_{[n]\in\bDelta^\op} \RGamma(\sX' \times \A^n, \sF)
      \]
      for every $\sX' \in \Sm_{/\sX}$, where $\A^\bullet$ is the cosimplicial affine scheme defined e.g. as in \cite[p.~45]{MorelVoevodsky}.
    \end{rem}

    \begin{exam}\label{exam:h(X)}
      Any $\sX' \in \Sm_{/\sX}$ represents a motivic $\Sm$-fibred \ani
      $$\L \h_\sX(\sX') \in \MotSpc(\sX)$$
      where $\sX' \mapsto \h_\sX(\sX')$ is the Yoneda embedding.
      When there is no risk of confusion, we will sometimes write simply $\sX'$ instead of $\L \h_\sX(\sX')$.
      These objects are compact, since the conditions of Nisnevich descent and homotopy invariance are stable under filtered colimits.
    \end{exam}

  \subsection{Generation}

    \begin{prop}\label{prop:gen}
      Let $\sX$ be a scalloped derived stack.
      \begin{thmlist}
        \item
        The \inftyCat $\MotSpc(\sX)$ is generated under sifted colimits by objects of the form $\L \h_\sX(\sX')$, where $\sX' \in \Sm_{/\sX}$ is \qfund.

        \item\label{item:0h103n}
        Suppose $\sX = [X/G]$, where $G$ is a nice group scheme over an affine scheme $S$ and $X$ is a quasi-compact derived algebraic space over $S$ with $G$-action.
        Then $\MotSpc(\sX)$ is generated under sifted colimits by objects of the form $\L \h_\sX([U/G])$, where $U$ is a \emph{quasi-affine} derived $G$-scheme, smooth over $X$.

        \item\label{item:0as7df}
        For every $\sX'\in\Sm_{/\sX}$, the object $\L \h_\sX(\sX') \in \MotSpc(\sX)$ is compact.
      \end{thmlist}
    \end{prop}

    \begin{rem}\label{rem:Hoyois}
      Let $G$ be a nice group scheme over an affine scheme $S$ and let $X$ be a $G$-quasi-projective scheme over $S$.
      Then it follows from \propref{prop:gen}\itemref{item:0h103n} that there is a canonical equivalence
      \[ \MotSpc([X/G]) \simeq \MotSpc^G(X), \]
      where the right-hand side is the $G$-equivariant motivic homotopy category of \cite{HoyoisEquivariant}.
      Note that homotopy invariance for affine bundles is automatic in our setting by \cite[Rem.~3.13]{HoyoisEquivariant}.
    \end{rem}

    \begin{proof}[Proof of \propref{prop:gen}]
      Let $\sC$ be the full subcategory of $\MotSpc(\sX)$ generated under sifted colimits by objects of the form $\L \h_\sX(\sX')$, where $\sX' \in \Sm_{/\sX}$ is \qfund.

      Suppose we are given a Nisnevich square in $\Sm_{\sX}$
      \[ \begin{tikzcd}
        \sW \ar{r}\ar{d}
        & \sV\ar{d}{p}
        \\
        \sU \ar{r}{j}
        & \sX'
      \end{tikzcd} \]
      where $j$ is an open immersion and $p$ is an étale morphism inducing an isomorphism away from $\sU$.
      Then by definition, the canonical map
      \[
        \h_\sX(\sU) \fibcoprod_{\h_\sX(\sW)} \h_\sX(\sV)
        \to \h_\sX(\sX')
      \]
      is a Nisnevich-local equivalence.
      Thus if $\L \h_\sX(\sU) \in \sC$ and $\sV$ (and hence $\sW$) is \qfund, then also $\L \h_\sX(\sX') \in \sC$.
      By induction and \thmref{thm:scallop}\itemref{item:scallop/cover}, it follows that for every $\sX' \in \Sm_{\sX}$, we have $\L \h_\sX(\sX') \in \sC$.
      Finally, it now follows from \cite[Lem.~5.5.8.14]{LurieHTT} that $\sC = \MotSpc(\sX)$.

      The second statement follows similarly by \thmref{thm:sumihiro}.
      For the last one, note that $\L$ preserves compact objects because the full subcategory of motivic $\Sm$-fibred \anis is closed under filtered colimits.
    \end{proof}

  \subsection{Functoriality}

    We record the basic functorialities.

    \begin{prop}\label{prop:f^*}
      Let $f : \sX \to \sY$ be a morphism of scalloped derived stacks.
      Then there exists a canonical functor
      \begin{align*}
        f^* : \MotSpc(\sY) \to \MotSpc(\sX)
      \end{align*}
      satisfying the following properties.
      \begin{thmlist}
        \item
        $f^*$ commutes with colimits, hence in particular admits a right adjoint
        \[ f_* : \MotSpc(\sX) \to \MotSpc(\sY). \]

        \item
        For any $\sV \in \Sm_{/\sY}$ and $n \ge 0$, there is a canonical isomorphism
        \[
          f^*\big(\L\h_\sY(\sV)\big)
          \simeq \L \h_\sX(\sV \fibprod_\sY \sX).
        \]

        \item
        $f^*$ is compact, i.e., its right adjoint $f_*$ preserves filtered colimits.

        \item
        $f^*$ is symmetric monoidal.
      \end{thmlist}
    \end{prop}
    \begin{proof}
      Compare \cite[Prop.~1.22]{KhanSix}.
    \end{proof}

    \begin{prop}\label{prop:f_sharp}
      Let $f : \sX \to \sY$ be a \emph{smooth} representable morphism of scalloped derived stacks.
      Then the inverse image functor $f^*$ admits a left adjoint
      \begin{align*}
        f_\sharp : \MotSpc(\sX) \to \MotSpc(\sY).
      \end{align*}
      This functor is characterized uniquely by commutativity with colimits and the formula
      \begin{align*}
        f_\sharp \big(\L\h_\sX(\sU)\big) &\simeq \L\h_\sY(\sU)
      \end{align*}
      for any $\sU \in \Sm_{/\sX}$.
      Moreover, it is $\MotSpc(\sY)$-linear; in particular, we have the projection formula
      \[ f_\sharp(\sF) \times \sG \simeq f_\sharp (\sF \times f^*(\sG)) \]
      for every $\sF \in \MotSpc(\sX)$ and $\sG \in \MotSpc(\sY)$.
    \end{prop}
    \begin{proof}
      Compare \cite[Prop.~1.23]{KhanSix}.
    \end{proof}

    \begin{prop}[Smooth base change]\label{prop:smooth bc}
      Suppose given a cartesian square of scalloped derived stacks
      \[ \begin{tikzcd}
        \sX' \ar{r}{q}\ar{d}{f}
        & \sY' \ar{d}{g}
        \\
        \sX \ar{r}{p}
        & \sY
      \end{tikzcd} \]
      where $p$ and $q$ are smooth representable.
      Then there are canonical isomorphisms
      \begin{align*}
        q_\sharp f^* &\to g^* p_\sharp,\\
        p^* g_* &\to f_* q^*.
      \end{align*}
    \end{prop}
    \begin{proof}
      Compare \cite[Prop.~1.26]{KhanSix}.
    \end{proof}

    \begin{prop}\label{prop:Nis sep}
      Let $(u_\alpha : \sX_\alpha \to \sX)_\alpha$ be a Nisnevich covering family of a scalloped derived stack $\sX$.
      Then the family of functors
      \[ u_\alpha^* : \MotSpc(\sX) \to \MotSpc(\sX_\alpha) \]
      is jointly conservative as $\alpha$ varies.
    \end{prop}
    \begin{proof}
      Compare e.g. \cite[Prop.~2.5.7]{KhanLocalization}.
    \end{proof}

  \subsection{Exactness of \texorpdfstring{$i_*$}{i\_*}}

    Recall from \propref{prop:f^*} that the direct image functor commutes with filtered colimits.
    For closed immersions, it commutes with almost all colimits:

    \begin{prop}\label{prop:i_* colimits}
      Let $i : \sX \to \sY$ be a closed immersion of scalloped derived stacks.
      Then the direct image functor
      \begin{equation*}
        i_* : \MotSpc(\sX) \to \MotSpc(\sY)
      \end{equation*}
      commutes with contractible\footnote{
        An \inftyCat is contractible if the \inftyGrpd/\ani obtained by formally adjoining inverses to all its morphisms is contractible.
      } colimits.
    \end{prop}

    \propref{prop:i_* colimits} follows from the following result of Alper--Hall--Halpern-Leistner--Rydh:

    \begin{thm}\label{thm:lifting closed}
      Let $i : \sZ \hookrightarrow \sX$ be a closed immersion of scalloped derived stacks.
      Let $f_0 : \sZ' \to \sZ$ be a smooth (resp. étale) representable morphism.
      Then, up to passing to a Nisnevich cover of $\sZ'$, there exists a smooth (resp. étale) representable morphism $f : \sX' \to \sX$ with $\sX'$ \fund and a cartesian square
      \[ \begin{tikzcd}
        \sZ' \ar{r}\ar{d}{f_0}
        & \sX' \ar{d}{f}
        \\
        \sZ \ar{r}{i}
        & \sX.
      \end{tikzcd} \]
      Moreover, if $\sX$ has affine diagonal then $f$ can be taken to be affine.
    \end{thm}
    \begin{proof}
      Note that the claim is Nisnevich-local both on $\sZ'$ and $\sX$.
      Therefore, we may assume $\sZ'$ and $\sX$ (and hence $\sZ$) are \fund (\thmref{thm:scallop}\itemref{item:scallop/cover}).
      In that case, the claim follows by combining Theorems~1.13 and 1.5 in \cite{AlperHallHalpernLeistnerRydh}.
    \end{proof}

    \begin{proof}[Proof of \propref{prop:i_* colimits}]
      This follows from \propref{prop:gen} and \thmref{thm:lifting closed}, exactly as in the proof of \cite[Thm.~3.1.1]{KhanLocalization}.
    \end{proof}

  \subsection{Derived invariance}

    \begin{thm}\label{thm:derinv}
      Let $\sX$ be a scalloped derived stack and let $i : \sX_\cl \to \sX$ denote the inclusion of the classical truncation.
      Then the functor $i^* : \MotSpc(\sX) \to \MotSpc(\sX_\cl)$ is an equivalence.
    \end{thm}

    \begin{constr}\label{constr:U,t}
      Let $\sX' \in \Sm_{/\sX}$ with structural morphism $p : \sX' \to \sX$.
      Given a section $t : \sX_\cl \to \sX'_\cl$ of the classical truncation $p_\cl : \sX'_\cl \to \sX_\cl$, we define the $\Asp$-fibred \ani over $\sX$ of ``$t$-trivialized maps'' to $\sX'$,
      \[
        \h_\sX(\sX', t).
      \]
      This is the presheaf on $\Asp_{/\sX}$ whose sections over $\sX'' \in \Asp_{/\sX}$ are pairs $(f, \alpha)$ with $f : \sX'' \to \sX'$ an $\sX$-morphism and $\alpha$ a ``$t$-trivialization'' of $f$, i.e., a commutative triangle
      \[ \begin{tikzcd}
        \sX'' \fibprod_\sX \sX_{\cl} \ar{rr}{f\fibprod_\sX \sX_\cl}\ar{rd}
        & & \sX' \fibprod_\sX \sX_{\cl} \simeq \sX'_\cl,\\
        & \sX_\cl \ar[hookrightarrow]{ru}{t}
      \end{tikzcd} \]
      where the isomorphism $\sX' \fibprod_\sX \sX_{\cl} \simeq \sX'_\cl$ is by flatness of $\sX'$ over $\sX$.
      More precisely, $\RGamma(\sX'', \h_\sX(\sX',t))$ is the homotopy fibre of the canonical map
      \[ \Maps_\sX(\sX'', \sX') \to \Maps_{\sX_\cl}(\sX''\fibprod_\sX \sX_\cl, \sX'_\cl) \]
      at the point defined by the morphism $\sX''\fibprod_\sX \sX_\cl \to \sX_\cl \stackrel{t}{\hook} \sX'_\cl$.
      See \cite[Constr.~4.1.3]{KhanLocalization} or \cite[\S 1]{HoyoisFramedLoc}.
    \end{constr}

    \begin{prop}\label{prop:contract}
      Let $(\sX', t)$ be as in \constrref{constr:U,t}.
      Then, Nisnevich-locally on $\sX'$, the $\Asp$-fibred \ani $\h_\sX(\sX', t)$ is motivically contractible, i.e.,
      \[ \L \h_\sX(\sX', t) \simeq \pt \]
      in $\MotSpc(\sX)$.
    \end{prop}

    The proof of \propref{prop:contract} will require a few geometric preliminaries.

    \begin{lem}\label{lem:Lifting-sec-of-sm.mor}
      Let $p : \sY \to \sX$ be a smooth representable morphism of derived stacks where $\sX$ is \fund.
      Suppose given a commutative triangle as on the left-hand side below:
      \begin{equation*}
        \begin{tikzcd}
          \sZ_0 \ar{r}{i_0}\ar[swap]{rd}{q_0}
            & \sY_\cl \ar{d}{p_\cl}
          \\
            & \sX_\cl
        \end{tikzcd}
        \qquad
        \begin{tikzcd}
          \sZ \ar[dashed]{r}{i}\ar[swap,dashed]{rd}{q}
            & \sY \ar{d}{p}
          \\
            & \sX,
        \end{tikzcd}
      \end{equation*}
      where $i_0$ is a closed immersion and $q_0$ is smooth (resp. étale).
      Then there exists, Nisnevich-locally on $\sY$, a quasi-smooth closed immersion $i : \sZ \to \sY$ lifting $i_0$ such that the composite $q : \sZ \to \sY \to \sX$ is a smooth (resp. étale) morphism lifting $q_0$.
    \end{lem}

    \begin{proof}
      Let $\sI$ denote the ideal of definition of $i_0: \sZ_0 \hookrightarrow \sY_\cl$.
      Since $q_0$ and $p_\cl$ are smooth, $i_0$ is quasi-smooth and hence has finite locally free conormal sheaf $\sN_{i_0}$.
      By \cite[Erratum, Thm.~B]{HoyoisEquivariant}, there exists an affine Nisnevich cover $\sY'_0 \twoheadrightarrow \sY_\cl$ and a finite locally free $\sE_0$ on $\sY'_0$ lifting $\sN_{i_0}$.
      By derived invariance of the étale site and \cite[Lem.~A.2.6]{BKRSMilnor}, we can lift this data to an affine Nisnevich cover $\sY' \twoheadrightarrow \sY$ and a finite locally free $\sE$ on $\sY'$.
      As the statement is Nisnevich-local on $\sY$, we may replace $\sY$ by $\sY'$ and thereby assume that there exists a finite locally free $\sE$ on $\sY$ such that $i_0^*k^* (\sE) \simeq \sN_{i_0}$, where $k : \sY_\cl \to \sY$ denotes the inclusion of the classical truncation.

      Consider the unit morphism $\phi : \sE \to k_* i_{0,*} i_0^* k^* (\sE) \simeq k_* i_{0,*}(\sN_{i_0})$.
      Since $\sE$ is projective on $\sY$ by \cite[Prop.~A.3.4]{BKRSMilnor}, $\phi$ lifts along the $\piz$-surjection $\sI \twoheadrightarrow \sI/\sI^2 \simeq k_* i_{0,*}(\sN_{i_0})$ to a morphism
      \[ \psi : \sE \to \sI \sub \sO_\sY. \]
      Up to passing to a Zariski cover of $\sY_\cl$, the Nakayama lemma implies that $\sZ_0$ is the zero locus of $\piz(\psi) : \sE_0 \to \sO_{\sY_\cl}$.
      In other words, the square
      \[ \begin{tikzcd}
        \sZ_0 \ar{r}{i_0}\ar{d}{i_0}
        & \sY_\cl \ar{d}{\psi_\cl}
        \\
        \sY_\cl \ar{r}{0}
        & \bV_{\sY_\cl}(\sE_\cl)
      \end{tikzcd} \]
      is (classically) cartesian.
      We define $\sZ$ as the derived zero locus of $\psi$, so that there is a homotopy cartesian square
      \[ \begin{tikzcd}
        \sZ \ar{r}{i}\ar{d}{i}
        & \sY \ar{d}{\psi}
        \\
        \sY \ar{r}{0}
        & \bV_{\sY}(\sE).
      \end{tikzcd} \]
      In particular, $\sZ_\cl \simeq \sZ_0$ and $i$ lifts $i_0$.
      It remains only to show that $q := p \circ i : \sZ \to \sX$ is smooth.

      Recall that smoothness can be checked after derived base change to the classical truncation of the target (e.g. \cite[Cor.~11.2.2.8]{LurieSAG}).
      Thus it will suffice to show that the following squares are homotopy cartesian:
      \[ \begin{tikzcd}
        \sZ_0 \ar{r}{i_0}\ar{d}
        & \sY_\cl \ar{r}{p_\cl}\ar{d}{k}
        & \sX_\cl \ar{d}
        \\
        \sZ \ar{r}{i}
        & \sY \ar{r}{p}
        & \sX.
      \end{tikzcd} \]
      The right-hand square is homotopy cartesian because $p$ is smooth, and the left-hand square is classically cartesian by construction.
      We claim the latter square is in fact homotopy cartesian.
      For this it is enough to show that the morphism $\sZ_0 \to \widetilde{\sZ} := \sZ \fibprodR_\sY \sY_\cl$ is étale.
      But in the transitivity triangle
      $$
      \sL_{\widetilde{\sZ}/\sY_\cl}|_{\sZ_0} \to \sL_{\sZ_0/\sY_\cl} \to \sL_{\sZ_0/\widetilde{\sZ}},
      $$
      the first map is identified with the canonical isomorphism $\sL_{\widetilde{\sZ}/\sY_\cl}|_{\sZ_0} \simeq \sE|_{\sZ_0}[1] \simeq \sN_{i_0}[1] \simeq \sL_{\sZ_0/\sY_\cl}$, so the relative cotangent complex $\sL_{\sZ_0/\widetilde{\sZ}}$ vanishes.
    \end{proof}

    \begin{lem}\label{lem:approx-of-sec}
      Let $p : \sY \to \sX$ be a smooth affine morphism of \fund derived stacks.
      Then for any section $s : \sX \to \sY$, there exists a morphism $f : \sY \to N_{\sX/\sY}$, étale on the image of $s$, and a homotopy cartesian square
      \begin{equation*}
        \begin{tikzcd}
          \sX \ar{r}{s}\ar[equal]{d}
            & \sY \ar{d}{f}
          \\
          \sX \ar{r}{0}
            & N_{\sX/\sY}.
        \end{tikzcd}
      \end{equation*}
      Here $N_{\sX/\sY} = \bV_{\sX}(\sN_{\sX/\sY})$ is the normal bundle, total space of the conormal sheaf $\sN_{\sX/\sY} = \sL_{\sX/\sY}[-1]$.
    \end{lem}
    \begin{proof}
      Since $s$ is a section of $p$, we have $\sN_{\sX/\sY} \simeq \sL_{\sY/\sX}|\sX$.
      Since $p$ is smooth, the latter is finite locally free and we have
      \[ \piz(\sN_{\sX/\sY}) \simeq \sN_{\sX/\sY}|_{\sX_\cl} \simeq \sN_{\sX_\cl/\sY_\cl} \simeq \sI/\sI^2, \]
      where $\sI$ is the ideal sheaf of the closed immersion $s_\cl: \sX_\cl \hookrightarrow \sY_\cl$.
      In particular, there is a canonical $\sO_{\sX_\cl}$-module surjection $p_{\cl,*}(\sI) \twoheadrightarrow \piz(\sN_{\sX/\sY})$.
      Since $\piz(\sN_{\sX/\sY})$ is projective (see \cite[Lem.~2.17]{HoyoisEquivariant} or \cite[Prop.~A.3.4]{BKRSMilnor}), this admits a splitting and the resulting morphism $\phi_0 : \piz(\sN_{\sX/\sY}) \to p_{\cl,*}(\sI) \sub p_{\cl,*}(\sO_{\sY_\cl})$ lifts from $\sX_\cl$ to $\sX$ by \cite[Lem.~A.2.6]{BKRSMilnor}:
      \[ \phi : \sN_{\sX/\sY} \to p_*(\sO_{\sY}). \]
      This determines a morphism $f: \sY \to N_{\sX/\sY}$ (recall that $p$ is affine) which fits in a commutative square
      \begin{equation*}
        \begin{tikzcd}
          \sX \ar{r}{s}\ar[equal]{d}
            & \sY \ar{d}{f}
          \\
          \sX \ar{r}{0}
            & N_{\sX/\sY}.
        \end{tikzcd}
      \end{equation*}
      which is cartesian on classical truncations by construction.
      To show that the square is homotopy cartesian, it will suffice to show that the induced morphism $\sX \to \widehat{\sX} := \sX \times_{\sN_{\sX/\sY}} \sY$ is étale.
      For this consider the exact triangle 
      $$
      \sL_{\widehat{\sX}/\sY}|\sX \to \sL_{\sX/\sY} \to \sL_{\sX/\widehat{\sX}}
      $$
      where the first arrow is identified with the canonical isomorphism
      $$\sL_{\widehat{\sX}/\sY}|\sX \simeq \sL_{\sX/N_{\sX/\sY}} \simeq \sN_{\sX/\sY}[1] \simeq \sL_{\sX/\sY}.$$
      Finally, the fact that $f$ is étale on the image of $s$ follows from the isomorphism $s^*\sL_{f} \simeq \sL_{\sX/\sX} \simeq 0$, which comes from the homotopy cartesianness of the above square.
    \end{proof}

    The preparations for the proof of \propref{prop:contract} are complete.

    \begin{proof}[Proof of \propref{prop:contract}]
      By \lemref{lem:Lifting-sec-of-sm.mor}, we may pass to a Nisnevich cover of $\sX'$ (if necessary) to assume that $t$ lifts to a section $s : \sX \to \sX'$ of $p : \sX' \to \sX$.
      By \lemref{lem:approx-of-sec} there exists a homotopy cartesian square
      \[ \begin{tikzcd}
        \sX \ar{r}{s}\ar[equals]{d}
        & \sX' \ar{d}{f}
        \\
        \sX \ar{r}{0}
        & N_{\sX/\sX'}
      \end{tikzcd} \]
      where $f$ is étale on a Zariski neighbourhood $\sX'_0 \sub \sX'$ of $s$.
      This gives rise to canonical isomorphisms of motivic $\Asp$-fibred \anis
      \[ \L \h_\sX(\sX', s) \simeq \L \h_\sX(\sX'_0, s) \simeq \L \h_\sX(N_{\sX/\sX'}, 0) \]
      by \cite[Lem.~4.2.6]{KhanLocalization} or \cite[Lem.~2]{HoyoisFramedLoc} (whose proofs extend to our setting without modification).
      Now note that, for any vector bundle $\sV \to \sX$, the $\Asp$-fibred \ani $\h_\sX(\sV, 0)$ is contracted to the zero section $0 : \sX \to \sV$ via the scaling action of $\A^1$ on $\sV$ (see \cite[Lem.~4.2.5]{KhanLocalization}).
    \end{proof}

    \begin{proof}[Proof of \thmref{thm:derinv}]
      By \thmref{thm:scallop}\itemref{item:scallop/cover} and \propref{prop:Nis sep}, we are immediately reduced to the case where $\sX$ is \fund.
      Since every $\sX'_0 \in \Sm_{/\sX_\cl}$ lifts (up to Nisnevich-localizing on $\sX'_0$) to some $\sX' \in \Sm_{/\sX}$ (\thmref{thm:lifting closed}), \propref{prop:gen} implies that $i_*$ generates its codomain $\MotSpc(\sX_\cl)$ under colimits.
      It will thus suffice to show that the unit morphism
      \[ \id \to i_* i^* \]
      is invertible, i.e., that $i^* : \MotSpc(\sX) \to \MotSpc(\sX_\cl)$ is fully faithful.
      
      We will show that for every $\sX' \in \Sm_{/\sX}$, the canonical map
      \begin{equation}\label{eq:Fefhaho}
        \h_\sX(\sX') \to i_* i^* \h_\sX(\sX') \simeq i_* \h_{\sX_\cl}(\sX'_\cl)
      \end{equation}
      is a motivic equivalence on $\Sm$-fibred anima (where $\sX'_\cl \simeq \sX' \fibprod_\sX \sX_\cl$ since $\sX'$ is flat over $\sX$).
      In \eqref{eq:Fefhaho} and the following, all operations (such as $i^*$) are at the level of ``unlocalized'' fibred anima.
      Since $i_*$ commutes with sifted colimits (\propref{prop:i_* colimits}) and motivic $\Sm$-fibred anima on $\sX$ are generated under sifted colimits by motivic localizations of $\h_\sX(\sX')$ (\propref{prop:gen}), the claim will follow.
      In fact, we claim that \eqref{eq:Fefhaho} is a motivic equivalence on all $\Asp$-fibred anima.

      By universality of colimits, it is enough to show that for every $\sX''\in\Asp_{/\sX}$ and every morphism $\h_\sX(\sX'') \to i_*\h_{\sX_\cl}(\sX'_\cl)$, corresponding to a morphism $t : \sX'' \fibprod_{\sX} \sX_\cl \to \sX'_\cl$ over $\sX_\cl$, the base change
      \begin{equation}\label{eq:Wazrerfod}
        \h_\sX(\sX') \fibprod_{i_*\h_{\sX_\cl}(\sX'_\cl)} \h_\sX(\sX'') \to \h_{\sX}(\sX'')
      \end{equation}
      is invertible.
      On $\Asp$-fibred anima, we can write $\h_\sX(\sX'') \simeq p_\sharp \h_{\sX''}(\sX'')$ where $p : \sX'' \to \sX$ is the projection.
      We moreover have the projection formula\footnote{%
        The reader is warned that the analogous formula does not hold on $\Sm$-fibred anima (even when $p$ is smooth), see \remref{rem:wrongproj}, which is why we work with $\Asp$-fibred anima in this proof.
        We thank the referee for pointing out this mistake in a previous revision.
      }
      \begin{equation}
        \sF \fibprod_{\sG} p_\sharp \h_{\sX''}(\sX'')
        \simeq p_\sharp (p^* \sF \fibprod_{p^* \sG} \h_{\sX''}(\sX'')),
      \end{equation}
      for any morphism $\sF \to \sG$ over $\sX$, and the base change formula $p^*i_* \simeq i''_* p_\cl^*$ where $i'' : \sX'' \fibprod_\sX \sX_\cl \to \sX''$ is the base change of $i$.
      It follows that \eqref{eq:Wazrerfod} is identified with the image by $p_\sharp$ of
      \[
        \h_{\sX''}(\sX'\fibprod_\sX\sX'', t')
        := \h_{\sX''}(\sX' \fibprod_\sX \sX'') \fibprod_{i''_*\h_{\sX''_\cl}(\sX'_\cl \fibprod_{\sX_\cl} \sX''_\cl)} \h_{\sX''}(\sX'')
        \to \h_{\sX''}(\sX''),
      \]
      where $t' : \sX'' \fibprod_\sX \sX_\cl \to (\sX'\fibprod_\sX \sX'')_\cl \simeq \sX'_\cl \fibprod_{\sX_\cl} \sX'_\cl$ is the section induced by $t : \sX'' \fibprod_{\sX} \sX_\cl \to \sX'_\cl$.
      Since $\h_{\sX''}(\sX'\fibprod_\sX\sX'', t')$ is motivically contractible by \propref{prop:contract}, the claim follows.
    \end{proof}

  \subsection{Localization}

    \begin{thm}\label{thm:loc unstable}
      Suppose given a complementary closed-open pair
      \[ \begin{tikzcd}
        \sZ \ar[hookrightarrow]{r}{i}
        & \sX \ar[hookleftarrow]{r}{j}
        & \sU
      \end{tikzcd} \]
      of scalloped derived stacks.
      Then the commutative square of endofunctors of $\MotSpc(\sX)$
      \[ \begin{tikzcd}
        j_\sharp j^* \ar{r}{\mrm{counit}}\ar{d}{\mrm{unit}}
        & \id \ar{d}{\mrm{unit}}
        \\
        j_\sharp j^* i_* i^* \ar{r}{\mrm{counit}}
        & i_* i^*,
      \end{tikzcd} \]
      where $j_\sharp j^* i_* i^*$ is contractible by the smooth base change formula (\propref{prop:smooth bc}), is homotopy cocartesian.
    \end{thm}
    \begin{proof}
      By \thmref{thm:derinv} we may assume that $\sX$ and $\sZ$ are classical.
      By \thmref{thm:scallop}\itemref{item:scallop/cover}, \propref{prop:Nis sep} and \propref{prop:smooth bc}, we may assume that $\sX$ is \fund.
      Now the statement is \cite[Thm.~4.18]{HoyoisEquivariant}, modulo the equivalence mentioned in \remref{rem:Hoyois}.
    \end{proof}

    \begin{rem}\label{rem:wrongproj}
      As pointed out to us by the referee, the proof of the localization theorem in \cite[Thm.~4.18]{HoyoisEquivariant} contains a subtle error.
      Namely, it makes use of the projection formula
      \[
        \sF \fibprod_\sG p_\sharp(\sF')
        \simeq p_\sharp (p^*\sF \fibprod_{p^*\sG} \sF')
      \]
      for a smooth morphism $p$, asserted in \cite[Prop.~4.3]{HoyoisEquivariant}, which in fact does not hold on $\Sm$-fibred \anis.\footnote{
        The same mistake is present in the proof of localization in \cite[Thm.~2.21]{MorelVoevodsky} (for classical schemes) and \cite[Thm.~B]{KhanLocalization} (for derived schemes).
        The incorrect projection formula is used implicitly in the proof of \cite[Thm.~2.21]{MorelVoevodsky}, and stated explicitly in \cite[Prop.~2.5.13]{KhanLocalization}.
        These can be fixed in the same way, by noting that the motivic contractibility statement in \cite[Thm.~4.1.6]{KhanLocalization} holds at the level of $\mrm{Sch}$-fibred anima.
      }
      However, it does hold on $\mrm{Sch}$-fibred or $\Asp$-fibred \anis, and one can moreover show that the motivic contractibility of $\h_S(X,t)$ (Assertion~$(\ast)$ in the proof of \cite[Thm.~4.18]{HoyoisEquivariant}) holds at the level of $\Asp$-fibred \anis; indeed, one easily observes that the proof of Assertion~$(\ast)$ in \cite[Thm.~4.18]{HoyoisEquivariant} demonstrates this stronger statement.
      This refinement of $(\ast)$ is enough to correct the gap in \cite[Thm.~4.18]{HoyoisEquivariant}.
      Thanks to the referee's suggestion, we have carried out the analogous fix in our proof here of \thmref{thm:derinv} (which is the special case of localization for the inclusion of the classical truncation), see \propref{prop:contract} for the refined contractibility statement in our context.
      See also \cite[\S 1]{HoyoisFramedLoc}, where the same fix is carried out in the setting of classical schemes.
    \end{rem}

\section{The stable homotopy category}
\label{sec:SH}

  \subsection{Thom \anis}

    The stabilization $\SH(\sX)$ of $\MotSpc(\sX)$ is defined, at least Nisnevich-locally on $\sX$, by adjoining $\otimes$-inverses of certain pointed objects in $\MotSpc(\sX)$.

    \begin{defn}[Thom \anis]\label{exam:Thom}
      Let $\sX$ be a scalloped derived stack.
      For any finite locally free sheaf $\sE$ on $\sX$, write $\sV = \bV_\sX(\sE)$ for its total space and $\sV \setminus \sX$ for the complement of the zero section.
      The \emph{Thom \ani} of $\sE$ is the pointed motivic \ani
      \[ \Th_\sX(\sE) := \L \h_\sX(\sV)/\L \h_\sX(\sV \setminus \sX), \]
      i.e., the cofibre of the inclusion $\sV \setminus \sX \hook \sV$ taken in the \inftyCat $\MotSpc(\sX)$.
      Note that this is compact by \propref{prop:gen}\itemref{item:0as7df}.
    \end{defn}

    \begin{exam}
      The Thom \ani of the free sheaf of rank one is
      $$
        \Th_\sX(\sO_\sX)
        \simeq \Sigma_{S^1} \L \h_\sX(\sX \times \A^1\setminus\{0\})
        \simeq \L \h_\sX(\P^1 \times \sX)
        \in \MotSpc(\sX)_\bullet,
      $$
      where $\Sigma_{S^1}$ denotes topological suspension, $\A^1\setminus\{0\}$ is pointed at $1$, and $\P^1$ is pointed at $\infty$.
    \end{exam}

    \begin{rem}\label{rem:Thf^*}
      For any morphism $f : \sX' \to \sX$ of scalloped derived stacks and any finite locally free $\sE$ on $\sX$, we have a canonical isomorphism $f^*(\Th_\sX(\sE)) \simeq \Th_{\sX'}(f^*\sE)$.
    \end{rem}

    \begin{rem}\label{rem:Thom exact}
      Given an exact triangle $\sE' \to \sE \to \sE''$ of finite locally free sheaves on $\sX$, there is a canonical isomorphism
      \[ \Th_\sX(\sE) \simeq \Th_\sX(\sE') \wedge \Th_\sX(\sE'') \]
      in $\MotSpc(\sX)_\bullet$.
      This follows by the arguments of \cite[2.4.10]{CisinskiDegliseBook} or \cite[\S 3.5]{HoyoisEquivariant}.
    \end{rem}

  \subsection{Characterization}

    The stable motivic homotopy category $\SH(\sX)$, as a functor in $\sX$, will be characterized uniquely as follows.

    \begin{thm}\label{thm:SH}
      For every scalloped derived stack $\sX$, there exists a symmetric monoidal colimit-preserving functor
      \[ \Sigma^\infty : \MotSpc(\sX)_\bullet \to \SH(\sX) \]
      satisfying the following properties:
      \begin{thmlist}
        \item\label{item:SH/Thom}
        For every scalloped derived stack $\sX$ and every finite locally free sheaf $\sE$ on $\sX$, the object $\Sigma^\infty \Th_\sX(\sE) \in \SH(\sX)$ is $\otimes$-invertible.
        Moreover, the assignment $\sE \mapsto \Sigma^\infty \Th_\sX(\sE)$ induces a canonical map of $\Einfty$-groups
        \[ \Th_\sX : \K(\sX) \to \Pic(\SH(\sX)), \]
        from the algebraic K-theory \ani of perfect complexes on $\sX$ to the Picard \ani of $\otimes$-invertible objects in $\SH(\sX)$.

        \item\label{item:SH/f^*}
        For every morphism $f : \sX \to \sY$ of scalloped derived stacks, there is a symmetric monoidal colimit-preserving functor
        \[ f^* : \SH(\sY) \to \SH(\sX) \]
        which commutes with $\Sigma^\infty$ and $\Th$.
        
        \item\label{item:SH/Nis}
        The assignments
        \[ \sX \mapsto \SH(\sX), \quad f \mapsto f^* \]
        of \itemref{item:SH/f^*} determine a presheaf $\SH^*$ with values in the \inftyCat of symmetric monoidal presentable \inftyCats and symmetric monoidal colimit-preserving functors, which satisfies Nisnevich descent.

        \item\label{item:SH/tensor}
        For every \emph{representable} morphism $f : \sX \to \sY$ of scalloped derived stacks, there is a canonical isomorphism of $\SH(\sY)$-modules
        \[ \MotSpc(\sX)_\bullet \otimes_{\MotSpc(\sY)_\bullet} \SH(\sY) \to \SH(\sX). \]

        \item\label{item:SH/BG}
        If $\sX = BG$ is the classifying stack of an embeddable nice group scheme $G$ over an affine scheme $S$, then
        \[ \Sigma^\infty : \MotSpc(\sX)_\bullet \to \SH(\sX) \]
        is the \emph{universal} symmetric monoidal colimit-preserving functor which $\otimes$-inverts the Thom \anis $\Th_\sX(\sE)$ of all finite locally free sheaves $\sE$ on $\sX$ (i.e., finite representations of $G$ over $S$).
      \end{thmlist}
    \end{thm}

    \begin{rem}\label{rem:-31g0g0p1}
      The proof of \thmref{thm:SH} will in fact construct morphisms of presheaves
      \begin{align*}
        &\Sigma^\infty : \MotSpc_\bullet^* \to \SH^*,\\
        &\Th : \K \to \Pic(\SH^*).
      \end{align*}
      Moreover, $\Sigma^\infty : \MotSpc^*_\bullet \to \SH^*$ will be the unique morphism of presheaves out of $\MotSpc_\bullet^*$, on the site of scalloped derived stacks, satisfying properties \itemref{item:SH/Nis}, \itemref{item:SH/tensor} and \itemref{item:SH/BG}.
    \end{rem}

    \begin{notat}
      We will write $\Omega^\infty : \SH(\sX) \to \MotSpc(\sX)_\bullet$ for the right adjoint of $\Sigma^\infty$.
      For a morphism $f : \sX \to \sY$, we will write $f_* : \SH(\sX) \to \SH(\sY)$ for the right adjoint of $f^*$.
    \end{notat}

    \begin{rem}\label{rem:invert quot}
      Let $G$ be an embeddable nice group scheme over an affine scheme $S$, and $X$ a qcqs derived algebraic space over $S$ with $G$-action.
      Then $f : [X/G] \to BG$ is a representable morphism of scalloped derived stacks (\thmref{thm:sumihiro}).
      Combining claims \itemref{item:SH/BG} and \itemref{item:SH/tensor} of \thmref{thm:SH}, we find that $\SH([X/G])$ is obtained from $\MotSpc([X/G])_\bullet$ by formally adjoining $\otimes$-inverses of the objects
      \[ f^*\Th_{BG}(\sE) \simeq \Th_{[X/G]}(f^*(\sE)) \]
      where $\sE$ ranges over finite locally free sheaves on $BG$.
    \end{rem}

    \begin{rem}\label{rem:Hoyois stable}
      In the situation of \remref{rem:invert quot}, it follows from Remarks~\ref{rem:Hoyois} and \ref{rem:invert quot} that there is a canonical equivalence of symmetric monoidal stable \inftyCats
      \[ \SH([X/G]) \simeq \SH^G(X) \]
      where the right-hand side is Hoyois's $G$-equivariant stable motivic homotopy category (see \cite[\S 6]{HoyoisEquivariant}), when the latter is defined (i.e., when $X$ is a $G$-quasi-projective scheme).
      For general $X$, the left-hand side can be taken as the definition of $\SH^G(X)$.
    \end{rem}

  \subsection{Functoriality}

    Before proceeding to the construction of $\SH^*$, let us also record its functorial properties:

    \begin{thm}\label{thm:funct SH}
      Let $f : \sX \to \sY$ be a morphism of scalloped derived stacks.
      \begin{thmlist}
        \item\label{item:funct SH/smooth}
        If $f$ is smooth representable, then $f^* : \SH(\sY) \to \SH(\sX)$ satisfies the following properties:
        \begin{thmlist}
          \item\label{item:funct SH/smooth/adj}
          \emph{Left adjoint}.
          It admits a left adjoint $f_\sharp$ which commutes with $\Sigma^\infty$.
          In particular, $f^*$ commutes with $\Omega^\infty$.

          \item\label{item:funct SH/smooth/bc}
          \emph{Smooth base change}.
          The functor $f_\sharp$ commutes with arbitrary $*$-inverse image.
          Equivalently, $f^*$ commutes with arbitrary $*$-direct image.

          \item\label{item:funct SH/smooth/proj}
          \emph{Smooth projection formula}.
          The functor $f_\sharp$ is a morphism of $\SH(\sY)$-modules.
        \end{thmlist}

        \item\label{item:funct SH/closed}
        If $f$ is a closed immersion, then $f_* : \SH(\sX) \to \SH(\sY)$ satisfies the following properties:
        \begin{thmlist}
          \item\label{item:funct SH/closed/bc}
          \emph{Closed base change.}
          It commutes with arbitrary $*$-inverse image.

          \item\label{item:funct SH/closed/bc2}
          \emph{Smooth-closed base change.}
          It commutes with $\sharp$-direct image by smooth representable morphisms.

          \item\label{item:funct SH/closed/proj}
          \emph{Closed projection formula.}
          It is a morphism of $\SH(\sY)$-modules.

          \item\label{item:funct SH/closed/loc}
          \emph{Localization.}
          It is fully faithful.
          If the complementary open immersion $j : \sU \to \sX$ is quasi-compact, then the essential image of $f_*$ is spanned by the kernel of $j^* : \SH(\sX) \to \SH(\sU)$.
          (For example, $f_*$ is an equivalence if $f$ is surjective.)
        \end{thmlist}

        \item\label{item:funct SH/comp}
        The functor $f^* : \SH(\sY) \to \SH(\sX)$ is compact, i.e., its right adjoint $f_*$ commutes with colimits.
      \end{thmlist}
    \end{thm}

  \subsection{The \qfund case}

    We will construct $\SH(\sX)$ in increasingly greater generality, starting with the \qfund case:

    \begin{constr}\label{constr:SH basic}
      Let $\sX$ be a \qfund derived stack.

      \begin{defnlist}
        \item
        Consider the \inftyCat $\MotSpc(\sX)_\bullet$ of pointed motivic \anis, with the symmetric monoidal structure given by smash product.
        Let $\sT_\sX$ denote the (small) set of objects
        \[ \Th_\sX(\sE) \in \MotSpc(\sX)_\bullet \]
        where $\sE$ ranges over all finite locally free sheaves on $\sX$.
        Now formally adjoin to $\MotSpc(\sX)_\bullet$ the $\otimes$-inverse of every object in $\sT_\sX$ (in the sense of \cite{RobaloBridge}, \cite[\S 6.1]{HoyoisEquivariant}) to get the symmetric monoidal presentable stable \inftyCat
        \[ \SH(\sX) = \MotSpc(\sX)_\bullet[\sT_\sX^{\otimes -1}] \]
        together with the canonical symmetric monoidal colimit-preserving functor $\Sigma^\infty : \MotSpc(\sX)_\bullet \to \SH(\sX)$.\footnote{%
          Note that the objects in $\sT_\sX$ are \emph{symmetric} in the sense of \cite{RobaloBridge}.
          Indeed, this follows by functoriality from the case of $BG$, for any nice embeddable group scheme $G$ over an affine $S$, which is a special case of \cite[Lem.~6.3]{HoyoisEquivariant}.
        }

        \item
        Let $f : \sX \to \sY$ be a morphism of \qfund derived stacks.
        Note that the functor $f^* : \MotSpc(\sY)_\bullet \to \MotSpc(\sX)_\bullet$ preserves Thom \anis, i.e.,
        \[ f^*( \Th_\sY(\sE) ) = \Th_\sX(f^*(\sE)) \]
        for every finite locally free $\sE$ over $\sY$.
        By universal properties it follows that there is a unique extension of $f^*$ to a symmetric monoidal colimit-preserving
        \[ f^* : \SH(\sY) \to \SH(\sX) \]
        such that $f^*$ commutes with $\Sigma^\infty$ and $f_*$ commutes with $\Omega^\infty$.
      \end{defnlist}

      This defines a morphism of presheaves $\Sigma^\infty : \MotSpc_\bullet^* \to \SH^*$ on the site of \qfund derived stacks.
    \end{constr}

    \begin{rem}\label{rem:a0s7dg}
      For a \qfund derived stack $\sX$, the assignment
      \[ \sE \mapsto \Sigma^\infty\Th_\sX(\sE) \]
      defines by construction a functor from the \inftyGrpd of finite locally free sheaves on $\sX$ to the Picard $\Einfty$-group of $\otimes$-invertible objects in $\SH(\sX)$.
      Moreover, since it sends direct sums to tensor products (\remref{rem:Thom exact}), it is a map of $\Einfty$-monoids.
      Since the target is group-complete, the map factors through the group completion of the source, which is the algebraic K-theory \ani of $\sX$ (since \qfund stacks have the resolution property, see \propref{prop:2rgoy10}).
      Thus we have a map of presheaves
      \[ \Th : \K \to \Pic(\SH^*) \]
      on the site of \qfund derived stacks.
    \end{rem}

    \begin{lem}\label{lem:basic stab}
      Let $f : \sX \to \sY$ be a representable morphism of \qfund derived stacks.
      Then the morphism of $\SH(\sX)$-modules
      \[ \MotSpc(\sX)_\bullet \otimes_{\MotSpc(\sY)_\bullet} \SH(\sY) \to \SH(\sX) \]
      is an equivalence.
    \end{lem}
    \begin{proof}
      First assume that the morphism $f$ is quasi-affine.
      In this case, for every finite locally free sheaf $\sE$ on $\sX$, the proof of \lemref{lem:8phi112} shows that one can find a finite locally free sheaf $\sE'$ on $\sY$ and a surjection
      \[ f^*(\sE') \twoheadrightarrow \sE. \]
      If $\sK$ denotes its kernel, then we get a canonical isomorphism (\remref{rem:Thom exact})
      \[ \Th_{\sX}(f^*\sE')  \simeq \Th_{\sX}(\sE) \wedge \Th_{\sX}(\sK) \]
      in $\MotSpc(\sX)_\bullet$.
      Thus $\Th_{\sX}(\sE)$ is invertible in $\MotSpc(\sX)_\bullet[f^*(\sT_\sY)^{\otimes-1}]$.

      Next consider the case of a general representable morphism.
      Since $\sX$ has affine diagonal, we can apply \thmref{thm:sumihiro}\itemref{item:sumihiro/scallop} to get a scallop decomposition $(\sU_i,\sV_i,u_i)$, where $\sV_i$ are quasi-affine over $\sY$.
      Hence by induction and the quasi-affine case above, it will suffice to show that for any Nisnevich square
      \[ \begin{tikzcd}
        \sW \ar{r}\ar{d}
        & \sV \ar{d}{p}
        \\
        \sU \ar{r}{j}
        & \sX
      \end{tikzcd} \]
      where $j$ is an open immersion and $p$ is an étale morphism inducing an isomorphism away from $\sU$, we have that if the claim holds for $f_\sU = f|_\sU : \sU \to \sY$, $f_\sV = f|_\sV$ and $f_\sW = f|_\sW$, then it also holds for $f$.
      Indeed, $\otimes$-invertibility of $\Th_\sX(\sE)$ in $\MotSpc(\sX)_\bullet[f^*(\sT_\sY)^{\otimes -1}]$ (for any finite locally free $\sE$ on $\sX$) is equivalent to invertibility of the canonical evaluation morphism
      \[
        \mrm{ev} : \Th_\sX(\sE) \wedge \uHom(\Th_\sX(\sE), \pt_+)
        \to \pt_+
      \]
      where $\pt_+$ is the monoidal unit, image of $\L\h_\sX(\sX)_+ \in \MotSpc(\sX)_\bullet$.
      Since smooth inverse image commutes with $\Th$, $\otimes$ and $\uHom$ (by \remref{rem:Thf^*}, \propref{prop:f^*}, and by adjunction from the projection formula of \propref{prop:f_sharp}, respectively), this immediately follows from Nisnevich separation (\propref{prop:Nis sep}).
    \end{proof}

    \begin{lem}\label{cor:gen SH}
      Let $\sX$ be a \qfund derived stack.
      \begin{thmlist}
        \item
        For every $\sX'\in \Sm_{/\sX}$ and every $\alpha \in \K(\sX)$, the object
        \[ \Sigma^\infty_+ (\sX') \otimes \Th_\sX(\alpha) \in \SH(\sX) \]
        is compact.
        
        \item
        Choose a quasi-affine morphism $f : \sX \to BG$ where $G$ is an embeddable nice group scheme over an affine scheme $S$.
        Then $\SH(\sX)$ is compactly generated by the set of objects
        $$\Sigma^{\infty}_+(\sX') \otimes \Th_\sX(f^*(\sE))^{\otimes -1},$$
        where $\sX' \in \Sm_{/\sX}$, and $\sE$ is a finite locally free sheaf on $BG$.
      \end{thmlist}
    \end{lem}
    \begin{proof}
      It is a formal consequence of \propref{prop:gen} and general facts about $\otimes$-inversion of objects (see e.g. the proof of \cite[Prop.~6.4]{HoyoisEquivariant}) that objects of the form $\Sigma^{\infty}_+(\sX') \otimes \Th_\sX(\sE)^{\otimes -1}$, where $\sX' \in \Sm_{/\sX}$ and $\sE$ is a finite locally free sheaf on $\sX$, generate $\SH(\sX)$ under sifted colimits.
      The second claim follows by combining this with the canonical equivalence $\MotSpc(\sX)_\bullet[f^*(\sT_{BG})^{\otimes -1}] \simeq \SH(\sX)$ (\lemref{lem:basic stab}).

      Recall that the object $\Th_\sX(\sE) \in \MotSpc(\sX)$ is compact for every finite locally free $\sE$ on $\sX$ (\examref{exam:Thom}).
      By construction of $\SH(\sX)$, this implies that $\Sigma^\infty$ is a compact functor.
      The functor $- \otimes \Th_\sX(\alpha) : \SH(\sX) \to \SH(\sX)$ is also compact, for every $\alpha \in \K(\sX)$, since it is an equivalence.
      Thus the claim follows from the fact that the object $\L \h_\sX(\sX') \in \MotSpc(\sX)$ is compact for every $\sX'\in\Sm_{/\sX}$ (\propref{prop:gen}\itemref{item:0as7df}).
    \end{proof}

    \begin{rem}\label{rem:SH basic proof}
      At this point, we can already prove Theorems~\ref{thm:SH} and \ref{thm:funct SH} on the site of \qfund derived stacks.
      Indeed, \thmref{thm:SH}\ref{item:SH/BG} holds by construction, \ref{item:SH/Thom} by \remref{rem:a0s7dg}, and \ref{item:SH/tensor} by \lemref{lem:basic stab}.
      The last property \ref{item:SH/Nis} is a standard consequence of the properties asserted in \thmref{thm:funct SH} (cf. \cite[Prop.~6.24]{HoyoisEquivariant}).
      
      \thmref{thm:funct SH}\ref{item:funct SH/comp} follows from \lemref{cor:gen SH}.
      \thmref{thm:funct SH}\ref{item:funct SH/smooth} follows immediately from the analogous unstable statements (Propositions~\ref{prop:f_sharp}, \ref{prop:smooth bc}, \ref{prop:Nis sep}) by extending scalars along $\Sigma^\infty$ (in view of \thmref{thm:SH}).
      For example, for smooth base change we argue as follows.
      By the unstable statement we have the commutative square
      \[
        \begin{tikzcd}
          \MotSpc(\sX) \ar{r}{p_\sharp}\ar{d}{f^*}
          & \MotSpc(\sY) \ar{d}{g^*}
          \\
          \MotSpc(\sX') \ar{r}{q_\sharp}
          & \MotSpc(\sY').
        \end{tikzcd}
      \]
      Then by \lemref{lem:basic stab}, extension of scalars along $\Sigma^\infty$ gives the upper commutative square in the diagram
      \[
        \begin{tikzcd}
          \SH(\sX) \ar{r}{p_\sharp}\ar{d}{f^*}
          & \SH(\sY) \ar{d}{g^*}
          \\
          \MotSpc(\sX') \otimes_{\MotSpc(\sX)} \SH(\sX) \ar{r}{q_\sharp}\ar{d}
          & \MotSpc(\sY') \otimes_{\MotSpc(\sX)} \SH(\sX)\ar{d}
          \\
          \MotSpc(\sX') \otimes_{\MotSpc(\sX')} \SH(\sX') \ar{r}\ar[equals]{d}
          & \MotSpc(\sY') \otimes_{\MotSpc(\sX')} \SH(\sX') \ar[equals]{d}
          \\
          \SH(\sX') \ar{r}{q_\sharp}
          & \SH(\sY')
        \end{tikzcd}
      \]
      where the middle square commutes tautologically and the lower right-hand isomorphism is \lemref{lem:basic stab}.

      In \thmref{thm:funct SH}\ref{item:funct SH/closed}, localization is deduced similarly from the unstable statement (\thmref{thm:loc unstable}, compare the proof of \cite[Thm.~1.36]{KhanSix}).
      The other properties are immediate consequences of localization, just as in \cite[Lems~2.19, 2.20, 2.21]{KhanSix}.
    \end{rem}

  \subsection{Proof of Theorem~\ref{thm:SH}}

    We now extend \constrref{constr:SH basic} to general scalloped stacks.

    \begin{constr}\label{constr:SH global}
      Let $\sX$ be a scalloped derived stack.
      We define the \inftyCat $\SH(\sX)$ by the formula
      \[ \SH(\sX) = \lim_{(\sU,\sU\to\sX)} \SH(\sU) \]
      where the limit is taken over the \inftyCat of pairs $(\sU, u)$ where $\sU$ is \qfund and $u : \sU \to \sX$ is a representable étale morphism.
    \end{constr}

    To show that this determines a presheaf $\SH^*$ satisfying Nisnevich descent, we will show that this construction is right Kan extended from the subcategory of \qfund derived stacks.

    \begin{lem}\label{lem:9sydagy}
      Let $\sC$ be the \inftyCat of scalloped derived stacks and $\sC_0$ the full subcategory of \qfund derived stacks.
      Regard $\sC$ and $\sC_0$ as sites with the representable Nisnevich topology.
      Then restriction along the inclusion $i : \sC_0 \hookrightarrow \sC$ induces an equivalence from the \inftyCat of sheaves on $\sC$ to the \inftyCat of sheaves on $\sC_0$.
      In particular, every sheaf on $\sC$ is right Kan extended from $\sC_0$.
    \end{lem}
    \begin{proof}
      Note that $i$ is both topologically continuous and cocontinuous, so that restriction $i^*$ and its right adjoint $i_*$ (right Kan extension) preserve sheaves.
      Its left adjoint $\L_{\mrm{Nis}} i_!$ is Nisnevich-localized left Kan extension.
      Since $i$ is fully faithful, it is clear that $\L_{\mrm{Nis}} i_!$ is fully faithful.
      Hence the claim follows from the fact that $\L_{\mrm{Nis}} i_!$ generates the \inftyCat of sheaves on $\sC$ under colimits, which is proven exactly as in \propref{prop:gen}.
    \end{proof}

    \begin{rem}\label{rem:0s7gf1}
      Let $\SH'^*$ denote the right Kan extension of the presheaf $\SH^*$, defined on the site of \qfund derived stacks (\constrref{constr:SH basic}), to the site of scalloped derived stacks.
      By \lemref{lem:9sydagy}, $\SH'^*$ satisfies Nisnevich descent.
      Recall that for any scalloped $\sX$, $\SH'(\sX)$ can be computed by the same limit as in \constrref{constr:SH global}, except taken over the \inftyCat of pairs $(\sU, u)$ where $\sU$ is \qfund and $u : \sU \to \sX$ is \emph{any} morphism.
      We claim that the canonical functor $\SH'(\sX) \to \SH(\sX)$ is an equivalence.
      This is clear when $\sX$ is \qfund.
      In general, $\sX$ admits a scallop decomposition $(\sU_i,\sV_i,u_i)_i$ of $\sX$ where $\sV_i$ are \qfund (\thmref{thm:scallop}\itemref{item:scallop/cover}).
      By induction, it will therefore suffice to show that if $\sX$ is covered by a Nisnevich square
      \[ \begin{tikzcd}
        \sW \ar{r}\ar{d}
        & \sV\ar{d}{p}
        \\
        \sU \ar{r}{j}
        & \sX
      \end{tikzcd} \]
      such that the claim is true for $\sU$, $\sV$ and $\sW$, then it is also true for $\sX$.
      By Nisnevich descent, we have a canonical equivalence
      \[
        \SH'(\sX)
        \simeq \SH'(\sU) \fibprod_{\SH'(\sW)} \SH'(\sV)
        \simeq \SH(\sU) \fibprod_{\SH(\sW)} \SH(\sV).
      \]
      Since each of the maps $\sU \to \sX$, $\sV \to \sX$ and $\sW \to \sX$ are representable and étale, there is a projection functor $\SH(\sX) \to \SH'(\sX)$, which one checks is inverse to the canonical one.
    \end{rem}

    \begin{constr}\leavevmode
      \begin{defnlist}
        \item
        It follows from \remref{rem:0s7gf1} that \constrref{constr:SH global} lifts to a presheaf $\SH^*$, given by
        \[ \sX \mapsto \SH(\sX), \quad f \mapsto f^*, \]
        on the site of scalloped derived stacks.
        This presheaf takes values in the \inftyCat of symmetric monoidal presentable stable \inftyCats and symmetric monoidal colimit-preserving functors (since the forgetful functor from the latter to the \inftyCat of large \inftyCats is limit-preserving).
        
        \item
        Since $\MotSpc^*_\bullet$ is also right Kan extended from \qfund stacks (by \lemref{lem:9sydagy}), it follows that $\Sigma^\infty : \MotSpc_\bullet^* \to \SH^*$ also admits a unique extension to the site of scalloped derived stacks.

        \item
        Similarly, right Kan extension\footnote{%
          Algebraic K-theory satisfies Nisnevich descent on scalloped derived stacks (see \thmref{thm:Thomason}), hence is also right Kan extended from \funds.
        } produces an extension of the map $\Th : \K \to \Pic(\SH^*)$ (\remref{rem:a0s7dg}) to the site of scalloped derived stacks.
      \end{defnlist}
    \end{constr}

  \subsection{Proof of Theorem~\ref{thm:funct SH}\itemref{item:funct SH/smooth}}

    \sssec{Left adjoint \itemref{item:funct SH/smooth/adj}.}

      Note that, over the site of \fund derived stacks and smooth representable morphisms, $\SH^*$ takes values in the \inftyCat of presentable \inftyCats and \emph{right adjoint} functors (by \thmref{thm:funct SH}\ref{item:funct SH/smooth}, which we have already proven in the \fund case).
      It follows from \cite[Thm.~5.5.3.18]{LurieHTT} that the same holds for the extension of $\SH^*$ to the site of scalloped derived stacks, since the transition functors in the limit in \constrref{constr:SH global} are in particular smooth representable.

      It follows that, for every smooth representable morphism $p : \sX \to \sY$ of scalloped derived stacks $p^*$ admits a left adjoint
      \[ p_\sharp : \SH(\sX) \to \SH(\sY). \]
      It follows from the construction that it commutes with $\Sigma^\infty$.

    \sssec{Smooth base change \itemref{item:funct SH/smooth/bc}.}
    \label{sssec:sf0g61oyl1}

      Let $p : \sX \to \sY$ be a smooth representable morphism of scalloped derived stacks.
      For any morphism of scalloped derived stacks $g : \sY' \to \sY$, form the base change square
      \[ \begin{tikzcd}
        \sX' \ar{r}{q}\ar{d}{f}
        & \sY' \ar{d}{g}
        \\
        \sX \ar{r}{p}
        & \sY.
      \end{tikzcd} \]
      The claim is that the exchange transformation
      \begin{equation*}
        \Ex^*_\sharp : q_\sharp f^* \to g^* p_\sharp,
      \end{equation*}
      is invertible.
      For the proof, we will require the following lemma.

      \begin{lem}\label{lem:a0s7dft}
        Let the notation be as above.
        Suppose that $\sX$ is covered by a Nisnevich square
        \[ \begin{tikzcd}
          \sW \ar{r}\ar{d}\ar{rd}{w}
          & \sV\ar{d}{v}
          \\
          \sU \ar{r}{u}
          & \sX
        \end{tikzcd} \]
        such that $(p \circ u)_\sharp$, $(p \circ v)_\sharp$, and $(p \circ w)_\sharp$ satisfy base change against $g^*$.
        Then $p_\sharp$ satisfies base change against $g^*$.
      \end{lem}
      \begin{proof}
        The assumption is that the exchange transformation
        \begin{equation*}
          \Ex^*_\sharp : q_{\sU',\sharp} f_\sU^* \to g^* p_{\sU,\sharp}
        \end{equation*}
        associated to the composite square
        \[ \begin{tikzcd}
          \sU' \ar{r}{u'}\ar{d}{f_\sU}
          & \sX' \ar{r}{q}\ar{d}{f}
          & \sY' \ar{d}{g}
          \\
          \sU \ar{r}{u}
          & \sX \ar{r}{p}
          & \sY,
        \end{tikzcd} \]
        is invertible, where $\sU' = \sU \fibprod_\sY \sY'$, $f_\sU : \sU' \to \sU$ is the base change of $f$, $p_\sU = p\circ u : \sU \to \sX$, and $q_{\sU'} = q\circ u' : \sU' \to \sY'$ is the base change of $p_\sU$.
        Similarly for $\Ex^*_\sharp : q_{\sV',\sharp} f_\sV^* \to g^* p_{\sV,\sharp}$ and $\Ex^*_\sharp : q_{\sW',\sharp} f_\sW^* \to g^* p_{\sW,\sharp}$.

        To show that the exchange transformation $\Ex^*_\sharp : q_\sharp f^* \to g^* p_\sharp$ is invertible, combine the above isomorphisms with the descent isomorphisms,
        \begin{align*}
          p_\sharp &\simeq p_{\sU,\sharp} u^* \fibcoprod_{p_{\sW,\sharp} w^*} p_{\sV,\sharp} v^*\\
          q_\sharp &\simeq q_{\sU',\sharp} u'^* \fibcoprod_{q_{\sW',\sharp} w'^*} q_{\sV',\sharp} v'^*,
        \end{align*}
        cf. \propref{prop:etale exc}.
      \end{proof}

      We now return to the proof of smooth base change.

      \emph{Case 0.}
      If $\sX$, $\sY$, $\sX'$ and $\sY'$ are \qfund, then the claim is \remref{rem:SH basic proof}.

      \emph{Case 1.}
      Assume $\sY$ and $\sY'$ are \qfund and $\sX$ and $\sX'$ are scalloped.
      Since $\sX$ is representable over $\sY$, we have $\sX = [X/G]$ where $\sY = [Y/G]$, $G$ is an embeddable nice group scheme over an affine scheme $S$, $Y$ is a quasi-affine derived $G$-scheme, and $X$ is a quasi-compact derived algebraic space over $Y$ with $G$-action.
      Then by \thmref{thm:sumihiro}\itemref{item:sumihiro/scallop} and \lemref{lem:a0s7dft}, we may reduce to the case where $X$ is quasi-affine.
      In this case $p : \sX \to \sY$ is quasi-affine (since $U \to Y$ is quasi-affine), hence $\sX$ is \qfund (since $\sY$ is \qfund).

      \emph{Case 2.}
      Assume $\sY$ is \qfund and $\sX$, $\sX'$, $\sY'$ are scalloped.
      By \thmref{thm:scallop}\itemref{item:scallop/cover} and Case~1 it will suffice to show that if we have a Nisnevich square
      \[ \begin{tikzcd}
        \sW \ar{r}\ar{d}\ar{rd}{w}
        & \sV \ar{d}{v}
        \\
        \sU \ar{r}{u}
        & \sY'
      \end{tikzcd} \]
      such that the claim holds after replacing $g : \sY' \to \sY$ by any of the composites $g\circ u$, $g\circ v$, or $g\circ w$, then it also holds for $g$.
      This follows immediately from Nisnevich descent and \cite[Pt.~I, Chap.~1, Lem.~2.6.4]{GaitsgoryRozenblyum}.

      \emph{Case 3.}
      Let $\sX$, $\sX'$, $\sY$, and $\sY'$ be scalloped.
      By \thmref{thm:scallop}\itemref{item:scallop/cover} and Case~2 it will suffice to show that if we have a Nisnevich square
      \[ \begin{tikzcd}
        \sW \ar{r}\ar{d}\ar{rd}{w}
        & \sV \ar{d}{v}
        \\
        \sU \ar{r}{u}
        & \sY
      \end{tikzcd} \]
      such that the claim holds after base changing $p : \sX \to \sY$ (and hence the whole square) along any of $u$, $v$, or $w$, then it also holds for $p$ itself.
      This follows immediately from Nisnevich descent and \cite[Pt.~I, Chap.~1, Lem.~2.6.4]{GaitsgoryRozenblyum}.

    \sssec{Smooth projection formula \itemref{item:funct SH/smooth/proj}.}
    \label{sssec:a0sgu0p1}

      By adjunction, $p_\sharp$ inherits from $p^*$ a canonical structure of colax morphism of $\SH(\sY)$-modules.
      In particular, there are canonical morphisms
      \[ p_\sharp (\sF \otimes p^*(\sG)) \to p_\sharp(\sF) \otimes \sG \]
      for every object $\sF \in \SH(\sX)$ and $\sG \in \SH(\sY)$, which we claim are invertible.
      This claim is Nisnevich-local on $\sY$, in view of smooth base change \itemref{item:funct SH/smooth/bc}.
      Arguing as in \lemref{lem:a0s7dft}, we also see that it is local on $\sX$.
      Hence we can reduce to the case where $\sX$ and $\sY$ are \qfund, proven in \remref{rem:SH basic proof}.

  \subsection{Proof of Theorem~\ref{thm:funct SH}\itemref{item:funct SH/closed}}

    As in \itemref{item:funct SH/smooth}, we can argue Nisnevich-locally on the target.
    When $\sY$ is \qfund, then so is $\sX$, and we are reduced to the situation of \remref{rem:SH basic proof}.

  \subsection{Proof of Theorem~\ref{thm:funct SH}\itemref{item:funct SH/comp}}

    Let $f : \sX \to \sY$ be a morphism of scalloped derived stacks.
    The claim is that for every filtered diagram $(\sF_\alpha)_\alpha$ of objects of $\SH(\sX)$, the canonical morphism in $\SH(\sY)$
    \[ \colim_\alpha f_*(\sF_\alpha) \to f_*\big( \colim_\alpha \sF_\alpha \big) \]
    is invertible.
    If $v : \sY_0 \twoheadrightarrow \sY$ is a representable Nisnevich cover by a \fund derived stack $\sY_0$, then since $v^*$ commutes with colimits and with $f_*$ (\thmref{thm:funct SH}\itemref{item:funct SH/smooth}), we may reduce to the case where $\sY = \sY_0$ is \qfund.

    The claim is also local on $\sX$ in the sense that if
    \[ \begin{tikzcd}
      \sW \ar{r}\ar{d}\ar{rd}{w}
      & \sV \ar{d}{v}
      \\
      \sU \ar{r}{u}
      & \sX
    \end{tikzcd} \]
    is a Nisnevich square such that the claim holds for $f_U = f \circ u$, $f_V = f \circ v$ and $f_W = f \circ w$ in place of $f$, then it also holds for $f$.
    Indeed by Nisnevich descent (see \propref{prop:etale exc}), there is a canonical isomorphism
    \[ f_* \simeq f_{U,*}u^* \fibprod_{f_{W,*}w^*} f_{V,*}v^*, \]
    and filtered colimits commute with finite limits in $\SH(\sY)$.

    Thus by induction (on the length of an appropriate scallop decomposition of $\sX$), we eventually reduce to the case where $\sX$ and $\sY$ are both \qfund, which was proven in \remref{rem:SH basic proof}.


\section{Axiomatization}
\label{sec:six}

  We now make a brief interlude to give an axiomatic description of the results proven so far.

  \subsection{\texorpdfstring{$(*,\sharp,\otimes)$}{(*,\#,(x))}-formalisms}

    \begin{notat}\label{notat:as07fgafo}\leavevmode
      \begin{defnlist}
        \item
        Given a presheaf of \inftyCats $\D^*$ on the \inftyCat of scalloped derived stacks, we will write
        \[ \D(\sX) := \D^*(\sX) \]
        for every scalloped derived stack $\sX$.
        For every morphism $f : \sX \to \sY$, we denote by
        \[ f^* := \D^*(f) : \D(\sY) \to \D(\sX) \]
        the functor of \emph{inverse image} along $f$.

        \item
        If $\D^*$ takes values in \emph{presentable} \inftyCats and colimit-preserving functors, then we say simply that $\D^*$ is a \emph{presheaf of presentable \inftyCats}.
        In this case, every inverse image functor $f^*$ admits a right adjoint $f_*$ called \emph{direct image} along $f$.

        \item
        If $\D^*$ moreover factors through the \inftyCat of \emph{symmetric monoidal} presentable \inftyCats, then we say that $\D^*$ is a \emph{presheaf of symmetric monoidal presentable \inftyCats}.
        We write
        $$\otimes : \D(\sX) \otimes \D(\sX) \to \D(\sX)$$
        for the monoidal product and $\un_\sX \in \D(\sX)$ for the monoidal unit over any $\sX$.
        Since $\otimes$ commutes with colimits in each argument (recall our conventions), it admits as right adjoint an internal hom bifunctor
        $$\uHom : \D(\sX)^\op \times \D(\sX) \to \D(\sX).$$
      \end{defnlist}
    \end{notat}

    \begin{defn}\label{defn:isyvq}
      A \emph{$(*,\sharp,\otimes)$-formalism} (on scalloped derived stacks) is a presheaf $\D^*$ of symmetric monoidal presentable \inftyCats on the site of scalloped derived stacks satisfying the following properties.
      \begin{defnlist}
        \item
        For every smooth representable morphism $f : \sX \to \sY$, the inverse image functor $f^*$ admits a left adjoint
        \[ f_\sharp : \D(\sY) \to \D(\sX). \]
        (Cf. \thmref{thm:funct SH}\itemref{item:funct SH/smooth}\itemref{item:funct SH/smooth/adj}.)

        \item
        The $\sharp$-direct image functors satisfy base change against arbitrary $*$-inverse image.
        (Cf. \thmref{thm:funct SH}\itemref{item:funct SH/smooth}\itemref{item:funct SH/smooth/bc}.)

        \item
        The $\sharp$-direct image functors satisfy the projection formula.
        That is, $f_\sharp : \D(\sX) \to \D(\sY)$ is a morphism of $\D(\sY)$-modules, where $\D(\sX)$ is regarded as a $\D(\sY)$-module via the symmetric monoidal functor $f^*$.
        (Cf. \thmref{thm:funct SH}\itemref{item:funct SH/smooth}\itemref{item:funct SH/smooth/proj}.)

        \item\emph{Additivity}.
        For any finite family $(\sX_\alpha)_\alpha$ of scalloped derived stacks, the canonical functor
        \[ \D \big(\coprod_\alpha \sX_\alpha \big) \to \prod_\alpha \D(\sX_\alpha) \]
        is an equivalence.
      \end{defnlist}
    \end{defn}

    \begin{rem}
      In practice, it is useful to consider the variant of \defnref{defn:isyvq} where the site of scalloped derived stacks is replaced with the site of scalloped derived stacks over some base algebraic space $S$, or some nice full subcategory thereof.
    \end{rem}

    \begin{defn}[Thom twist]
      Let $\D^*$ be a $(*,\sharp,\otimes)$-formalism.
      Let $\sE$ be a finite locally free sheaf on a scalloped derived stack $\sX$.
      The \emph{Thom twist} $\vb{\sE}$ is the endofunctor of $\D(\sX)$ given by
      \[ \sF \mapsto \sF\vb{\sE} := p_\sharp 0_*(\sF), \]
      where $p$ is the total space of $\sE$ and $0$ is the zero section.
    \end{defn}

  \subsection{The Voevodsky conditions}
  \label{ssec:six/Voev}

    The following key conditions were singled out by Voevodsky in the case of schemes (see \cite[\S 2, 1.2.1]{VoevodskyCross}).

    \begin{defn}\label{defn:0paguh107}
      Let $\D^*$ be a $(*,\sharp,\otimes)$-formalism.
      We say that $\D^*$ \emph{satisfies the Voevodsky conditions} if the following all hold:
      \begin{defnlist}
        \item\emph{Homotopy invariance}.
        For every vector bundle $p : \sV \to \sX$, the unit map
        \[ \id \to p_* p^* \]
        is invertible.

        \item\emph{Localization}.\label{item:0paguh107/localization}
        For every complementary closed-open pair
        \[ \begin{tikzcd}
          \sZ \ar[hookrightarrow]{r}{i}
          & \sX \ar[hookleftarrow]{r}{j}
          & \sU
        \end{tikzcd} \]
        of scalloped derived stacks, the functor $i_*$ is fully faithful with essential image spanned by objects in the kernel of $j^*$.

        \item\emph{Thom stability}.
        For every finite locally free sheaf $\sE$ on a scalloped derived stack $\sX$, the endofunctor $\vb{\sE}: \D(\sX) \to \D(\sX)$ is an equivalence.
      \end{defnlist}
    \end{defn}

    \begin{rem}\label{rem:iadsfp}
      The base change formula for $\sharp$-direct image implies that for any complementary closed-open pair
      \[ \begin{tikzcd}
        \sZ \ar[hookrightarrow]{r}{i}
        & \sX \ar[hookleftarrow]{r}{j}
        & \sU,
      \end{tikzcd} \]
      we have the identities
      \begin{equation*}
        j^*j_\sharp = \id,
        \quad j^*j_* = \id,
        \quad j^*i_* \simeq 0,
        \quad i^* j_\sharp \simeq 0.
      \end{equation*}
      The localization property can be reformulated as exactness of the triangle
      \[
        j_\sharp j^*
        \xrightarrow{\mrm{counit}} \id
        \xrightarrow{\mrm{unit}} i_*i^*,
      \]
      or by passing to right adjoints,
      \[
        i_*i^! \simeq i_!i^!
        \xrightarrow{\mrm{counit}} \id
        \xrightarrow{\mrm{unit}} j_*j^*.
      \]
    \end{rem}

    \begin{rem}[Tate twists]
      Given a scalloped derived stack $\sX$, let $q : (\A^1\setminus \{0\}) \times \sX \to \sX$ denote the projection.
      The unit section of $\A^1$ defines a canonical morphism $\un_\sX \to q_\sharp (\un_\sX)$ whose fibre we denote $\un_\sX(1)$.
      Since $\un_\sX\vb{1} \simeq \un_\sX(1)[2]$, Thom stability implies that $\un_\sX(1)$ is $\otimes$-invertible.
      The \emph{Tate twist} by $n\in\bZ$ is the endofunctor of $\D(\sX)$ given by
      \[ \sF \mapsto \sF(n) := \sF \otimes (\sF(1))^{\otimes n}. \]
    \end{rem}

  \subsection{Constructible separation}

    Let $\D^*$ be a $(*,\sharp,\otimes)$-formalism satisfying the Voevodsky conditions.
    The following is an immediate consequence of the localization property (see e.g. \cite[2.11, 2.13]{KhanSix}).

    \begin{prop}\label{prop:constrsep}
      For any constructible covering\footnote{
        i.e., a family that generates a covering for the constructible topology
      } $(j_\alpha : \sX_\alpha \to \sX)_\alpha$ of a scalloped derived stack $\sX$, the inverse image functors
      \[ j_\alpha^* : \D(\sX) \to \D(\sX_\alpha) \]
      are jointly conservative.
    \end{prop}

    \begin{cor}[Nil invariance]\label{cor:nil inv}
      Let $i : \sX' \to \sX$ be a surjective closed immersion of scalloped derived stacks.
      Then the pair of adjoint functors
      \[
        i^* : \D(\sX) \to \D(\sX'),
        \quad
        i_* : \D(\sX') \to \D(\sX)
      \]
      is an equivalence of \inftyCats.
      In particular, for every scalloped derived stack $\sX$, there are canonical equivalences
      \[ \D(\sX) \simeq \D(\sX_\cl) \simeq \D(\sX_{\cl,\red}) \]
      where $\sX_\cl$ is the classical truncation and $\sX_{\cl,\red}$ is its reduction.
    \end{cor}

  \subsection{Nisnevich descent}

    Let $\D^*$ be a $(*,\sharp,\otimes)$-formalism satisfying the Voevodsky conditions.
    We have the following (see e.g. \cite[4.26, 4.52]{KhanSix}):

    \begin{prop}[Étale excision]\label{prop:etale exc}
      Let $f : \sX' \to \sX$ be a representable étale morphism which induces an isomorphism away from a quasi-compact open immersion $j : \sU \to \sX$.
      Then the commutative square
      \begin{equation*}
        \begin{tikzcd}
          \id \ar{r}\ar{d}
          & j_*j^* \ar{d}
          \\
          f_*f^* \ar{r}
          & g_*g^*,
        \end{tikzcd}
      \end{equation*}
      is homotopy cartesian, where $g : f^{-1}(\sU) \to \sX$.
    \end{prop}

    \begin{prop}[Nisnevich descent]\label{prop:Nis desc}
      The presheaf of \inftyCats $\D^*$ satisfies Nisnevich descent on the \inftyCat of scalloped derived stacks.
      In particular, for any scalloped derived stack $\sX$ and Nisnevich covering family $(f_\alpha : \sX_\alpha \to \sX)_\alpha$, the inverse image functors
      \[ f_\alpha^* : \D(\sX) \to \D(\sX_\alpha) \]
      are jointly conservative.
    \end{prop}

  \subsection{The example of \texorpdfstring{$\SH^*$}{SH}}

    \begin{exam} \label{exam:SH}
      Consider the presheaf $\SH^*$, given by $\sX \mapsto \SH(\sX)$, $f \mapsto f^*$ (see \thmref{thm:SH}).
      This is a $(*,\sharp,\otimes)$-formalism satisfying the Voevodsky conditions:
      \begin{defnlist}
        \item
        The existence of $\sharp$-direct images for smooth representable morphisms, satisfying the base change and projection formula, was proven in \thmref{thm:funct SH}\itemref{item:funct SH/smooth}.

        \item
        The additivity property follows from Nisnevich descent (\thmref{thm:SH}\itemref{item:SH/Nis}).

        \item
        For homotopy invariance, use Nisnevich descent to reduce to the case of $\sX$ \fund, in which case it holds by construction.
        
        \item
        For Thom stability, observe that by localization (\thmref{thm:funct SH}\itemref{item:funct SH/closed}\itemref{item:funct SH/closed/loc}) and the smooth and closed projection formulas (\thmref{thm:funct SH}\itemref{item:funct SH/smooth}\itemref{item:funct SH/smooth/proj} and \itemref{item:funct SH/closed}\itemref{item:funct SH/closed/proj}), there are canonical isomorphisms of functors
        \[ \vb{\sE} \simeq (-) \otimes \Sigma^\infty\Th_\sX(\sE) \]
        for every finite locally free $\sE$ on $\sX$.
        These are equivalences by \thmref{thm:SH}\itemref{item:SH/Thom}.
      \end{defnlist}
    \end{exam}

    This example is universal in the following sense:

    \begin{prop}\label{prop:univ}
      Let $\D^*$ be a $(*,\sharp,\otimes)$-formalism on scalloped derived stacks satisfying the Voevodsky conditions.
      Then there exists a unique system of colimit-preserving functors
      \[ R_\sX : \SH(\sX) \to \D(\sX) \]
      for every scalloped derived stack $\sX$, which commute with $\sharp$-direct images (along smooth representable morphisms), inverse images (along arbitrary morphisms), tensor products, and arbitrary Thom twists.
    \end{prop}
    \begin{proof}
      By Nisnevich descent (\propref{prop:Nis desc}), both $\SH^*$ and $\D^*$ are right Kan extended from the subcategory of \fund stacks.
      For \fund stacks, this holds by \remref{rem:invert quot} and the construction of $\MotSpc^*$; compare \cite[Rem.~2.14]{KhanSix}.
    \end{proof}

    \begin{exam}[Étale motivic spectra]
      The étale-local stable motivic homotopy category $\SH_\et^*$ defines a $(*,\sharp,\otimes)$-formalism satisfying the Voevodsky conditions.
      In fact, this formalism extends to the site of all derived algebraic stacks (not necessarily scalloped).
      More generally, any constructible \inftyCat (as in \cite{KhanSix}) satisfying étale descent on the site of schemes or algebraic spaces admits a canonical extension to the site of algebraic stacks.
      Moreover, $\sharp$-direct image also exists for non-representable smooth morphisms.
      See \cite[App.~A]{KhanVirtual} and \cite{LiuZheng}, and also \examref{exam:0--81h}.
    \end{exam}

    \begin{cor}\label{cor:0pagh0p1}
      Let $\D^*$ be a $(*,\sharp,\otimes)$-formalism on scalloped derived stacks satisfying the Voevodsky conditions.
      Then for every scalloped derived stack $\sX$, the assignment $\sE \mapsto \vb{\sE}$ extends from finite locally frees to a map of \anis
      \[ \K(\sX) \to \Aut_{\D(\sX)}(\D(\sX)) \simeq \Pic(\D(\sX)), \]
      to the \ani of $\D(\sX)$-linear autoequivalences of $\D(\sX)$, i.e., the Picard \ani of $\otimes$-invertible objects in $\D(\sX)$.
    \end{cor}
    \begin{proof}
      In fact, we have the canonical map of presheaves
      \[ \K \xrightarrow{\Th} \Pic(\SH^*) \xrightarrow{\Pic(R)} \Pic(\D^*) \]
      where the first map is as in \remref{rem:-31g0g0p1} and the second is induced by the (symmetric monoidal) realization $R : \SH^* \to \D^*$ of \propref{prop:univ}.
    \end{proof}

    \begin{notat}\label{notat:aps-8hfp01h}
      In the situation of \corref{cor:0pagh0p1}, every K-theory class $\alpha \in \K(\sX)$ induces a canonical auto-equivalence of $\D(\sX)$ which we denote by
      \[ \vb{\alpha} : \D(\sX) \to \D(\sX) \]
      and continue to call the \emph{Thom twist} by $\alpha$.
    \end{notat}

\section{Proper base change}
\label{sec:prop}

  \subsection{Statement}
  \label{ssec:07ag014}

    For the duration of this section, we fix a $(*,\sharp,\otimes)$-formalism $\D^*$ satisfying the Voevodsky conditions on the site of scalloped derived stacks.
    The following theorem summarizes our results about direct image along proper representable morphisms:

    \begin{thm}\label{thm:ipqnosdf}
      Let $f : \sX \to \sY$ be a proper representable morphism of scalloped derived stacks.
      Then we have:
      
      \begin{thmlist}
        \item\label{item:ipqnosdf/compact}
        \emph{Compactness.}
        The functor $f^* : \D(\sY) \to \D(\sX)$ is compact, i.e., its right adjoint $f_*$ commutes with colimits.
        Equivalently, $f_*$ admits a right adjoint $f^!$.

        \item\label{item:ipqnosdf/bc}
        \emph{Proper base change.}
        For any commutative square of scalloped derived stacks
        \[ \begin{tikzcd}
          \sX' \ar{r}{g}\ar{d}{u}
          & \sY' \ar{d}{v}
          \\
          \sX \ar{r}{f}
          & \sY
        \end{tikzcd} \]
        which is cartesian on classical truncations, there is a canonical isomorphism
        \[ \Ex^*_* : v^* f_* \to g_* u^* \]
        of functors $\D(\sX) \to \D(\sY')$.

        \item\label{item:ipqnosdf/smprop}
        \emph{Smooth-proper base change.}
        For any cartesian square
        \[ \begin{tikzcd}
          \sX' \ar{r}{g}\ar{d}{p}
          & \sY' \ar{d}{q}
          \\
          \sX \ar{r}{f}
          & \sY,
        \end{tikzcd} \]
        where $p$ and $q$ are smooth representable, there is a canonical isomorphism
        \[ \Ex_{\sharp,*} : q_\sharp g_* \to f_* p_\sharp \]
        of functors $\D(\sX') \to \D(\sY)$.

        \item\label{item:ipqnosdf/Atiyah}
        \emph{Atiyah duality.}
        If $f$ is moreover \emph{smooth}, then the canonical morphism of functors $\D(\sX) \to \D(\sY)$ (see \constrref{constr:epsilon})
        \[ \varepsilon_f : f_\sharp\vb{-\sL_f} \to f_* \]
        is invertible.
        In particular, $f_*\vb{\sL_f}$ is left adjoint to $f^*$.

        \item\label{item:ipqnosdf/excision}
        \emph{Proper excision.}
        If $f$ is an isomorphism away from a closed substack $\sZ \sub \sY$, then the commutative square
        \begin{equation*}\label{eq:ofybo}
          \begin{tikzcd}
            \id \ar{r}\ar{d}
            & i_*i^* \ar{d}
            \\
            f_*f^* \ar{r}
            & g_*g^*
          \end{tikzcd}
        \end{equation*}
        is homotopy cartesian in $\D(\sY)$, where $i : \sZ \to \sY$ is the inclusion and $g : f^{-1}(\sZ) \sub \sX \to \sY$ is the induced morphism.
      \end{thmlist}
    \end{thm}

  \subsection{Cdh descent}

    Before proceeding, let us mention a reformulation (\remref{rem:0as71uug1}) of proper excision as a descent statement for the $*$-direct image functor.

    \begin{defn}\label{defn:cdh}\leavevmode
      \begin{defnlist}
        \item
        The \emph{proper cdh topology} on the site of scalloped derived stacks is the Grothendieck topology associated to the pretopology generated by the following covering families:
        \begin{inlinelist}
          \item
          the empty family, covering the empty stack $\initial$;
          \item
          for every scalloped derived stack $\sX$ and every representable proper morphism $f : \sX' \to \sX$ inducing an isomorphism away from a closed immersion $i : \sZ \hook \sX$, the family $\{i,f\}$ covering $\sX$.
        \end{inlinelist}
        
        \item
        The \emph{cdh topology} on the site of scalloped derived stacks is the union of the Nisnevich (Definition~\ref{defn:Nis}) and proper cdh topologies.
      \end{defnlist}
    \end{defn}

    \begin{rem}\label{rem:0sftd07g}
      In the definition of the proper cdh topology, it suffices to take only families $\{i,f\}$ where $f$ is a \emph{projective} morphism inducing an isomorphism away from $i$.
      This follows from \cite[Cor.~2.4]{HoyoisKrishna}.
    \end{rem}

    \begin{rem}\label{rem:0as71uug1}
      Fix a scalloped derived stack $\sY$ and a coefficient $\sF \in \D(\sY)$, and let $\sC$ be (a full subcategory of) the \inftyCat of scalloped derived stacks over $\sY$.
      Proper excision, for all proper representable morphisms $f : \sX' \to \sX$ in $\sC$, is equivalent to the assertion that the presheaf on $\sC$ given by $(h : \sX \to \sY) \mapsto h_*h^*(\sF)$ satisfies descent for the proper cdh topology.
      This follows from \cite[Thm.~2.2.7]{KhanLocalization}.
    \end{rem}

    \begin{cor}\label{cor:D^* descent}
      Let $\D^*$ be a $(*,\sharp,\otimes)$-formalism satisfying the Voevodsky conditions.
      Then the presheaf of \inftyCats $\D^*$ satisfies cdh descent.
    \end{cor}
    \begin{proof}
      Nisnevich descent is \propref{prop:Nis desc}.
      Proper cdh descent follows from proper base change (\thmref{thm:ipqnosdf}) and proper cdh descent for $f_*f^*$ (\remref{rem:0as71uug1}); see \cite[Thm.~2.52]{KhanSix} for details.
    \end{proof}

  \subsection{Relative purity}

    We begin with a construction of the morphism which is asserted to be invertible in \thmref{thm:ipqnosdf}\itemref{item:ipqnosdf/Atiyah}.
    This will require a preliminary result:

    \begin{thm}[Relative purity]\label{thm:relpur closed}
      Let $i : \sX \to \sY$ be a closed immersion of scalloped derived stacks which are smooth and representable over a scalloped derived stack $\sS$.
      If $\sN_{\sX/\sY}$ denotes the conormal sheaf of $i$, and $p : \sX \to \sS$ and $q : \sY \to \sS$ denote the structural morphisms, then there is a canonical isomorphism
      \[ q_\sharp i_* \simeq p_\sharp \vb{\sN_{\sX/\sY}} \]
      of functors $\D(\sX) \to \D(\sS)$.
    \end{thm}
    \begin{proof}
      For any $i : \sX \to \sY$ as in the statement, set $P_\sS(\sX,\sY) := q_\sharp i_*$.
      Then we have
      \[ P_\sS(\sX, N_{\sX/\sY}) = p_\sharp \pi_\sharp s_* \simeq p_\sharp \vb{\sN_{\sX/\sY}}, \]
      where $\pi : N_{\sX/\sY} \to \sX$ is the projection of the normal bundle (i.e., the total space of $\sN_{\sX/\sY}$) and $s : \sX \to N_{\sX/\sY}$ is the zero section.
      The deformation to the normal bundle $D_{\sX/\sY}$ (see \cite[Thm.~4.1.13]{KhanRydh}) is a scalloped derived stack (as $D_{\sX/\sY} \to \sY$ is representable), smooth over $\sS$ and equipped with canonical morphisms of pairs (i.e., homotopy cartesian squares)
      \[ \begin{tikzcd}
        (\sX, \sY) \to (\sX \times \A^1, D_{\sX/\sY}) \gets (\sX, N_{\sX/\sY})
      \end{tikzcd} \]
      given by the inclusions of the fibres over $0$ and $1$ of $D_{\sX/\sY} \to \A^1$.
      It will suffice to show that the induced morphism (cf.~\cite[2.4.32]{CisinskiDegliseBook})
      \[ P_\sS(\sX,\sY) \to P_\sS(\sX\times\A^1, D_{\sX/\sY}) \circ \pr^* \gets P_\sS(\sX, N_{\sX/\sY}) \]
      is invertible, where $\pr : \sX \times \A^1 \to \sX$ is the projection.
      By Nisnevich descent (\propref{prop:Nis desc}) and derived invariance (\corref{cor:nil inv}) we may reduce to the case where $\sY$ (and hence $\sX$) is \fund.
      In that case the result is proven as in \cite[Prop.~5.7]{HoyoisEquivariant}.
    \end{proof}

    \begin{constr}\label{constr:epsilon}
      Let $f : \sX \to \sY$ be a smooth proper representable morphism of scalloped derived stacks.
      Its diagonal is a closed immersion $\Delta_f : \sX \to \sX \fibprod_\sY \sX$ (since $f$ is separated and representable) whose conormal sheaf is canonically identified with the cotangent complex $\sL_f$.
      The exchange transformation $\Ex_{\sharp,*}$ associated to the square
      \[ \begin{tikzcd}
        \sX \fibprod_\sY \sX \ar{r}{\pr_1}\ar{d}{\pr_2}
        & \sX \ar{d}{f}
        \\
        \sX \ar{r}{f}
        & \sY
      \end{tikzcd} \]
      gives rise to a natural transformation
      \[
        \varepsilon_f
        : f_\sharp
        = f_\sharp \pr_{2,*} \Delta_{f,*}
        \xrightarrow{\Ex_{\sharp,*}} f_* \pr_{1,\sharp} \Delta_{f,*}
        \simeq f_* \vb{\sL_f}
      \]
      where the isomorphism is relative purity (\thmref{thm:relpur closed}).
      By the smooth base change formula, formation of $\varepsilon_f$ commutes with smooth representable inverse image.
    \end{constr}

  \subsection{Reductions}

    We discuss some of the relationships between the various assertions in \thmref{thm:ipqnosdf}.

    \begin{lem}\label{lem:psdfij}
      Let $f : \sX \to \sY$ be a smooth proper representable morphism.

      \begin{thmlist}
        \item
        If $\D^*$ satisfies Atiyah duality (\thmref{thm:ipqnosdf}\itemref{item:ipqnosdf/Atiyah}) for any base change of $f$, then it also satisfies compactness, proper base change, and smooth-proper base change (\thmref{thm:ipqnosdf}\itemref{item:ipqnosdf/compact},\itemref{item:ipqnosdf/bc},\itemref{item:ipqnosdf/smprop}) for $f$.
      
        \item
        If $\D^*$ satisfies smooth-proper base change (\thmref{thm:ipqnosdf}\itemref{item:ipqnosdf/smprop}) for $f$, then it also satisfies Atiyah duality for $f$ (\thmref{thm:ipqnosdf}\itemref{item:ipqnosdf/Atiyah}).
      \end{thmlist}
    \end{lem}
    \begin{proof}
      Immediate consequence of the definition of $\varepsilon$, see \cite[Lem.~2.28, 2.29]{KhanSix}.
    \end{proof}

    \begin{lem}\label{lem:obqobu}
      Let $\sS$ be a scalloped derived stack.
      For any finite locally free sheaf $\sE$ on $\sS$, let $f : \P(\sE) \to \sS$ denote the associated projective bundle.
      Then $\D^*$ satisfies Atiyah duality for $f$.
    \end{lem}
    \begin{proof}
      Set $\sX := \P(\sE)$ to simplify the notation.
      By Nisnevich descent (\propref{prop:Nis desc}) and the fact that formation of $\varepsilon_f$ commutes with inverse image by representable étale morphisms, we may reduce to the case where $\sS$ is \fund.
      We may also assume $\sS$ is classical by derived invariance (\corref{cor:nil inv}).
      In this case the Pontryagin--Thom collapse map $\eta_f : \un_\sS \to f_\sharp (\un_\sX \vb{-\sL_f})$ in $\SH(\sS)$, constructed in \cite[\S 5.3]{HoyoisEquivariant}, induces by \propref{prop:univ} a morphism of the same form in $\D(\sS)$, which in turn induces by the smooth projection formula a natural transformation
      \[ \eta_f : \id_{\D(\sS)} \to f_\sharp f^* \vb{-\sL_f}. \]
      To show that $\varepsilon_f : f_* \to f_\sharp \vb{\sL_f}$ is invertible it will suffice to show that its left transpose $\varepsilon'_f : f^* f_\sharp \vb{-\sL_f} \to \id$ and $\eta_f$ are the counit and unit of an adjunction $(f^*, f_\sharp \vb{-\sL_f})$.
      The verification of the triangle identities reduces to showing that the composite
      \[
        f^* \xrightarrow{\eta_f} f^* f_\sharp \vb{-\sL_f} f^* \xrightarrow{\varepsilon'_{f}} f^*
      \]
      induces the identity when evaluated on the unit $\un_\sS$, as in the beginning of the proof of \cite[Thm.~5.22]{HoyoisEquivariant}.
      Again by \propref{prop:univ} it will suffice to show this in the case $\D^* = \SH^*$, which is done in \cite[Thm.~6.9]{HoyoisEquivariant}.
    \end{proof}

    \begin{lem}\label{lem:apsdfihq}
      Let $f : \sX \to \sY$ be a proper representable morphism of scalloped derived stacks.
      If $\D^*$ satisfies proper base change for $f$ (\thmref{thm:ipqnosdf}\itemref{item:ipqnosdf/bc}), then it also satisfies proper excision for $f$ (\thmref{thm:ipqnosdf}\itemref{item:ipqnosdf/excision}).
    \end{lem}
    \begin{proof}
      Let $i : \sZ \hook \sY$ be a closed immersion, with complementary open immersion $j : \sU \to \sX$, such that $f$ is an isomorphism over $\sU$.
      By \propref{prop:constrsep} it will suffice to show the square in question is homotopy cartesian after applying either $i^*$ or $j^*$.
      The $i^*$ case follows easily from the localization property and proper base change (\thmref{thm:ipqnosdf}\itemref{item:ipqnosdf/bc}).
      The $j^*$ case follows immediately from the smooth base change formula.
    \end{proof}

  \subsection{Proof of \thmref{thm:ipqnosdf}, projective case}
  \label{ssec:as08dfg}

    Let $f : \sX \to \sY$ be a projective morphism between scalloped derived stacks.
    
    Let us prove the proper base change formula \itemref{item:ipqnosdf/bc} and smooth-proper base change formula \itemref{item:ipqnosdf/smprop} for $f$.
    By Nisnevich descent and the smooth base change formula (\thmref{thm:funct SH}\itemref{item:funct SH/smooth}\itemref{item:funct SH/smooth/bc}), both claims are local on $\sY$ (for proper base change, choose a Nisnevich cover of $\sY$ and apply $*$-inverse image by the induced cover of $\sY'$).
    Therefore, we may assume that $\sY$ has the resolution property (e.g., that $\sY$ is \qfund, by \thmref{thm:scallop}\itemref{item:scallop/cover} and \propref{prop:2rgoy10}).

    Since $\sY$ has the resolution property, the projective morphism $f$ factors through a closed immersion into a projective bundle over $\sY$.
    For the closed immersion the claims are standard consequences of the localization property (\defnref{defn:0paguh107}\itemref{item:0paguh107/localization}), see e.g. \cite[\S 2.2]{KhanSix}.
    For the projective bundle the claims follow from Lemmas~\ref{lem:psdfij} and \ref{lem:obqobu}.
    The claims for $f$ then follow immediately.
    
    Atiyah duality \itemref{item:ipqnosdf/Atiyah} and proper excision \itemref{item:ipqnosdf/excision} then follow in view of Lemmas~\ref{lem:psdfij} and \ref{lem:apsdfihq}.

  \subsection{Proof of \thmref{thm:ipqnosdf}, general case}
    
    Combining Remarks~\ref{rem:0sftd07g} and \ref{rem:0as71uug1} with the projective case of \thmref{thm:ipqnosdf} (proven in \ssecref{ssec:as08dfg}), we have both proper cdh descent and proper excision \itemref{item:ipqnosdf/excision} for all proper representable morphisms.

    For compactness \itemref{item:ipqnosdf/compact}, proper base change \itemref{item:ipqnosdf/bc}, and smooth-proper base change \itemref{item:ipqnosdf/smprop}, we can use descent along a proper cdh cover of the following type to reduce to the case of a projective morphism.
    
    \begin{thm}\label{thm:Chow}
      Let $f : \sX \to \sY$ be a separated representable morphism of finite type between quasi-compact quasi-separated derived algebraic stacks.
      Assume that $\sY$ is scalloped.
      Then there exists a proper cdh cover $g : \widetilde{\sX} \to \sX$ such that the composite $f \circ g : \widetilde{\sX} \to \sY$ is quasi-projective.
    \end{thm}
    \begin{proof}
      The inclusion of the classical truncation is a proper cdh cover, so we may assume that $\sX$ and $\sY$ are classical.

      By the generalization of Chow's lemma proven in \cite{RydhFunctFlat}, we can find a projective morphism $\pi : \widetilde{\sX} \to \sX$ which is an isomorphism over a non-empty open $\sU \sub \sX$, such that $f \circ \pi : \widetilde{\sX} \to \sX \to \sY$ is quasi-projective.
      Then the family $\{\pi,i\}$, where $i : \sZ \to \sX$ is any closed immersion complementary to $j$, generates a proper cdh cover of $\sX$.
      Thus if $\sX$ is noetherian, then the claim follows by noetherian induction.

      In general, since $\sY$ is scalloped it is of global type in the sense of \cite[Def.~2.1]{RydhApprox} by \thmref{thm:scallop}\itemref{item:scallop/cover}, as is $\sX$ by \corref{cor:scallop more}.
      Thus we may apply noetherian approximation in the form of \cite[Thm.~D]{RydhApprox} to the morphism $\sX \to \sY$.
      The conclusion is that $f$ factors through an affine morphism $\sX \to \sX_0$ and a separated representable morphism $\sX_0 \to \sY$ which is of finite \emph{presentation}.
      Since proper cdh covers are stable under base change, and quasi-projective morphisms are stable under composition, we may replace $\sX$ by $\sX_0$ and thereby assume that $f$ is of finite presentation.

      By another application of \cite[Thm.~D]{RydhApprox} to the morphism $\sY \to \Spec(\bZ)$ we find an affine morphism $\sY \to \sY_0$ such that $\sY_0$ is of finite type over $\Spec(\bZ)$ (hence noetherian).
      Since $f$ is of finite presentation, we can moreover choose this approximation such that $f$ descends to a separated representable morphism of finite presentation $f_0 : \sX_0 \to \sY_0$.
      By the noetherian case we have a proper cdh cover $\widetilde{\sX_0} \to \sX_0$ such that $\widetilde{\sX_0} \to \sY_0$ is quasi-projective.
      Base changing from $\sX_0$ to $\sX$ now yields the desired proper cdh cover.
    \end{proof}

    Finally, Atiyah duality \itemref{item:ipqnosdf/Atiyah} again follows from \lemref{lem:psdfij}.

\section{The \texorpdfstring{$!$}{!}-operations}
\label{sec:shriek}

  \subsection{Statement}
  \label{ssec:shriek}

    Let $\D^*$ be a $(*,\sharp,\otimes)$-formalism on the \inftyCat of scalloped derived stacks.
    We assume $\D^*$ satisfies the Voevodsky conditions.

    \begin{thm}\label{thm:exc}
      For any representable morphism of finite type $f : \sX \to \sY$ in $\sC$, there exists a pair of adjoint functors
      \[ f_! : \D(\sX) \to \D(\sY), \quad f^! : \D(\sY) \to \D(\sX), \]
      and a natural transformation $\alpha_f : f_! \to f_*$, satisfying the following conditions:
      \begin{thmlist}
        \item\label{item:exc/open}
        There are canonical isomorphisms $f_! \simeq f_\sharp$ and $f^! \simeq f^*$ if $f$ is an open immersion.

        \item\label{item:exc/alpha}
        The natural transformation $\alpha_f : f_! \to f_*$ is invertible if $f$ is proper.

        \item\label{item:exc/bc}
        The functor $f_!$ satisfies the base change formula.
        That is, for any commutative square of scalloped derived stacks
        \[ \begin{tikzcd}
          \sX' \ar{r}{g}\ar{d}{u}
          & \sY' \ar{d}{v}
          \\
          \sX \ar{r}{f}
          & \sY
        \end{tikzcd} \]
        which is cartesian on classical truncations, the canonical morphisms of functors $\D(\sX) \to \D(\sY')$
        \begin{align*}
          \Ex^*_! : v^* f_! &\to g_! u^*,\\
          \Ex^!_* : u_* g^! &\to f^! v_*
        \end{align*}
        are invertible.

        \item\label{item:exc/proj}
        The functor $f_!$ satisfies the projection formula.
        That is, $f_! : \D(\sX) \to \D(\sY)$ is a morphism of $\D(\sY)$-module \inftyCats, where $\D(\sX)$ is regarded as a $\D(\sY)$-module via the symmetric monoidal functor $f^* : \D(\sY) \to \D(\sX)$.
        In particular, the canonical morphisms
        \begin{align*}
          \sG \otimes f_!(\sF) &\to f_!(f^*(\sG) \otimes \sF),\\
          \uHom(f^*(\sG), f^!(\sG')) &\to f^!(\uHom(\sG, \sG')),\\
          f_*\big(\uHom(\sF, f^!(\sG))\big) &\to \uHom(f_!(\sF), \sG)
        \end{align*}
        are invertible for all $\sF \in \D(\sX)$ and $\sG, \sG' \in \D(\sY)$.
      \end{thmlist}
      Moreover, the assignment $f \mapsto f_!$ (resp. $f \mapsto f^!$) extends to a functor $\D_!$ (resp. to a contravariant functor $\D^!$) from the \inftyCat of scalloped derived stacks to the \inftyCat of presentable \inftyCats and left-adjoint functors (resp. right-adjoint functors).
    \end{thm}

    \begin{rem}\label{rem:Thom !}
      Thom twists (\notatref{notat:aps-8hfp01h}) commute with each of the six operations.
      That is, for every morphism $f : \sX \to \sY$ and $\alpha \in \K(\sY)$, we have canonical isomorphisms of functors
      \begin{multline*}
        \qquad
        f^* \circ \vb{\alpha} \simeq \vb{f^*\alpha} \circ f^*,
        \quad
        f_* \circ \vb{f^*\alpha} \simeq \vb{\alpha} \circ f_*,\\
        f_! \circ \vb{f^*\alpha} \simeq \vb{\alpha} \circ f_!,
        \quad
        f^! \circ \vb{\alpha} \simeq \vb{f^*\alpha} \circ f^!.
      \end{multline*}
      Indeed, it suffices by \propref{prop:Nis desc} to check this on \fund stacks, in which case we can assume that $\alpha$ is the class of a locally free sheaf (since \fund stacks have the resolution property, see \propref{prop:2rgoy10}).
      In that case the claim is an easy consequence of various base change formulas.
      In view of these formulas, we will often abuse notation by writing e.g. $f_* \circ \vb{\alpha}$ instead of $\vb{\alpha} \circ f_*$ when $\alpha$ lives on the target.
    \end{rem}

  \subsection{Compactifications}

    The construction of the $!$-operations will be done, at least locally, by compactifying.

    \begin{defn}
      Let $f : \sX \to \sY$ be a representable morphism of scalloped derived stacks.
      We say that $f$ is \emph{compactifiable} if there exists a factorization
      \[ f : \sX \xrightarrow{j} \overline{\sX} \xrightarrow{g} \sY \]
      where $j$ is an open immersion and $g$ is proper representable.
      Note that if $f$ is compactifiable, then it is separated of finite type.
    \end{defn}

    \begin{rem}
      Compactifiability can be checked on classical truncations.
      Indeed, if $f : \sX \to \sY$ is a morphism and $\sX_\cl \xrightarrow{j_0} \overline{\sX}_\cl \xrightarrow{g_0} \sY_\cl$ is a compactification of $f_\cl$, then
      \[ \sX \to \overline{\sX}_\cl \fibcoprod_{\sX_\cl} \sX \to \sY \]
      is a compactification of $f$.
      See the proof of \cite[Pt.~II, Chap.~5, 2.1.6]{GaitsgoryRozenblyum}.
    \end{rem}

    \begin{exam}\label{exam:affcomp}
      Any affine morphism of finite type of derived stacks is compactifiable.
    \end{exam}

    \begin{exam}
      If $\sY$ is Deligne--Mumford, or at least has quasi-finite diagonal, then a representable morphism $f : \sX \to \sY$ is compactifiable if and only if it is separated and of finite type (see \cite[Thm.~B]{RydhCompact}).
    \end{exam}

    \begin{rem}\label{rem:scovolq}
      For any representable morphism $f : \sX \to \sY$, the \inftyCat of compactifications of $f$ is either empty or contractible.
      In the case of classical stacks this follows from \cite[Exp.~XVII, Prop.~3.2.6(ii)]{SGA4}.
      The derived case follows by the argument of \cite[Pt.~II, Chap.~5, 2.1.6]{GaitsgoryRozenblyum}.
    \end{rem}

    \begin{rem}\label{rem:repr loc compact}
      Let $f : \sX \to \sY$ be a representable morphism of finite type between scalloped derived stacks.
      Then there exists a commutative square
      \[ \begin{tikzcd}
        \sU \ar{r}{f_0}\ar[twoheadrightarrow]{d}{u}
        & \sV \ar[twoheadrightarrow]{d}{v}
        \\
        \sX \ar{r}{f}
        & \sY
      \end{tikzcd} \]
      where $u$ and $v$ are representable Nisnevich covers and $f_0$ is affine of finite type.
      (Indeed, choose $v$ such that $\sV$ is \fund by \thmref{thm:scallop}\itemref{item:scallop/cover2} and then use \thmref{thm:sumihiro}\itemref{item:sumihiro/cover} to choose a Nisnevich cover $\sU\twoheadrightarrow \sX \fibprodR_\sY \sV$ such that $f_0 : \sU\to\sV$ is affine.)
      In particular, every such $f$ is ``locally compactifiable'' (on the source and target, in the representable Nisnevich topology).
      Furthermore, note that if $\sX$ and $\sY$ have affine diagonal, then by \thmref{thm:scallop}\itemref{item:scallop/cover2} and the last part of \thmref{thm:sumihiro}\itemref{item:sumihiro/cover}, we can take $u$ and $v$ to be affine.
    \end{rem}

  \subsection{Proof of \thmref{thm:exc}}

    For compactifiable morphisms, the claims follow from \thmref{thm:ipqnosdf}, \remref{rem:scovolq}, and the general machinery of \cite{LiuZhengGluing} or \cite[Chap.~8, Thm.~6.1.5]{GaitsgoryRozenblyum} (cf. \cite[Thm.~9.4.8]{LiuZhengGluing}, \cite[Eqn.~(3.8)]{LiuZheng}, \cite[Chap.~5, Thm.~3.4.3]{GaitsgoryRozenblyum}, \cite[Thm.~2.34]{KhanSix}).
    Then one extends to all representable morphisms of finite type by \remref{rem:repr loc compact} and \cite[Thm.~4.1.8]{LiuZheng}.

  \subsection{Constructible separation}

    We can formulate an analogue of \propref{prop:constrsep} using $!$-inverse image:

    \begin{prop}\label{prop:constrsep!}
      For any constructible covering family $(j_\alpha : \sX_\alpha \to \sX)_\alpha$ of scalloped derived stacks, the inverse image functors
      \[ j_\alpha^! : \D(\sX) \to \D(\sX_\alpha) \]
      are jointly conservative.
    \end{prop}
    \begin{proof}
      We reduce to the case of a closed-open pair $(i,j)$ and use the exact triangle
      \[
        i_* i^!
        \xrightarrow{\mrm{counit}} \id
        \xrightarrow{\mrm{unit}} j_*j^*
      \]
      from \remref{rem:iadsfp} (recall that $j^* \simeq j^!$).
    \end{proof}

  \subsection{Purity}
  \label{ssec:purity}

    \begin{thm}\label{thm:purity}
      Let $f : \sX \to \sY$ be a representable smooth morphism of scalloped derived stacks.
      Assume that $f$ is compactifiable or that $\sX$ and $\sY$ have affine diagonal.
      Then there is a canonical isomorphism
      \[ \pur_f : f^! \simeq f^*\vb{\sL_{\sX/\sY}} \]
      of functors $\D(\sY) \to \D(\sX)$.
      Equivalently, by adjunction, there is a canonical isomorphism $f_! \simeq f_\sharp\vb{-\sL_{\sX/\sY}}$.
    \end{thm}

    \begin{rem}\label{rem:)&gh1}
      In \thmref{thm:purity}, and throughout the statements below, it is enough to assume that $f$ is Nisnevich-locally compactifiable in the sense that there exists a commutative square
      \[ \begin{tikzcd}
        \sU \ar{r}{f_0}\ar[twoheadrightarrow]{d}{u}
        & \sV \ar[twoheadrightarrow]{d}{v}
        \\
        \sX \ar{r}{f}
        & \sY
      \end{tikzcd} \]
      where $u$ and $v$ are \emph{compactifiable} Nisnevich covers and $f_0$ is compactifiable.
      This includes the case of $\sX$ and $\sY$ having affine diagonal (by \remref{rem:repr loc compact}), and this is all that will be necessary for our argument (see the proof of \corref{cor:relpur}).
    \end{rem}

    The following is a corollary, but will in fact feature in our proof of \thmref{thm:purity}.

    \begin{cor}\label{cor:spdfi}
      Suppose given a commutative square of scalloped derived stacks
      \[
        \begin{tikzcd}
          \sX' \ar{r}{g}\ar{d}{p}
          & \sY' \ar{d}{q}
          \\
          \sX \ar{r}{f}
          & \sY
        \end{tikzcd}
      \]
      which is cartesian on classical truncations, where $f$ is representable of finite type and $p$ and $q$ are representable and smooth.
      Consider the natural transformation
      \[
        \Ex^{*!} : p^*f^!
        \xrightarrow{\mrm{unit}} g^! g_! p^* f^!
        \simeq g^! q^* f_! f^!
        \xrightarrow{\mrm{counit}} g^! q^*,
      \]
      where the isomorphism comes from the base change formula (\thmref{thm:exc}\itemref{item:exc/bc}).
      Assume either that $f$ is compactifiable or that $\sX$, $\sY$, $\sX'$ and $\sY'$ have affine diagonal.
      Then $\Ex^{*!}$ is invertible.
    \end{cor}

    Before proving \corref{cor:spdfi}, we record a couple weaker variants of \thmref{thm:purity} which it implies.

    \begin{cor}\label{cor:etale f^!}
      Let $f : \sX \to \sY$ be a representable étale morphism of scalloped derived stacks.
      Assume that $f$ is compactifiable or that $\sX$ and $\sY$ have affine diagonal.
      Then there is a canonical isomorphism $f^! \simeq f^*$.
    \end{cor}
    \begin{proof}
      Since the diagonal $\Delta : \sX \to \sX \fibprod_\sY \sX$ is an open immersion, we have $\Delta^! \simeq \Delta^*$ (\thmref{thm:exc}\itemref{item:exc/open}).
      Applying \corref{cor:spdfi} to the homotopy cartesian square
      \[ \begin{tikzcd}
        \sX\fibprod_\sY \sX \ar{r}{\pr_2}\ar{d}{\pr_1}
        & \sX \ar{d}{f}
        \\
        \sX \ar{r}{f}
        & \sY,
      \end{tikzcd} \]
      we get the invertible natural transformation
      \[
        f^!
        = \Delta^*\pr_1^*f^!
        \xrightarrow{\Ex^{*!}} \Delta^*\pr_2^!f^*
        \simeq \Delta^!\pr_2^!f^*
        = f^*.
      \]
    \end{proof}

    \begin{cor}\label{cor:relpur}
      Let $f : \sX \to \sY$ be an unramified morphism of derived stacks which are representable and smooth over a scalloped derived stack $\sS$.
      Assume that $f$ is compactifiable or that $\sX$ and $\sY$ have affine diagonal.
      Then there is a canonical isomorphism
      \[ f^! q^* \simeq p^* \vb{\sL_{\sX/\sY}}, \]
      where $p : \sX \to \sS$ and $q : \sY \to \sS$ are the structural morphisms and $\sL_{\sX/\sY}$ is the relative cotangent complex of $f$.
    \end{cor}
    \begin{proof}
      If $f$ is a closed immersion, then this follows from \thmref{thm:relpur closed} by transposition.
      In general, there exists by \cite[Thm.~(1.2)]{RydhUnramified} a canonical global factorization of $f$ through a closed immersion $i$ and a representable étale morphism $g$:
      \[ \sX \xrightarrow{i} \sX' \xrightarrow{g} \sY. \]
      Let $p' = q \circ g : \sX' \to \sS$ denote the structural morphism.
      Combining the closed immersion case and \corref{cor:etale f^!}, we get a canonical isomorphism
      \[
        f^! q^*
        = i^! g^! q^*
        \simeq i^! g^* q^*
        = i^! (p')^*
        \simeq p^* \vb{-\sN_{\sX/\sX'}}
        \simeq p^* \vb{\sL_{\sX/\sY}},
      \]
      where the identification $-\sN_{\sX/\sX'} \simeq \sL_{\sX/\sY}$ in $\K(\sX)$ is induced by the isomorphism of perfect complexes $\sN_{\sX/\sX'}[1] = \sL_{\sX/\sX'} \simeq \sL_{\sX/\sY}$ induced by the étale morphism $g$.
    \end{proof}

    We now return to the proof of \corref{cor:spdfi}.

    \begin{proof}[Proof of \corref{cor:spdfi}]
      If $f$ is proper and representable, then $\Ex^{*!}$ is the right transpose of the smooth-proper base change isomorphism (hence is an isomorphism).
      If $f$ is an open immersion, then invertibility of $\Ex^{*!}$ is clear from \thmref{thm:exc}\itemref{item:exc/open}.
      This shows the claim when $f$ is compactifiable.
      We also get \corref{cor:etale f^!} for compactifiable étale morphisms.

      For the case of general $f$ (but where the stacks have affine diagonal), choose an affine morphism of finite type $f_0 : \sU \to \sV$ and a commutative square as in \remref{rem:repr loc compact}.
      Its base change along $q : \sY' \to \sY$ defines the cartesian cube:
      \[ \begin{tikzcd}[matrix yscale=0.8]
        &
        \sX'
        \ar{rr}{g}
        \ar{dd}[swap,near end]{p}
        & & \sY'
        \ar{dd}{q}
        \\
        \sU'
        \ar{ur}{u'}
        \ar[crossing over, near end]{rr}{g_0}
        \ar{dd}[swap]{p'}
        & & \sV' \ar{ur}[swap, near start]{v'}
        \\
        &
        \sX
        \ar[near start]{rr}{f}
        & & \sY
        \\
        \sU \ar[near end]{ur}{u}
        \ar{rr}{f_0}
        & & \sV
        \ar[crossing over, leftarrow, near end, swap]{uu}{q'}
        \ar{ur}{v}
      \end{tikzcd} \]
      By \propref{prop:Nis desc} it will suffice to show that the morphism
      \[ \Ex^{*!} : u'^*p^*f^! \to u'^*g^!q^* \]
      is invertible.
      Since $u$ and $v$ are affine, and hence compactifiable (\examref{exam:affcomp}), there are canonical isomorphisms $u^* \simeq u^!$ and $v^* \simeq v^!$ by above.
      Under these identifications the above morphism is identified with
      \[ \Ex^{*!} : p'^*f_0^!v^* \to g_0^!q'^*v^*, \]
      which is invertible by the compactifiable case applied to the front face (since $f_0$ is affine).
    \end{proof}

    Finally, now that \corref{cor:spdfi} (and hence \corref{cor:relpur}) is available to us, we are in position to prove \thmref{thm:purity}.

    \begin{proof}[Proof of \thmref{thm:purity}]
      Applying \corref{cor:spdfi} to the homotopy cartesian square
      \[ \begin{tikzcd}
        \sX\times_\sY \sX \ar{r}{\pr_1}\ar{d}{\pr_2}
        & \sX \ar{d}{f}
        \\
        \sX \ar{r}{f}
        & \sY
      \end{tikzcd} \]
      yields a canonical isomorphism $\Ex^{*!} : \pr_1^* f^! \simeq \pr_2^! f^*$.
      Since $f$ is representable and smooth, the diagonal $\Delta : \sX \to \sX\fibprod_\sY\sX$ is unramified with cotangent complex $\sL_\Delta \simeq \sL_{f}[1]$.
      Applying $\Delta^!$ and using the relative purity isomorphism $\Delta^! \pr_1^* \simeq \vb{-\sL_{f}}$ (\corref{cor:relpur}), we get the canonical isomorphism
      \[
        f^! \vb{-\sL_f}
        \simeq \Delta^! \pr_1^* f^!
        \xrightarrow{\Ex^{*!}} \Delta^! \pr_2^! f^*
        \simeq f^*.
      \]
      The purity isomorphism $\pur_f : f^! \simeq f^*\vb{\sL_f}$ is obtained by Thom twisting by $\sL_f$.
    \end{proof}

  \subsection{Descent}

    \begin{cor}\label{thm:D^! descent}
      Let $\sC$ be the \inftyCat of scalloped derived stacks and representable morphisms of finite type.
      Then the presheaf of \inftyCats $\D^!$ on $\sC$ satisfies cdh descent.
      Similarly, $\D_!$ satisfies cdh co-descent on $\sC$ when regarded as a co-presheaf with values in the \inftyCat of presentable \inftyCats and left-adjoint functors.
    \end{cor}
    \begin{proof}
      Follows from \thmref{thm:exc}, see \cite[Prop.~6.24]{HoyoisEquivariant} or \cite[Thm.~2.52]{KhanSix}.
    \end{proof}

\section{The Euler and Gysin transformations}
\label{sec:gys}

  We fix a $(*,\sharp,\otimes)$-formalism $\D^*$ satisfying the Voevodsky conditions, so that $\D^*$ extends to a formalism of six operations by \thmref{thm:exc}.

  \subsection{Euler transformation}

    \begin{constr}\label{constr:eul}
      Let $\sX$ be a scalloped derived stack.
      Given a finite locally free sheaf $\sE$ on $\sX$, let $p$ be the projection of its total space and $0$ the zero section.
      The \emph{Euler transformation} associated to $\sE$, denoted
      \[ \eul_\sE : \id \to \vb{\sE}, \]
      is the composite
      \[
        \id \simeq p_! p^! \xrightarrow{\mrm{unit}} p_! 0_! 0^* p^! \simeq \vb{\sE}
      \]
      where the first isomorphism is homotopy invariance and the second is purity (\thmref{thm:purity}).
    \end{constr}

    \begin{lem}\label{lem:eul=0}
      Let the notation be as in \constrref{constr:eul}.
      Suppose $\sE$ admits a surjective cosection $s : \sE \twoheadrightarrow \sO_\sX$.
      Then $s$ induces a null-homotopy $\eul_\sE \simeq 0$.
    \end{lem}
    \begin{proof}
      Write $\sV = \bV_\sX(\sE)$ for the total space.
      Note that $s$ corresponds to a nowhere zero section $s : \sX \to \sV$, i.e., it factors through the complement $\sV \setminus 0$ of the zero section.
      Let $q : \sV \setminus 0 \to \sX$ denote the projection.
      The localization triangle
      \[
        q_! q^!
        \xrightarrow{\mrm{counit}} p_! p^!
        \xrightarrow{\mrm{unit}} p_! 0_! 0^* p^!
      \]
      is isomorphic to
      \[
        q_! q^!
        \xrightarrow{\mrm{counit}} \id
        \xrightarrow{\eul_\sE} \vb{\sE}.
      \]
      Since $s : \sX \to \sV \setminus 0$ is a section of $q$, the counit $s_!s^! \to \id$ induces a natural transformation $\id \to q_! q^!$ splitting this triangle, and hence a null-homotopy of $\eul_\sE$.
    \end{proof}

  \subsection{Gysin transformation}

    The Gysin transformation of \cite[4.3.1]{DegliseJinKhan}\footnote{%
      called the \emph{purity} transformation in \emph{op. cit.}
    } extends immediately to our setting.
    We will need the following technical hypothesis on our morphisms.

    \begin{defn}\label{defn:smoothable}
      Let $f : \sX \to \sY$ be a morphism of scalloped derived stacks.
      We say that $f$ is \emph{representably smoothable} if it admits a global factorization
      \[ \sX \xrightarrow{i} \sA \xrightarrow{p} \sY \]
      where $p$ is smooth representable and $i$ is finite unramified.
    \end{defn}

    Recall that a representable morphism of derived stacks is \emph{quasi-smooth} if it is locally of finite presentation and the relative cotangent complex is of Tor-amplitude $[0, 1]$ (with homological grading), see e.g. \cite{KhanRydh}.

    \begin{thm}\label{thm:gys}
      Let $f : \sX \to \sY$ be a quasi-smooth, representably smoothable morphism of scalloped derived stacks with affine diagonal.
      Then there exists a natural transformation
      \[ \gys_{\sX/\sY} := \gys_f : f^*\vb{\sL_f} \to f^! \]
      of functors $\D(\sY) \to \D(\sX)$, satisfying the following properties.
      \begin{thmlist}
        \item
        If $f$ is smooth, then $\gys_f$ is the purity isomorphism of \thmref{thm:purity}.

        \item
        If $f$ is a closed immersion and $\sX$ and $\sY$ are smooth representable over a base $\sS$, then $\gys_f \ast q^*$ (where $\ast$ denotes horizontal composition) is canonically identified with the relative purity isomorphism (\corref{cor:relpur})
        \[ p^* \vb{\sL_{\sX/\sY}} \simeq f^! q^*, \]
        where $p : \sX \to \sS$ and $q : \sY \to \sS$ are the structural morphisms.
      \end{thmlist}
    \end{thm}

    \begin{rem}
      By adjunction, the Gysin transformation can be rewritten as a trace or cotrace:
      \begin{align*}
        \on{tr}_f& : f_!f^*\vb{\sL_{f}} \to \id,\\
        \on{cotr}_f& : \id \to f_*f^!\vb{-\sL_{f}}
      \end{align*}
      cf. \cite[Exp.~XVIII, \S 3.2]{SGA4}.
    \end{rem}

    \begin{rem}
      The Gysin transformation is functorial in $f$, up to homotopy, and also enjoys a base change property for homotopy cartesian squares.
      See \cite[Prop.~2.5.4]{DegliseJinKhan} for the precise formulation.
    \end{rem}

  \subsection{Proof of \thmref{thm:gys}}

    We only briefly sketch the construction, as it is the same as in the proofs of \cite[Thms.~4.1.4, 4.3.1]{DegliseJinKhan}.

    \begin{constr}
      Assume first that $f$ is finite unramified.
      Consider the deformation to the normal bundle \cite[Thm.~4.1.13]{KhanRydh}, which fits in the following diagram of homotopy cartesian squares:
      \[ \begin{tikzcd}
        \sX \ar{r}\ar{d}{0}
        & \sX \times \A^1 \ar{d}{\widehat{f}}
        & \sX \times \bG_m \ar{l}\ar{d}{f\times\id}
        \\
        N_{\sX/\sY} \ar{r}{\widehat{i}}\ar{d}{u}
        & D_{\sX/\sY} \ar[leftarrow]{r}{\widehat{j}}\ar{d}{t}
        & \sY \times \bG_m \ar[equals]{d}
        \\
        \sY \ar{r}
        & \sY \times \A^1 \ar[leftarrow]{r}
        & \sY \times \bG_m
      \end{tikzcd} \]
      where the left-hand side is the fibre over $0$ and the right-hand side is the complement.
      The morphism $u$ is the projection $\pi : N_{\sX/\sY} \to \sX$ followed by $f : \sX \to \sY$.

      The boundary map for the localization triangle associated to the closed-open pair $(\widehat{i},~\widehat{j})$ yields a natural transformation
      \[ \partial : q_*q^![-1] \to u_*u^!, \]
      where $q : \sY \times \bG_m \to \sY$ is the projection.
      Using the unit section $1 : \sY \to \sY \times \bG_m$ to split $q$, we get a canonical isomorphism $q_*q^! \simeq \id[1] \oplus \id(1)[2]$ and thus, by including as the first component, a natural transformation
      \begin{equation}\label{eq:sp}
        \sp_{\sX/\sY} := \sp_{f} : \id \to u_*u^!.
      \end{equation}

      Now by homotopy invariance and purity (\thmref{thm:purity}), we have a canonical isomorphism
      \[ \pi_*\pi^! \simeq \pi_*\pi^*\vb{-\sL_{\sX/\sY}} \simeq \vb{-\sL_{\sX/\sY}} \]
      using the canonical identification $\sL_{\pi} \simeq \sN_{\sX/\sY} \simeq -\sL_{\sX/\sY}$ in $\K(\sX)$.
      In particular, we get
      \[ u_*u^! \simeq f_*\pi_*\pi^!f^! \simeq f_*f^!\vb{-\sL_{\sX/\sY}}. \]
      The Gysin transformation for $f$, or rather the cotrace, is then
      \begin{equation*}
        \on{cotr}_{f} : \id \xrightarrow{\sp_f} u_*u^! \simeq f_*f^!\vb{-\sL_{\sX/\sY}}.
      \end{equation*}
      The compatibility with relative purity (\corref{cor:relpur}) is proven as in \cite[Lem.~3.2.15]{DegliseJinKhan}.

      Finally, for the general case, choose a factorization $f = p \circ i$ as in \defnref{defn:smoothable} and define $\on{cotr}_f$ to be the composite
      \[
        \id
        \xrightarrow{\on{cotr}_p} p_*p^!\vb{-\sL_p}
        \xrightarrow{\on{cotr}_i} i_*p_*p^!i^!\vb{-\sL_p-\sL_i}
        \simeq f_*f^!\vb{-\sL_f},
      \]
      where $\on{cotr}_p$ is the transpose of the purity isomorphism (\thmref{thm:purity}).
      One checks this is independent of the choice up to homotopy just as in \cite[Thm.~3.3.2]{DegliseJinKhan}.
    \end{constr}

  \subsection{Self-intersection formula}

    Let us also record the following formulation of the self-intersection formula, proven the same way as \cite[Cor.~4.2.3]{DegliseJinKhan}, which for a closed immersion relates the Gysin transformation with the Euler transformation of its conormal sheaf.

    \begin{prop}
      Let $i : \sX \to \sY$ be a quasi-smooth closed immersion of scalloped derived stacks.
      Then there is a commutative diagram
      \[ \begin{tikzcd}
        i^*\vb{-\sN_{\sX/\sY}} \ar{rr}{\eul_{\sN_{\sX/\sY}}}\ar[equals]{d}
        &
        & i^* \ar[equals]{d}
        \\
        i^*\vb{\sL_{\sX/\sY}} \ar{r}{\gys_{\sX/\sY}}
        & i^! \ar{r}{\mrm{Ex}^{*!}}
        & i^*.
      \end{tikzcd} \]
    \end{prop}

    Here $\Ex^{*!} : i^! \to i^*$ is the exchange transformation (\corref{cor:spdfi}) associated to the self-intersection square 
    \[ \begin{tikzcd}
      \sX \ar[equals]{r}\ar[equals]{d}
      & \sX \ar{d}{i}
      \\
      \sX \ar{r}{i}
      & \sY.
    \end{tikzcd} \]


\section{Cohomology and Borel--Moore homology theories}
\label{sec:coh}

  \subsection{Definitions}

    We fix a $(*,\sharp,\otimes)$-formalism $\D^*$ on scalloped derived stacks satisfying the Voevodsky conditions.
    Recall that by \thmref{thm:exc}, $\D^*$ extends to a formalism of six operations.

    \begin{defn}
      Let $f : \sX \to \sY$ be a representable morphism of finite type between scalloped derived stacks.
      The \emph{relative Borel--Moore homology spectrum} with coefficients in a sheaf $\sF \in \D(\sY)$ is the following mapping spectrum
      \[
        \CBM(\sX_{/\sY}, \sF)
        := \Maps_{\D(\sX)}(\un_\sX, f^!(\sF)).
      \]
    \end{defn}

    \begin{exam}
      Let $X$ be a derived algebraic space of finite type over an affine scheme $S$, with an action of a nice group scheme $G$ over $S$.
      The relative Borel--Moore homology spectrum with coefficients in $\sF \in \D(BG)$,
      \[ \CBM([X/G]_{/BG}, \sF), \]
      can be regarded as a (genuine) $G$-equivariant Borel--Moore homology spectrum for $X$ over $S$.
    \end{exam}

    \begin{defn}\label{defn:absolute}\leavevmode
      \begin{defnlist}
        \item
        An \emph{absolute object} of $\D$ (over scalloped derived stacks) is a collection $\sF = (\sF_\sX)_\sX$ of objects $\sF_\sX \in \D(\sX)$, for every scalloped derived stack $\sX$, together with a homotopy coherent system of isomorphisms
        \[ f^*(\sF_\sY) \simeq \sF_\sX \]
        in $\D(\sX)$ for every \emph{representable} morphism $f : \sX \to \sY$.
        More precisely, $\sF$ is a section of the cartesian fibration classified by the presheaf $\D^*$, which is cartesian over representable morphisms.
        
        \item
        An \emph{absolute twist} (over scalloped derived stacks) is similarly a collection $\alpha = (\alpha_\sX)_\alpha$ of points $\alpha_\sX \in \K(\sX)$, for every scalloped derived stack $\sX$, together with a homotopy coherent system of isomorphisms
        \[ f^*(\alpha_{\sY}) \simeq \alpha_{\sX} \]
        in $\K(\sY)$ for every representable morphism $f : \sX \to \sY$.
      \end{defnlist}
      We will usually omit the subscripts by abuse of notation.
      We also consider the variant of these definitions over some fixed base scalloped derived stack $\sS$, e.g. an absolute object $\sF$ over $\sS$ is as above except that $\sF_\sX$ is only given for $\sX$ which live over $\sS$.
    \end{defn}

    \begin{exam}
      Let $\sF$ be an absolute object of $\D$.
      For a scalloped derived algebraic stack $\sX$, the \emph{cohomology spectrum} with coefficients in $\sF$ is
      \[
        \Ccoh(\sX, \sF) := \Maps_{\D(\sX)}(\un_\sX, \sF).
      \]
    \end{exam}

    \begin{rem}
      For any absolute object $\sF$ and absolute twist $\alpha$, we adopt the convention
      \[
        \Ccoh(\sX, \sF)\vb{\alpha}
        := \Ccoh(\sX, \sF\vb{\alpha})
      \]
      and similarly
      \[
        \CBM(\sX_{/\sY}, \sF)\vb{\alpha}
        := \CBM(\sX_{/\sY}, \sF\vb{\alpha})
      \]
      for any representable morphism of finite type $\sX \to \sY$.
    \end{rem}

  \subsection{Operations}
  \label{ssec:coh/op}

    Just as in \cite[\S 2]{DegliseJinKhan} and \cite[\S 2.2]{KhanVirtual}, we immediately get the following structure on Borel--Moore homology from the formalism of six operations.
    This structure is subject to the same type of compatibilities as in Fulton and MacPherson's formalism of bivariant theories \cite[Sect.~2.2]{FultonMacPherson}, see also \cite[\S 2.3]{KhanVirtual}.

    \begin{notat}\label{notat:07agsfd}
      We fix a base $\sS$, a scalloped derived stack, and an absolute object $\sF$ over $\sS$ (\defnref{defn:absolute}).
      We denote by $\sC_{/\sS}$ for the \inftyCat of scalloped derived stacks $\sX$ over $\sS$, and $\sC^\rep_{/\sS}$ for the full subcategory spanned by $\sX \in \sC_{/\sS}$ for which $\sX \to \sS$ is representable of finite type.
    \end{notat}

    \sssec{Direct image}\label{sssec:BM/proper}
      If $f : \sX \to \sY$ is a proper morphism in $\sC^\rep_{/\sS}$, then there are direct image maps
      \begin{equation*}
        f_* : \CBM(\sX_{/\sS}, \sF) \to \CBM(\sY_{/\sS}, \sF).
      \end{equation*}

      If $\sX,\sY \in \sC_{/\sS}$ have affine diagonal and $f : \sX \to \sY$ is a quasi-smooth, proper, representably smoothable morphism, then there are also Gysin maps in cohomology
      \[
        f_! : \Ccoh(\sX, \sF) \to \Ccoh(\sY, \sF)\vb{-\sL_{\sX/\sY}}.
      \]

    \sssec{Inverse image}\label{sssec:BM/Gysin}
      If $f : \sX \to \sY$ is a representable morphism in $\sC_{/\sS}$, then there are inverse image maps
      \[ f^* : \Ccoh(\sY, \sF) \to \Ccoh(\sX, \sF). \]

      If $\sX,\sY \in \sC^\rep_{/\sS}$ have affine diagonal and $f : \sX \to \sY$ is a quasi-smooth, representably smoothable morphism, then there are also Gysin maps in Borel--Moore homology
      \begin{equation*}
        f^! : \CBM(\sY_{/\sS}, \sF) \to \CBM(\sX_{/\sS}, \sF)\vb{-\sL_{\sX/\sY}}.
      \end{equation*}

    \sssec{Change of base}\label{sssec:BM/base change}
      For any commutative square of scalloped derived stacks
      \begin{equation*}
        \begin{tikzcd}
          \sY \ar{r}\ar{d}\ar[phantom]{rd}{\scriptstyle\Delta}
            & \sT \ar{d}{f}
          \\
          \sX \ar{r}
            & \sS
        \end{tikzcd}
      \end{equation*}
      which is cartesian on classical truncations, where $\sX \to \sS$ and $\sY \to \sT$ are representable of finite type, there are maps
      \begin{equation*}
        f^*_\Delta : \CBM(\sX_{/\sS}, \sF) \to \CBM(\sY_{/\sT}, \sF).
      \end{equation*}

    \sssec{Euler class}
      Assume that $\sF$ is unital, i.e., it admits a unit map $\eta : \un \to \sF$.
      For any finite locally free sheaf $\sE$ on $\sX \in \sC_{/\sS}$, there is an Euler class
      \begin{equation*}
        e(\sE) \in \Ccoh(\sX, \sF)\vb{\sE}.
      \end{equation*}

    \sssec{Composition product}\label{sssec:BM/composition product}
      Assume that $\sF$ is multiplicative, i.e., it admits a multiplication map $\mu : \sF \otimes \sF \to \sF$.
      Given representable of finite type morphisms $\sX \to \sT$ and $\sT \to \sS$ between scalloped derived stacks, there is a pairing
      \begin{equation*}
        \circ
        : \CBM(\sX_{/\sT}, \sF) \otimes \CBM(\sT_{/\sS}, \sF)
        \to \CBM(\sX_{/\sS}, \sF).
      \end{equation*}

    Special cases of the composition product are cap and cup products:

    \sssec{Cap product}\label{sssec:BM/cap product}
      Given $\sX \in \sC^\rep_{/\sS}$, there is a pairing
      \begin{equation}
        \cap
        : \Ccoh(\sX, \sF) \otimes \CBM(\sX_{/\sS}, \sF)
        \to \CBM(\sX_{/\sS}, \sF).
      \end{equation}

    \sssec{Cup product}\label{sssec:BM/cup product}
      Given $\sX\in\sC_{/\sS}$, there is a pairing
      \begin{equation}
        \cup
        : \Ccoh(\sX, \sF) \otimes \Ccoh(\sX, \sF)
        \to \Ccoh(\sX, \sF).
      \end{equation}

    \begin{rem}\label{rem:graded}
      For every $\sX \in \sC_{/\sS}$ and $\sF \in \D(\sS)$, the twisted cohomology spectra of $\sX$ can be assembled into a $\K(\sY)$-graded spectrum
      \[
        \Ccoh(\sX, \sF)\vb{\ast}
        :=  \bigoplus_{\alpha \in \K(\sY)} \Ccoh(\sX, \sF)\vb{\alpha},
      \]
      with graded ring structure coming from the cup product.
      For every morphism $f : \sX \to \sY$ in $\sC^\rep_{/\sS}$, the Borel--Moore homology spectra can be assembled into a $\K(\sY)$-graded spectrum
      \[
        \CBM(\sX_{/\sY}, \sF)\vb{\ast}
        :=  \bigoplus_{\alpha \in \K(\sY)} \CBM(\sX_{/\sY}, \sF)\vb{\alpha},
      \]
      which becomes a graded module over $\Ccoh(\sX, \sF)\vb{\ast}$ via cap product.
      We can also collapse these into $\bZ$-gradings, where the homogeneous components of degree $r \in \bZ$ are
      \[
        \bigoplus_{\rk(\alpha)=r} \Ccoh(\sX, \sF)\vb{\alpha},
        \quad
        \bigoplus_{\rk(\alpha)=r} \CBM(\sX_{/\sY}, \sF)\vb{\alpha},
      \]
      respectively.
    \end{rem}

  \subsection{Properties}

    Let the notation be as in \notatref{notat:07agsfd}.
    The following properties follow immediately from the results of \secref{sec:six}, just as in \cite{DegliseJinKhan}.

    \begin{prop}[Localization]
      Given a complementary closed-open pair
      \[ \begin{tikzcd}
        \sZ \ar[hookrightarrow]{r}{i}
        & \sX \ar[hookleftarrow]{r}{j}
        & \sU
      \end{tikzcd} \]
      in $\sC^\rep_{/\sS}$, there is an exact triangle
      \[
        \CBM(\sZ_{/\sS}, \sF)
        \xrightarrow{i_*} \CBM(\sX_{/\sS}, \sF)
        \xrightarrow{j^!} \CBM(\sU_{/\sS}, \sF).
      \]
    \end{prop}

    \begin{prop}[Derived invariance]\label{prop:BM/nil-invariance}
      For every $\sX \in \sC^\rep_{/\sS}$, we have:
      \begin{thmlist}
        \item
        Change of base along the inclusion of the classical truncation $\sS_\cl \to \sS$ induces isomorphisms
        \begin{equation*}
          \CBM(\sX_{/\sS}, \sF) \to \CBM({\sX\fibprodR_\sS \sS_\cl}_{/\sS_\cl}, \sF).
        \end{equation*}
        
        \item
        Direct image along the inclusion of the classical truncation $i_\sX : \sX_\cl \to \sX$ induces an isomorphism
        \begin{equation*}
          i_{\sX,*} : \CBM({\sX_\cl}_{/\sS}, \sF) \to \CBM(\sX_{/\sS}, \sF).
        \end{equation*}
      \end{thmlist}
      Moreover, both statements also hold with $\sX_\cl$ replaced by the reduction $\sX_{\cl,\red}$.
    \end{prop}

    \begin{prop}[Thom isomorphism]\label{prop:BM/homotopy}
      Let $\sX \in \sC^\rep_{/\sS}$ and $\sE$ a finite locally free sheaf on $\sX$ with total space $\pi : \sV \to \sX$.
      Then the Gysin map
      \begin{equation*}
        \pi^! : \CBM(\sX_{/\sS}, \sF) \to \CBM(\sV_{/\sS}, \sF)\vb{-\sE},
      \end{equation*}
      is an isomorphism.
    \end{prop}

  \subsection{Fundamental classes and Poincaré duality}

    Let the notation be as in \ref{notat:07agsfd}, and assume that $\sF$ is unital.

    \begin{defn}\label{defn:fund smooth}
      Let $\sX \in \sC^\rep_{/\sS}$ such that $\sX$ and $\sS$ have affine diagonal.
      If $\sX$ is smooth over $\sS$, or more generally quasi-smooth and representably smoothable, then there is a relative \emph{fundamental class}
      \begin{equation*}
        [\sX/\sS] \in \CBM(\sX_{/\sS}, \sF)\vb{-\sL_{\sX/\sS}}
      \end{equation*}
      defined as the image of the unit by the Gysin map
      \[
        f^! :
        \CBM(\sS_{/\sS}, \sF) \to
        \CBM(\sX_{/\sS}, \sF)\vb{-\sL_{\sX/\sS}}.
      \]
    \end{defn}

    Since the Gysin transformation is invertible for smooth morphisms (see Theorems~\ref{thm:gys} and \ref{thm:purity}), we have:

    \begin{prop}[Poincaré duality]\label{prop:poincare}
      Let $\sX \in \sC^\rep_{/\sS}$ such that $\sX$ and $\sS$ have affine diagonal.
      If $\sX$ is smooth over $\sS$, then cap product with the fundamental class $[\sX/\sS]$ induces isomorphisms
      \begin{equation*}
        \Ccoh(\sX, \sF)
        \xrightarrow{\cap [\sX/\sS]} \CBM(\sX/\sS, \sF)\vb{-\sL_{\sX/\sS}}.
      \end{equation*}
    \end{prop}

\section{Examples}
\label{sec:ex}

  \ssec{Homotopy invariant K-theory}
  \label{ssec:KH}

    \begin{notat}
      Given a scalloped derived stack $\sX$, we let $\KB(\sX)$ denote the Bass--Thomason--Trobaugh K-theory spectrum of the stable \inftyCat of perfect complexes on $\sX$ (see \cite[Def.~2.6]{KhanKstack}, \cite[Sect.~4]{CisinskiKhan}).
      By construction, its infinite loop \ani $\Omega^\infty (\KB(\sX))$ is the K-theory \ani $\K(\sX)$.
    \end{notat}

    We have the following extension of the celebrated result of Thomason--Trobaugh \cite{ThomasonTrobaugh}:

    \begin{thm}\label{thm:Thomason}
      The assignment $\sX \mapsto \KB(\sX)$ determines a Nisnevich sheaf of spectra on the site of scalloped derived stacks.
    \end{thm}
    \begin{proof}
      This follows from \thmref{thm:perfect} and \cite[Thm.~2.13 and Rem.~2.15]{KhanKstack}.
    \end{proof}

    \begin{constr}\label{constr:07g301}
      Let $\sX$ be a scalloped derived stack.
      Restricting the presheaf $\sX' \mapsto \KB(\sX')$ to the site $\Sm_{/\sX}$ (notation as in \ssecref{ssec:H/constr}) and applying the (exact) $\A^1$-localization functor, we get a motivic $\Sm$-fibred $S^1$-spectrum $\KH_\sX$ (with $\Einfty$-ring structure).
      The homotopy invariant K-theory spectrum $\KH(\sX)$ is given by its global sections:
      \[
        \KH(\sX)
        = \RGamma(\sX, \KH_\sX)
        \simeq \colim_{[n]\in\bDelta^\op} \K(\sX \times \A^n).
      \]
      This definition agrees with \cite[\S 4C]{HoyoisKrishna} and \cite[\S 5A]{KrishnaRavi} in case $\sX$ is classical and with \cite[5.4.1]{KhanKblow} in case $\sX$ is a derived algebraic space.
      If $\sX$ is regular, then $\KH(\sX) \simeq \KB(\sX) \simeq \K(\sX) \simeq \G(\sX)$, where $\K(\sX)$ is the (connective) K-theory spectrum of perfect complexes on $\sX$ and $\G(\sX)$ is the (connective) K-theory spectrum of coherent sheaves on $\sX$ (see \cite[Thm.~3.5]{KhanKstack}).
    \end{constr}

    \begin{rem}\label{rem:apysbqqef}
      The motivic $S^1$-spectrum $\KH_\sX$ is stable under representable $*$-inverse image.
      Indeed, the proof over classical stacks in \cite[Prop.~4.6]{HoyoisKH} generalizes in view of \cite[Prop.~A.2.5]{BKRSMilnor}.
    \end{rem}

    \begin{rem}\label{rem:0afsgdh013}
      Combining \remref{rem:apysbqqef} with \thmref{thm:derinv}, we deduce that for every scalloped derived stack $\sX$, the canonical map
      \[ i^* : \KH(\sX) \to \KH(\sX_\cl), \]
      where $i$ is the inclusion of the classical truncation, is invertible.
      This gives a proof of \corref{cor:intro/KH} which is independent of our stable results such as proper base change (\thmref{thm:ipqnosdf}).
    \end{rem}

    \begin{rem}\label{rem:KH cdh}
      Using the cdh descent criterion of \cite[Thm.~E, Rem.~5.11(c)]{KhanKblow}, we can give a direct proof of \corref{cor:intro/KH cdh} by following \cite[5.3.4]{KhanKblow}.
      The new input in our setting is \remref{rem:0afsgdh013} and the localization theorem for $\MotSpc^*$ (\thmref{thm:loc unstable}), which together imply closed descent (cf. \cite[Ex.~5.9]{KhanKblow}).
    \end{rem}

    We have the following stable representability result:

    \begin{thm}\label{thm:KGL}
      For every scalloped derived stack $\sX$, there exists a canonical motivic $\sE_\infty$-ring spectrum $\KGL_\sX \in \SH(\sX)$ satisfying the following properties:
      \begin{thmlist}
        \item\label{item:0h013u0g}
        For every $\sX'\in \Sm_{/\sX}$, there are functorial isomorphisms of spectra
        \[
          \KH(\sX') \simeq \Ccoh(\sX', \KGL_\sX).
        \]

        \item\label{item:ad07sfg1}
        For every finite locally free sheaf $\sE$ on $\sX$, there is a canonical Bott periodicity isomorphism
        \[ \KGL_\sX\vb{\sE} \simeq \KGL_\sX \]
        in $\SH(\sX)$.

        \item\label{item:0asfh08}
        For any representable morphism $f : \sX \to \sY$, there is a canonical isomorphism
        \[
          f^*(\KGL_\sY) \simeq \KGL_{\sX}
        \]
        in $\SH(\sX)$.
        In fact, the collection $(\KGL_\sX)_\sX$ forms an absolute motivic spectrum in the sense of \defnref{defn:absolute}.
      \end{thmlist}
    \end{thm}
    \begin{proof}
      We follow the proof of \cite[Thm.~1.7]{HoyoisKH}, which proves the result in the case of certain (classical) global quotient stacks, and ``N-quasi-projective'' morphisms between them.

      If $\sX$ is \fund, it follows from the description of $\SH(\sX)$ in \remref{rem:invert quot}, \cite[Prop.~3.2]{HoyoisKH}, and the projective bundle formula, that there is a unique Bott-periodic delooping $\KGL_\sX \in \SH(\sX)$ of $\KH_\sX \in \MotSpc(\sX)$.
      In particular, \itemref{item:0h013u0g} and \itemref{item:ad07sfg1} hold by construction and \itemref{item:0asfh08} follows from \remref{rem:apysbqqef}.
      See the discussion around Def.~5.1 in \cite{HoyoisKH}.

      For general $\sX$, the claim now follows from \remref{rem:apysbqqef} and Nisnevich descent (cf. \lemref{lem:9sydagy}).
      More precisely, there exists a unique motivic $\sE_\infty$-ring spectrum $\KGL_\sX \in \SH(\sX)$ with a homotopy coherent system of isomorphisms $$u^*(\KGL_\sX) \simeq \KGL_\sU$$ for every \fund derived stack $\sU$ and every Nisnevich covering $u : \sU \twoheadrightarrow \sX$.
      See the comments on Thm.~1.7 in \cite[p.~16]{HoyoisKH}.
    \end{proof}

    \begin{rem}\label{rem:KQ}
      One can similarly construct a motivic spectrum $\mrm{KQ}_\sX \in \SH(\sX)$ representing hermitian K-theory, see \cite{PaninWalter} or \cite[\S 6]{HJNYHermitian}, at least assuming that $2$ is invertible on $\sX$ (although see \cite[Rem.~6.3]{HJNYHermitian}).
    \end{rem}

  \ssec{Algebraic cobordism}
  \label{ssec:coh/MGL}

    Following Voevodsky \cite[\S 6.3]{VoevodskyICM}, we can use our formalism to introduce a theory of algebraic cobordism for stacks.

    The following definition generalizes the one in \cite[\S 16]{BachmannHoyoisNorms} in the case of schemes.

    \begin{constr}\label{constr:07sg113}
      Given a scalloped derived stack $\sX$, we define $\MGL_\sX \in \SH(\sX)$ as the colimit
      \[
        \MGL_\sX = \colim_{(\sU, \alpha)} f_\sharp (\un_\sU\vb{\alpha})
      \]
      over the \inftyCat\footnote{
        i.e., the ``total space'' of the cartesian fibration associated to the presheaf sending $\sU \in \Sm_{/\sX}$ to the virtual rank $0$ part of $\K(\sU)$
      } of pairs $(\sU,\alpha)$ with $f : \sU \to \sX$ smooth representable and $\alpha \in \K(\sU)$ a K-theory class of virtual rank $0$.
      By construction, $\MGL_\sX$ admits a canonical (homotopy coherent) \emph{orientation} in the sense that there is a homotopy coherent system of Thom isomorphisms
      \[ \MGL_\sX\vb{\sE} \simeq \MGL_\sX(r)[2r] \]
      for every locally free sheaf $\sE$ of rank $r$ on $\sX$.
      See \cite[Prop.~16.28, Ex.~16.30]{BachmannHoyoisNorms}.
    \end{constr}

    \begin{prop}\label{prop:dogo011}
      For every representable morphism $f : \sX \to \sY$ of scalloped derived stacks, there is a canonical isomorphism
      $$f^*(\MGL_\sY) \simeq \MGL_{\sX}$$
      in $\SH(\sX)$.
      In fact, the collection $(\MGL_\sX)_\sX$ forms an absolute motivic spectrum in the sense of \defnref{defn:absolute}.
    \end{prop}
    \begin{proof}
      Let $\K^\circ_\sX$ denote the presheaf on $\Sm_{/\sX}$ sending $\sX'$ to the rank $0$ component of the $K$-theory \ani $\K(\sX')$.
      Then the canonical morphism
      \[ f^*(\K^\circ_{\sY}) \to \K^\circ_{\sX} \]
      is a Nisnevich-local equivalence by \cite[Cor.~2.9]{HoyoisKH} (cf. the discussion near the end of \cite[\S 5]{HoyoisKH}).
      Then the claim follows by construction of $\MGL$.
      Absoluteness, which asserts homotopy coherence of these isomorphisms, is a straightforward exercise with $\infty$-categorical fibrations, see e.g. \cite[Thm.~16.19]{BachmannHoyoisNorms}.
    \end{proof}

    \begin{rem}\label{rem:apfsdin}
      It follows from \cite[Cor.~2.10]{HoyoisKH} that, when $\sX$ is a quotient stack, $\MGL_\sX$ can be described in terms of tautological bundles over ``infinite Grassmannians'', similarly to Voevodsky's original definition in the case of schemes \cite[\S 6.3]{VoevodskyICM}.
      Compare \cite[Thm.~16.13]{BachmannHoyoisNorms}.
    \end{rem}

    \begin{rem}\label{rem:MSL}
      Following \cite[Ex.~16.22]{BachmannHoyoisNorms} one can similarly construct a motivic spectrum $\mrm{MSL}_\sX \in \SH(\sX)$ representing special linear algebraic cobordism.
    \end{rem}

  \ssec{Motivic cohomology}
  \label{ssec:coh/Z}

    We construct a motivic cohomology spectrum for scalloped derived stacks following the framed description given in \cite[Thm.~21]{HoyoisFramedLoc}.
    We begin with the following natural generalization of \cite[Def.~2.3.4]{EHKSY}.

    \begin{defn}\label{defn:07g1ufdl}
      Let $\sX$ be a scalloped derived stack and $\sX', \sX'' \in \Sm_{/\sX}$.
      A \emph{framed correspondence} from $\sX'$ to $\sX''$ is a diagram
      \[ \begin{tikzcd}
        & \sZ\ar[swap]{ld}{f}\ar{rd}{g}
        \\
        \sX' & & \sX''
      \end{tikzcd} \]
      where $f$ is a finite quasi-smooth morphism and $g$ is representable, together with an isomorphism $\sL_{\sZ/\sX'} \simeq 0$ in the \inftyGrpd $\K(\sZ)$.
      (Compare \cite[Def.~2.3.4]{EHKSY}.)
    \end{defn}

    \begin{constr}
      Given a scalloped derived stack $\sX$, we let $\Sm^\fr_{/\sX}$ denote the \inftyCat whose objects are those of $\Sm_{/\sX}$ and morphisms are framed correspondences (defined as in \cite[\S 4]{EHKSY}).
      The \inftyCats $\MotSpc_\fr(\sX)$ and $\SH_\fr(\sX)$ of framed motivic spectra over $\sX$ are defined by repeating the constructions of $\MotSpc(\sX)$ and $\SH(\sX)$ over $\Sm^\fr_{/\sX}$ (the conditions of Nisnevich descent and $\A^1$-invariance being imposed on the restrictions to $\Sm_{/\sX}$).
      As $\sX$ varies, this defines a $(*,\sharp,\otimes)$-formalism which satisfies homotopy invariance and Thom stability by construction, and should also satisfy localization (the proof in the case of schemes \cite[Thm.~8]{HoyoisFramedLoc} likely generalizes).
    \end{constr}

    \begin{rem}
      A framed correspondence as above acts on cohomology with coefficients in $\sF \in \D(\sX)$, for any $(*,\sharp,\otimes)$-formalism $\D$:
      \[
        \Ccoh(\sX'', \sF)
        \xrightarrow{g^*} \Ccoh(\sZ, \sF)
        \simeq \Ccoh(\sZ, \sF)\vb{\sL_f}
        \xrightarrow{f_!} \Ccoh(\sX', \sF).
      \]
      This observation implies that the canonical morphism $R : \SH^* \to \D^*$ (\propref{prop:univ}) admits a canonical factorization through $\SH_\fr^*$.\footnote{%
        This argument requires a \emph{homotopy coherent} action of framed correspondences, which we do not construct here.
      }
      On the restriction to (derived) schemes this factorization is \emph{unique}, i.e., the morphism $\SH^* \to \SH_\fr^*$ (``free transfers'') is invertible (see \cite[Thm.~18]{HoyoisFramedLoc}).
      This should generalize to scalloped stacks, but the necessary analysis of the geometry of framed correspondences over stacks will not be undertaken here.
    \end{rem}

    \begin{constr}\label{constr:afsdo0by}
      Consider the constant sheaf $\bZ_\sX$ on $\Sm_{/\sX}$, with its canonical framed transfers (see \cite[\S 4]{HoyoisFramedLoc}).
      We may regard it as an object of the unstable category $\MotSpc_\fr(\sX)$, form the framed infinite suspension $\Sigma^\infty_{\fr}(\bZ_\sX) \in \SH_\fr(\sX)$, and forget transfers to get a motivic $\Einfty$-ring spectrum
      $$\bZ_\sX \in \SH(\sX)$$
      that we call the (integral) \emph{motivic cohomology spectrum} over $\sX$.
      In the same manner, we get an $A$-linear motivic cohomology spectrum $A_\sX \in \SH(\sX)$ for every abelian group $A$.
      The argument of \cite[Lem.~20]{HoyoisFramedLoc} should generalize to show that this construction is stable under representable $*$-inverse image.
    \end{constr}

    \begin{rem}
      Note that the \emph{definition} of $\bZ_\sX$ is unconditional on the above conjectures on framed correspondences over stacks.
    \end{rem}

    \begin{rem}
      If the description of $\MGL$ in \cite[Thm.~3.4.1]{EHKSY3} is extended to stacks, then we also get a canonical $\sE_\infty$-ring morphism
      \[ \MGL_\sX \to \bZ_\sX \]
      for every scalloped derived stack $\sX$.
    \end{rem}

    \begin{rem}\label{rem:Ztilde}
      One can similarly construct a motivic spectrum $\widetilde{\bZ}_\sX \in \SH(\sX)$ representing Milnor--Witt motivic cohomology, following the framed construction in \cite[Thm.~7.3]{HJNYHermitian}.
    \end{rem}

\section{Fixed point localization}
\label{sec:fix}

  In this section we prove \thmref{thm:intro/conc}.
  We fix the following notation.

  \begin{notat}\leavevmode
    \begin{defnlist}
      \item
      Let $S$ be a connected noetherian affine base scheme.
      Let $T = \bG_{m,S}^{\times l}$ be a split torus over $S$ of dimension $l \ge 0$.
      
      \item
      Given a motivic $\Einfty$-spectrum $\sF \in \SH(BT)$, we consider a certain localization\footnote{%
        in the sense of $\Einfty$-ring spectra, see e.g. \cite[\S 7.2.3]{LurieHA}
      } of the $\bZ$-graded cohomology ring spectrum (see \remref{rem:graded}),
      \[
        \Ccoh(BT, \sF)_{\mrm{loc}} := \sS^{-1}\Ccoh(BT, \sF)\vb{\ast}
      \]
      at a set $\sS$ of homogeneous elements of degree $1$.
      Namely, let $L = [\A^1_S/\bG_{m,S}]$ denote the tautological line bundle on $B\bG_{m,S}$, where $\bG_{m,S}$ acts on $\A^1_S$ by scaling with weight $1$, and let $\pr_i : BT \to B\bG_{m,S}$ denote the $i$th projection.
      Then $\sS$ is the multiplicative closure of the set of Euler classes
      $$ \pr_i^* e(L^{\otimes n}) \in \piz\Ccoh(BT, \sF)\vb{L^{\otimes n}} $$
      for $n\ge 1$ and $1\le i\le l$.

      \item
      Given an $\Ccoh(BT, \sF)\vb{\ast}$-module spectrum $M$, we also set
      \[ M_{\mrm{loc}} := M \otimes^\bL_{\Ccoh(BT, \sF)\vb{\ast}} \Ccoh(BT, \sF)_{\mrm{loc}} \]
      for the extension of scalars.
    \end{defnlist}
  \end{notat}

  \begin{thm}[Concentration]\label{thm:conc}
    Let $i : Z \to X$ be a closed immersion of $T$-equivariant derived algebraic spaces of finite type over $S$, such that $T$ acts without fixed points on the complement $X \setminus Z$.
    Then for every motivic spectrum $\sF \in \SH(BT)$, the $\Ccoh(BT, \sF)\vb{\ast}$-module map
    \[ i_* : \CBM([Z/T]_{/BT}, \sF)\vb{\ast} \to \CBM([X/T]_{/BT}, \sF)\vb{\ast} \]
    induces an isomorphism of $\Ccoh(BT, \sF)_{\mrm{loc}}$-modules
    \[ i_* : \CBM([Z/T]_{/BT}, \sF)_{\mrm{loc}} \simeq \CBM([X/T]_{/BT}, \sF)_{\mrm{loc}}. \]
  \end{thm}

  \begin{cor}\label{cor:conc}
    Let $X$ be a $T$-equivariant derived algebraic space, separated of finite type over $S$.
    If $i : X^T \to X$ is the inclusion of the locus of fixed points
    \[
      X^T := \Maps_{BT}(BT, X),
    \]
    then for every motivic spectrum $\sF \in \SH(BT)$, there is an isomorphism of $\Ccoh(BT, \sF)_{\mrm{loc}}$-modules
    \[ i_* : \CBM((X^T \times BT)_{/BT}, \sF)_{\mrm{loc}} \simeq \CBM([X/T]_{/BT}, \sF)_{\mrm{loc}}. \]
  \end{cor}

  \begin{rem}
    In the situation of \thmref{thm:intro/conc}, where $S$ is the spectrum of a field and $T=\bG_{m,S}$, the separation hypothesis in \corref{cor:conc} can be dropped in view of \cite[Prop.~1.2.2]{DrinfeldGm}.
  \end{rem}

  \begin{rem}
    \thmref{thm:conc} and \corref{cor:conc} hold more generally, with the same proof, for any coefficient $\sF \in \D(BT)$, where $\D^*$ is as in \ssecref{ssec:shriek}.
  \end{rem}

  \begin{proof}[Proof of \thmref{thm:conc}]
    We may assume that $X$ is classical and reduced by \propref{prop:BM/nil-invariance}.
    By Nisnevich descent and \thmref{thm:sumihiro}, we may also assume that $X$ is separated.
    By the localization triangle (\thmref{thm:funct SH}\ref{item:funct SH/closed}(d)), it will suffice to show that if $T$ acts without fixed points on the whole of $X$, then
    \[
      \CBM([X/T]_{/BT}, \sF)_{\mrm{loc}} \simeq 0.
    \]
    Since $T$ then acts without fixed points on every $T$-invariant proper closed subspace $Y \subsetneq X$ as well, it will suffice by noetherian induction (on the quotient stack $[X/T]$), and the localization triangle again, to show that this claim holds after replacing $X$ by some nonempty $T$-invariant open subspace.

    By Thomason's generic slice theorem (see \cite[Thm.~4.10, Rem.~4.11]{ThomasonComp}), there exists a nonempty open $U \sub X$, a proper diagonalizable subgroup $T' \subsetneq T$, and a $T'$-equivariant algebraic space $V$ such that $[U/T] \simeq [V/T']$.
    Therefore, the $\Ccoh(BT, \sF)\vb{\ast}$-module structure on $\CBM([U/T]_{/BT}, \sF)\vb{\ast}$ is obtained by restriction of scalars from the $\Ccoh(BT', \sF)\vb{\ast}$-module structure on $\CBM([V/T']_{/BT}, \sF)\vb{\ast}$.
    Then we have
    \begin{equation*}
      \CBM([U/T]_{/BT}, \sF)_{\mrm{loc}}
      \simeq \CBM([U/T]_{/BT}, \sF) \otimes_{\Ccoh(BT', \sF)\vb{\ast}} \Ccoh(BT', \sF)_{\mrm{loc}},
    \end{equation*}
    so it will suffice to show that $\Ccoh(BT', \sF)_{\mrm{loc}}$ vanishes.

    For simplicity, assume $T = \bG_{m,S}$ and $T' = \mu_{n,S}$ ($n\ge 1$); since $T' \subsetneq T$ is a proper inclusion and $S$ is connected, the argument readily generalizes using \cite[Exp.~VIII, 1.4, 3.2]{SGA3} (as in the proof of \cite[Prop.~1.2]{ThomasonLefschetz}).
    Note that the line bundle $L^{\otimes n}$ on $BT$, which corresponds to the $\bG_{m,S}$-equivariant line bundle $\A^1_S$ where $\bG_{m,S}$ acts with weight $n$, restricts to the \emph{trivial} line bundle on $\mu_{n,S}$.
    Therefore, by \lemref{lem:eul=0} its Euler class $e(L^{\otimes n})$ restricts to $0$ on $BT'$.
    Hence $0$ is a unit in $\Ccoh(BT', \sF)_{\mrm{loc}}$.
  \end{proof}

  \begin{rem}\label{rem:07sg0o1}
    One can show that the classes $t_{i,n} = \pr_i^* e(L^{\otimes n})$ are usually nonzero.
    Indeed, rationalization and étale localization gives a canonical map $\sF \to \sF_{\Q,\et}$.
    Using, say, the additive formal group law to orient $\sF_{\Q,\et}$, this determines a homomorphism of $\bZ$-graded ring spectra
    \[
      \Ccoh(BT, \sF)\vb{\ast}
      \to \Ccoh(BT, \sF_{\Q,\et})\vb{\ast},
    \]
    which in degree $r$ is the map
    \begin{equation*}
      \bigoplus_{\rk(\alpha)=r} \Ccoh(BT, \sF)\vb{\alpha}
      \to \bigoplus_{\rk(\alpha)=r} \Ccoh(BT, \sF_{\Q,\et})\vb{\alpha}
      \xrightarrow{\mrm{fold}} \Ccoh(BT, \sF_{\Q,\et})\vb{r}.
    \end{equation*}
    Under the Thom isomorphism
    \[
      \Ccoh(BT, \sF_{\Q,\et})\vb{L^{\otimes n}} \simeq \Ccoh(BT, \sF_{\Q,\et})\vb{1},
    \]
    the element $t_{i,n}$ maps to $\pr_i^* c_1(L^{\otimes n}) = n \cdot t_i$, where
    $$t_i = \pr_i^* c_1(L) \in \pi_0 \Ccoh(BT, \sF_{\Q,\et})\vb{1}$$
    is nonzero in the polynomial ring (of characteristic zero)
    \[ \pi_0 \Ccoh(BT, \sF_{\Q,\et})\vb{\ast} \simeq R [ t_1,\ldots,t_l ], \]
    see \cite[\S 4, Prop.~3.7]{MorelVoevodsky}, as long as $R = \pi_0 \Ccoh(S, \sF_{\Q,\et})\vb{\ast}$ is nonzero.
  \end{rem}

  \begin{rem}
    For unoriented examples such as Milnor--Witt motivic cohomology, hermitian K-theory, or special linear algebraic cobordism, \thmref{thm:conc} is probably not very satisfactory.
    At least for the lisse-extended theories (see \secref{sec:lim}), the Witt cohomology of $B\bG_m$ is trivial, so that Euler classes in these theories should have no ``Witt contribution''.
    We thank Marc Levine for explaining this to us.
  \end{rem}

\section{Lisse extensions}
\label{sec:lim}

  \ssec{Lisse-extended categories}
  \label{ssec:lim/lisse}

    We begin with a construction of \emph{lisse extensions} of categories of coefficients from algebraic spaces to stacks, and a proof of \thmref{thm:intro/Bor}.\footnote{%
      The idea to consider such a generalization originally arose in unpublished work of Marc Hoyois with the first author.
    }
    Throughout the section, we implicitly assume that all (derived) algebraic spaces are quasi-separated, and all (derived) Artin stacks have quasi-separated representable diagonal.

    Let $S$ be a derived algebraic space and $\D^*$ a $(*,\sharp,\otimes)$-formalism satisfying the Voevodsky conditions over derived algebraic spaces over $S$ (see \cite[\S 2]{KhanSix}).

    \begin{constr}\label{constr:DBor}
      Let $\sX$ be a derived Artin stack over $S$.
      We define
      \[ \D_\Bor(\sX) = \lim_{(T, t)} \D(T) \]
      where the limit is taken over the \inftyCat $\Lis_{\sX}$ of pairs $(T, t)$ where $T$ is a derived algebraic space and $t : T \to \sX$ is a smooth morphism.
      Note that $\D_\Bor(\sX) \simeq \D(\sX)$ if $\sX = X$ is a derived algebraic \emph{space}.
    \end{constr}

    \begin{exam}\label{exam:0--81h}
      When $\D^*$ has étale descent (on algebraic spaces), this construction was considered in \cite[App.~A]{KhanVirtual}.
    \end{exam}

    \begin{rem}\label{rem:Bor aff}
      The limit in \constrref{constr:DBor} can also be taken over the full subcategory $\Lis^{\mrm{aff}}_\sX$ of $\Lis_\sX$ spanned by $(T, t)$ with $T$ \emph{affine}.
      More precisely, the canonical functor
      \begin{equation}\label{eq:padsifhn}
        \lim_{(T, t) \in \Lis_\sX} \D(T) \to \lim_{(T, t) \in \Lis^\mrm{aff}_\sX} \D(T)
      \end{equation}
      is an equivalence.
      This follows from Nisnevich descent for $\D$ (over algebraic spaces).
      Indeed, we can write the source as
      \[ \lim_{(T, t) \in \Lis_\sX} \lim_{(S \text{ aff}, s : S \to T \text{ ét})} \D(S) ~\simeq~ \lim_{(T, t, S \text{ aff}, s : S \to T \text{ ét})} \D(S). \]
      The forgetful functor from the right-hand indexing category to $\Lis_\sX^{\mrm{aff}}$ (which remembers only $S$ and $S \to T \to \sX$) induces a functor to this category from the target of \eqref{eq:padsifhn}, which one checks is inverse to \eqref{eq:padsifhn}.
    \end{rem}

    \begin{rem}
      The discussion of \cite[App.~A]{KhanVirtual} can be adapted to show that on the lisse-extended categories, the following operations extend: $\otimes$ and $\uHom$, $f^*$ and $f_*$ for arbitrary morphisms, $f_!$ and $f^!$ for finite type morphisms, and $\vb{\alpha}$ for K-theory classes $\alpha$.
      We also have the base change formula, the isomorphism $\alpha_f : f_! \simeq f_*$ for $f$ proper representable, the purity isomorphism $\pur_f : f^! \simeq f^*\vb{\sL_f}$ for a smooth morphism, homotopy invariance for vector bundles, and the localization triangles for complementary closed/open pairs.
      The only nontrivial part is the exceptional functoriality, which we will not use here.
    \end{rem}

  \subsection{Cohomology}

    \begin{defn}
      For every derived Artin stack $\sX$ over $S$, the cohomology spectrum with coefficients in $\sF \in \D(\sX)$ is defined by the formula
      \[
        \Ccoh_\Bor(\sX, \sF)
        = \Maps_{\D_\Bor(\sX)}(\un, \sF).
      \]
      Given $\alpha\in\K(\sX)$ we will write
      \begin{equation}\label{eq:anfp1n0}
        \H^\alpha_\Bor(\sX, \sF) = \Hom_{\D_\Bor(\sX)}(\un, \sF\vb{\alpha}) = \pi_0 \big(\Ccoh_\Bor(\sX, \sF)\vb{\alpha}\big).
      \end{equation}
      We have inverse images along arbitrary morphisms and Gysin direct images along proper smooth representable morphisms.
    \end{defn}

    \begin{rem}
      For any derived Artin stack $\sX$, $\sF \in \D_\Bor(\sX)$, the cohomology spectrum is by construction the homotopy limit
      \[ 
        \Ccoh_\Bor(\sX, \sF)
        \simeq \lim_{(T,t)}\Ccoh_\Bor(T, \sF)
      \]
      over $(T,t) \in \Lis_\sX$.
    \end{rem}

  \subsection{The Borel construction}
  \label{ssec:Borel}

    The following result shows that, for quotient stacks, lisse-extended cohomology theories can be computed via Totaro's algebraic version of the Borel construction.

    Let $S$ be the spectrum of a perfect field $k$.
    (For non-perfect fields the result will also follow, up to inverting the characteristic, in view of \cite{ElmantoKhan}.)

    \begin{thm}\label{thm:Borel}
      Let $G$ be an fppf group scheme over $S$.
      Suppose given a tower
      \[ \sV_0 \hook \sV_1 \hook \sV_2 \hook \cdots \]
      of inclusions of vector bundles over $BG$, together with closed substacks $\sW_i \subseteq \sV_i$, such that for each $i$ we have:
      \begin{thmlist}
        \item
        The open complement $U_i = \sV_i \setminus \sW_i$ is representable (by an algebraic space).

        \item
        There is an inclusion $U_i \sub U_{i+1}$.

        \item
        Given any integer $n\ge 0$, there exists an $i\gg 0$ such that $\codim_{\sV_i}(\sW_i) > n$.
      \end{thmlist}
      Then for any motivic spectrum $\sF \in \SH(S)$, there is a canonical isomorphism
      \[
        \Ccoh_\Bor(BG, \sF)
        \simeq \lim_i \Ccoh(U_i, \sF).
      \]
      More generally, if $\sX = [X/G]$ is the quotient of a smooth algebraic space $X$ with $G$-action, we have
      \[
        \Ccoh_{\Bor}(\sX, \sF)
        \simeq \lim_i \Ccoh(\sX \fibprod_{BG} U_i, \sF).
      \]
    \end{thm}

    \begin{exam}\label{exam:071g20}
      If $G$ is a smooth embeddable group scheme over $S$, then there exists a choice of $(\sV_i, \sW_i)_i$ as in \thmref{thm:Borel} by \cite[Rem.~1.4]{Totaro}.
      For example, for the multiplicative group $\bG_m$ we get
      \[
        \Ccoh_\Bor(B\bG_m, \sF)
        \simeq \lim_i \Ccoh(\bP^i, \sF).
      \]
      More generally for the general linear group $\mrm{GL}_n$ we get
      \[
        \Ccoh_\Bor(B\mrm{GL}_n, \sF)
        \simeq \Ccoh(\mrm{Gr}_{n,\infty}, \sF).
      \]
      where $\mrm{Gr}_{n,\infty}$ is the Grassmannian ind-scheme of rank $n$ vector subspaces.
    \end{exam}

    \begin{rem}\label{rem:Milnor}
      In the situation of \thmref{thm:Borel}, the canonical maps
      \[
        \H^\alpha_\Bor(\sX, \sF)
        \simeq \piz \lim_i \Ccoh(\sX \fibprod_{BG} U_i, \sF)\vb{\alpha}
        \twoheadrightarrow \lim_i \H^\alpha(\sX \fibprod_{BG} U_i, \sF)
      \]
      are always surjective by the Milnor exact sequence.
      More generally for every $s\in\bZ$ we have surjections
      \[
        \pi_s \Ccoh_\Bor(\sX, \sF)\vb{\alpha}
        \twoheadrightarrow \lim_i \pi_s \Ccoh(\sX \fibprod_{BG} U_i, \sF)\vb{\alpha}.
      \]
      In the case of motivic cohomology, we will show (see \remref{rem:Milnor mot}) that these are in fact bijective.
    \end{rem}

    \begin{rem}
      See \cite{equilisse} for generalizations of \thmref{thm:Borel} to the case where $\sX$ is singular.
    \end{rem}

  \subsection{Proof of \thmref{thm:Borel} for motivic cohomology}
  \label{ssec:Borel mot}

    Let $\Lambda$ be a commutative ring in which the characteristic exponent of the field $k$ is invertible.
    In this subsection, we will give a proof of \thmref{thm:Borel} in the special case of $\Lambda$-linear motivic cohomology, i.e., where $\sF = \Lambda\vb{n}$ for any $n\in\bZ$.
    This will be independent of the proof of the general statement proven in the next subsection, but we decided to also include this argument due to its comparative simplicity.

    \begin{proof}
      Let $\pi_i : \sX \fibprod_{BG} \sV_i \to \sX$ denote the projections and $j_i : \sX \fibprod_{BG} U_i \to \sX \fibprod_{BG} \sV_i$ the inclusions.
      The inverse image maps
      \[
        \Ccoh_\Bor(\sX, \Lambda)\vb{n}
        \xrightarrow{\pi_i^*} \Ccoh_\Bor(\sX \fibprod_{BG} \sV_i, \Lambda)\vb{n}
        \xrightarrow{j_i^*} \Ccoh_\Bor(\sX \fibprod_{BG} U_i, \Lambda)\vb{n},
      \]
      where
      \[
        \Ccoh_\Bor(\sX \fibprod_{BG} U_i, \Lambda)\vb{n}
        \simeq \Ccoh(\sX \fibprod_{BG} U_i, \Lambda)\vb{n}
      \]
      since $U_i$ is an algebraic space, induce a canonical map
      \begin{equation*}
        \Ccoh_\Bor(\sX, \Lambda)\vb{n}
        \to \lim_i \Ccoh(\sX \fibprod_{BG} U_i, \Lambda)\vb{n}.
      \end{equation*}
      By homotopy invariance, $\pi_i^*$ is invertible for every $i$, so it will suffice to show that $j_i^*$ is invertible for $i\gg 0$.
      More precisely, we will show that this holds for all $i$ such that $\codim_{\sV_i}(\sW_i) > n$.
      
      By construction of the lisse-extended theory and cofinality, it is enough to prove the claim with $\sX$ replaced by $T$, for any $(T,t)\in\Lis^{\mrm{aff}}_\sX$.
      By the localization triangle and Poincaré duality, the fibre of the map $j_i^*$ is the Borel--Moore homology spectrum
      \[
        \CBM(T \fibprod_{BG} {\sW_i}_{/S}, \Lambda)\vb{-d+n},
      \]
      where $d = \dim (T \fibprod_{BG} \sV_i)$.
      But this spectrum vanishes as soon as $d-n > \dim(T \fibprod_{BG} \sW_i)$, i.e., whenever
      \[
        \codim_{T \fibprod_{BG} \sV_i}(T \fibprod_{BG} \sW_i) 
        = \codim_{\sV_i}(\sW_i)
      \]
      is strictly larger than $n$.
      (This follows from the comparison with the Bloch cycle complex; see \cite[Prop.~19.18]{MazzaWeibelVoevodsky} and \cite[Cor.~8.12]{CisinskiDegliseCdh}, and note that $T \fibprod_{BG} \sW_i$ is affine since $T$ is.)
    \end{proof}

    \begin{rem}\label{rem:Milnor mot}
      Our proof also shows that the canonical homomorphisms (\remref{rem:Milnor})
      \[
        \pi_s \Ccoh_\Bor(\sX, \Lambda)\vb{n}
        \twoheadrightarrow \lim_i \pi_s \Ccoh(\sX \fibprod_{BG} U_i, \Lambda)\vb{n}
      \]
      are bijective for all $n,s \in \bZ$.
      Indeed, these are the limits over $i$ of the restriction maps
      \[
        \pi_s \Ccoh_\Bor(\sX, \Lambda)\vb{n}
        \simeq \pi_s \Ccoh_\Bor(\sX \fibprod_{BG} \sV_i, \Lambda)\vb{n}
        \xrightarrow{j_i^*} \pi_s \Ccoh(\sX \fibprod_{BG} U_i, \Lambda)\vb{n}
      \]
      which we showed were invertible for $i\gg 0$.
      (The first isomorphism is homotopy invariance for the vector bundle $\sX\fibprod_{BG}\sV_i \to \sX$.)
    \end{rem}

  \subsection{Proof of \thmref{thm:Borel} in general}
  \label{ssec:Borel homotopy}

    We will deduce the general case of \thmref{thm:Borel} from a stronger comparison at the level of stable motivic homotopy types.

    \begin{notat}\label{notat:as0fg1}
      Let $\sX$ be a smooth Artin stack over an algebraic space $S$.
      Then we write
      \[ M_S^\Bor(\sX) := f_{\sharp}(\un_\sX) \in \SH(S), \]
      where $f_\sharp : \SH_\Bor(\sX) \to \SH_\Bor(S) \simeq \SH(S)$ is $\sharp$-direct image along the structural morphism $f : \sX \to S$.
      Note that if $\sX = X$ is an algebraic space, then $M_S(X) \simeq \Sigma^\infty_+ \L \h_S(X)$.
      In general it is computed as
      \[
        M_S^\Bor(\sX) \simeq \colim_{(T,t)\in\Lis_{\sX}} M_S^\Bor(T).
      \]
    \end{notat}

    The following can be regarded as a comparison of the lisse-extended motivic stable homotopy type of a quotient stack with its Morel--Voevodsky motivic stable homotopy type (see \cite[\S 4.2]{MorelVoevodsky} and \cite[\S 3]{KrishnaCompletion}).

    \begin{thm}\label{thm:Borel M}
      Let $G$ be an fppf group scheme over $S$.
      Let $(\sV_i)_i$, $(\sW_i)_i$ and $(U_i)_i$ be as in \thmref{thm:Borel}.
      Let $\sX = [X/G]$ be the quotient of a smooth algebraic space $X$ over $S$ with $G$-action.
      Then there is a canonical isomorphism
      \[
        M_S^\Bor(\sX)
        \simeq \colim_i M_S(\sX \fibprod_{BG} U_i)
      \]
      in $\SH(S)$.
    \end{thm}

    \begin{proof}
      As in the proof in \ssecref{ssec:Borel mot}, the morphism
      \[
        \colim_i M_S(\sX \fibprod_{BG} U_i)
        \to M_S^\Bor(\sX)
      \]
      is induced by the canonical morphisms
      \[
        M_S^\Bor(\sX \fibprod_{BG} U_i)
        \to M_S^\Bor(\sX \fibprod_{BG} \sV_i)
        \simeq M_S^\Bor(\sX)
      \]
      for every $i$.
      These can be described as the colimits over $(T,t)\in\Lis^{\mrm{aff}}_\sX$ over the analogous morphisms
      \[
        M_S(T \fibprod_{BG} U_i)
        \to M_S(T \fibprod_{BG} \sV_i)
        \simeq M_S(T).
      \]
      We claim that for every $(T,t)$, the colimit over $i$ of the cofibres
      \[
        K_i := M_S(T \fibprod_{BG} \sV_i)/M_S(T \fibprod_{BG} U_i)
      \]
      vanishes.
      In fact, we will show that each $K_i$ is $c_i$-connective for the homotopy t-structure on $\SH(S)$, where $c_i$ is the codimension of $\sW_i$ in $\sV_i$.
      This will imply (see \cite[Cor.~2.4]{HoyoisHopkinsMorel}) that for every $Y \in \Sm_{/S}$ and $r,s\in\bZ$,  we have
      \[
        \Hom_{\SH(S)}(M_S(Y)(r)[s], K_i) = 0
      \]
      for $i$ large enough that $c_i > s-r+\dim(Y)$.
      Since the objects $M_S(Y)(r)[s]$ form a set of compact generators of $\SH(S)$, it will follow that $\colim_i K_i = 0$ as claimed.
      
      Since $k$ is perfect, each scheme $T\fibprod_{BG}\sW_i$ can be stratified by smooth closed subschemes.
      By the localization triangle and relative purity, it follows that each $K_i$ is contained in the full subcategory of $\SH(S)$ generated under colimits and extensions by objects of the form $M_S(W)\vb{\sE}$ for $W \in \Sm_{/S}$ and $\sE$ a locally free sheaf on $W$ of rank $\ge c_i$.
      It will thus suffice to show that every such $M_S(W)\vb{\sE}$ is $c_i$-connective, since this property is preserved under colimits and extensions.
      But this holds by \cite[Lem.~3.1]{HoyoisHopkinsMorel}.
    \end{proof}

  \subsection{Equivariant Chow groups, cobordism and K-theory}
  \label{ssec:lim/equiv}

    Let $S$ be the spectrum of a perfect field $k$, $G$ be a smooth embeddable group scheme over $S$, and $X$ a smooth $G$-quasi-projective $S$-scheme.
    \thmref{thm:Borel} yields the following comparisons.

    \begin{exam}\label{exam:equiv Chow}
      The lisse-extended motivic cohomology of $[X/G]$
      \[
        \Ccoh_\Bor([X/G], \Lambda)\vb{n}
      \]
      is computed for every $n\in\bZ$ by $\Lambda$-linear Bloch cycle complexes of the Borel construction.
      Here $\Lambda$ is any commutative ring in which the characteristic exponent of $k$ is invertible.
      This follows from the comparison of the motivic complexes and Bloch cycle complexes of schemes (see \cite[Lect.~19]{MazzaWeibelVoevodsky}).
      By \remref{rem:Milnor mot} we moreover get for all $n,s\in\bZ$ canonical isomorphisms
      \[
        \pi_s \Ccoh_\Bor([X/G], \Lambda)\vb{n} \simeq \CH^G_{d-n}(X, s) \otimes \Lambda,
      \]
      if $X$ is of pure dimension $d$, where on the right-hand side are the $G$-equivariant higher Chow groups of $X$ as defined by Edidin--Graham \cite{EdidinGraham}.
      In particular,
      \[
        \H^n_\Bor([X/G], \Lambda) \simeq \CH^G_{d-n}(X) \otimes \Lambda
      \]
      where the left-hand side is defined as in \eqref{eq:anfp1n0}.
    \end{exam}

    \begin{exam}\label{exam:equiv cob}
      Applying \thmref{thm:Borel} to the algebraic cobordism spectrum $\MGL$, we find that the lisse-extended algebraic cobordism of $[X/G]$
      \[
        \Ccoh_\Bor([X/G], \MGL)\vb{n}
      \]
      can be computed via Voevodsky's algebraic cobordism \cite[\S 6.3]{VoevodskyICM} of the Borel construction.
      If $k$ is of characteristic zero, then it follows from \cite{LevineMGL,HoyoisHopkinsMorel} that there are surjections
      \begin{equation}\label{eq:07ugbp1j}
        \H^n_\Bor([X/G], \MGL)
        \twoheadrightarrow \lim_i \Omega^n([X/G] \fibprod_{BG} U_i)
      \end{equation}
      where the notation is as in \thmref{thm:Borel}.
      The right-hand side here has been considered in \cite{HellerMalagonLopez} and \cite[Thm.~6.1]{KrishnaCobordism}.
      We note that these theories are not known to satisfy several fundamental properties such as the right-exact localization sequence\footnote{\label{fn:1390gy1}%
        The contrary is claimed in \cite[Thm.~20]{HellerMalagonLopez}.
        However, as brought to our attention by M.~Levine (who attributes the observation to A.~Merkurjev) and H.~Park, the proof relies on the (false) assertion that the limit of a right-exact sequence of a projective system of abelian groups, the $\bR^1\lim$ terms each of which vanish, is still right-exact.
      } (see however \cite[Cor.~6.2]{KrishnaCobordism} for the special case of sections of projective morphisms).
      Therefore, lisse-extended cobordism can be viewed as a well-behaved replacement for the latter theories which does admit the right-exact localization sequence.
      Moreover, using the higher groups it is also extends to the left.
      Indeed, we have:
    \end{exam}

    \begin{prop}\label{prop:07ug71}
      Let $i : Z \to X$ and $j : U \to X$ be a complementary closed-open pair of smooth $G$-equivariant $k$-schemes.
      Write $\sX = [X/G]$, $\sZ = [Z/G]$, and $\sU = [U/G]$ for the quotient stacks.
      Then for every $n\in\bZ$ there is a long-exact sequence
      \begin{multline*}
        \cdots \xrightarrow{\partial}
        \pi_s \Ccoh_\Bor(\sZ, \MGL)\vb{n-c}
        \xrightarrow{i_!} \pi_s \Ccoh_\Bor(\sX, \MGL)\vb{n}
        \xrightarrow{j^*} \pi_s \Ccoh_\Bor(\sU, \MGL)\vb{n}
        \xrightarrow{\partial}
        \\
        \cdots \xrightarrow{\partial}
        \H^{n-c}_\Bor(\sZ, \MGL)
        \xrightarrow{i_!} \H^{n}_\Bor(\sX, \MGL)
        \xrightarrow{j^*} \H^n_\Bor(\sU, \MGL)
        \to 0
      \end{multline*}
      where $c = \codim_X(Z)$.
    \end{prop}
    \begin{proof}
      In view of the localization triangle and relative purity for $\SH_\Bor$, we only need to demonstrate right-exactness of the last row.
      In other words, it is enough to show that $\pi_{-1} \Ccoh_\Bor(\sZ, \MGL)\vb{n} = 0$ for every $n$.
      For this we use the description in terms of the Borel construction (\thmref{thm:Borel}).
      By the Milnor exact sequence and the fact that the Mittag--Leffler condition holds for the projective system $\{\pi_0 \Ccoh(\sZ\fibprod_{BG}U_i, \MGL)\vb{n}\}_i$ (see \cite[Lem.~18]{HellerMalagonLopez}), we have
      \[ \pi_{-1} \Ccoh_\Bor(\sZ, \MGL)\vb{n} \simeq \lim_i \pi_{-1} \Ccoh(\sZ \fibprod_{BG} U_i, \MGL)\vb{n}. \]
      We can choose $U_i$ to be quasi-projective as in \examref{exam:071g20}, so the claim follows from the corresponding vanishing on smooth quasi-projective schemes (see e.g. \cite{Hoyois112064}).
    \end{proof}

    \begin{exam}
      In view of \examref{exam:071g20}, the lisse-extended cobordism ring of the classifying stack $B\bG_m$ or more generally $BGL_n$ can be computed as in \cite[Props.~3.4, 3.5]{Vezzosi}.
      In particular, in that case the canonical surjections \eqref{eq:07ugbp1j} are in fact bijective.
    \end{exam}

    \begin{exam}\label{exam:equiv k}
      There is a canonical ring homomorphism\footnote{%
        The $0$ in the target does not correspond to the $0$ in the source: we have $\H^n_\Bor([X/G], \KGL) \simeq \H^0_\Bor([X/G], \KGL)$ for all $n\in\bZ$ by Bott periodicity.
      }
      \[
        \K_0([X/G]) \to \H^0_\Bor([X/G], \KGL)
      \]
      which however is \emph{not} an isomorphism.
      In fact, \cite[Thm.~9.10]{KrishnaCompletion} (combined with \thmref{thm:Borel}) shows that it exhibits the target as the completion of the source along the augmentation ideal.
    \end{exam}

  \subsection{Lisse vs. Kan extension}
  \label{ssec:lim/kan}

    Let $\D^*$ be as in \ssecref{ssec:lim/lisse}.
    Define $\D^*_\mrm{Kan}$ as the right Kan extension of $\D^*$ from derived algebraic spaces to derived stacks.
    In this subsection we will show that the lisse extension $\D^*_\Bor$ agrees with $\D^*_\mrm{Kan}$ for a large class of stacks (see \corref{cor:Chowdhury}).
    This result was inspired by C.~Chowdhury's thesis \cite{Chowdhury}, and another proof of the comparison has appeared in \cite{ChowdhuryArXiv}.
    We will restrict our attention to classical (underived) stacks, since by derived invariance there is no loss of generality.

    \begin{defn}\label{defn:Chowdhury}\leavevmode
      \begin{defnlist}
        \item
        If $Y$ is an algebraic space, a \emph{Nis-Artin atlas} of $Y$ is a smooth morphism $f : X \to Y$ of algebraic spaces such that there exists an algebraic space $Y'$ and a Nisnevich cover $Y' \twoheadrightarrow Y$ for which the base change $X \fibprod_Y Y' \to Y'$ admits a section.

        \item
        A \emph{Nis-Artin atlas} is a smooth representable morphism of stacks $f : \sX \to \sY$ such that for any morphism $Y \to \sY$ with $Y$ an algebraic space, the base change $Y \fibprod_\sY \sX \to Y$ is a Nis-Artin atlas.

        \item
        A \emph{Nis-Artin stack} is a stack $\sX$ which has quasi-separated representable diagonal and admits a Nis-Artin atlas $p : X \to \sX$ where $X$ is a scheme.
      \end{defnlist}
    \end{defn}

    \begin{rem}\label{rem:afd0u1}\leavevmode
      \begin{defnlist}
        \item
        Any Nis-Artin stack is Artin.

        \item\label{item:afd0u1/;10g1}
        A morphism $f : \sX \to \sY$ is a Nis-Artin atlas if and only if it is smooth, representable, and surjective on field-valued points.
        This follows from \cite[Lem.~0.6]{equilisse}.
        
        \item
        Any algebraic space\footnote{recall our implicit quasi-separatedness convention} is a Nis-Artin stack, since it admits a Nisnevich cover by a scheme.
        See \cite[I, Prop.~5.7.6]{RaynaudGruson}.

        \item
        More generally, any quasi-separated Artin stack with separated diagonal is Nis-Artin; see \cite[Thm.~0.7]{equilisse}.

        \item
        One can also define \emph{higher} Nis-Artin stacks by analogy with higher Artin stacks: see \cite[0.2.2]{equilisse}.
        Our results in this and the following subsection continue to hold in that case.
      \end{defnlist}
    \end{rem}

    Let $\tau$ denote the representable Nisnevich topology on Nis-Artin stacks, which by \cite[Lem.~0.6]{equilisse} is generated by the pretopology in which covers are Nis-Artin atlases.

    \begin{prop}\label{prop:pubn-u-1}
      Let $\sC$ be the category of Nis-Artin stacks and $\sC_0$ the full subcategory of algebraic spaces (or affine schemes).
      Then for any \inftyCat $\sV$ admitting small limits, restriction along the inclusion $i : \sC_0 \to \sC$ induces an equivalence
      \[
        \Shv_\tau(\sC)_\sV
        \to \Shv_\tau(\sC_0)_\sV = \Shv_\Nis(\sC_0)_\sV
      \]
      on \inftyCats of $\sV$-valued $\tau$-sheaves.
    \end{prop}
    \begin{proof}
      As in the proof of \lemref{lem:9sydagy}, this follows from the fact that any Nis-Artin stack $\sX$ can be written as the colimit of the \v{C}ech nerve of a Nis-Artin atlas by a scheme.
      See also \cite[Thm.~3.4.1]{Chowdhury}.
    \end{proof}

    Since the inverse to the equivalence of \propref{prop:pubn-u-1} is given by right Kan extension, we have in particular:

    \begin{cor}\label{cor:0afgs01}
      On Nis-Artin stacks, $\D^*_\mrm{Kan}$ is the unique $\tau$-sheaf which restricts to $\D^*$ on algebraic spaces (or affine schemes).
    \end{cor}

    \begin{thm}\label{thm:Bor=Kan}
      Let $\sV$ be an \inftyCat admitting small limits, and $\sF$ a $\sV$-valued Nisnevich sheaf on algebraic spaces.
      Let $\sF_{\mrm{Kan}}$ denote the right Kan extension of $\sF$ from algebraic spaces to Artin stacks, and let $\sF_\Bor$ denote the lisse extension, i.e.
      \[
        \RGamma(\sX, \sF_\Bor) = \lim_{(T,t)\in\Lis_\sX} \RGamma(T,\sF).
      \]
      Then for any Nis-Artin stack $\sX$, there is a canonical isomorphism
      \[
        \RGamma(\sX, \sF_\Bor) \to \RGamma(\sX, \sF_{\mrm{Kan}})
      \]
      in $\sV$.
    \end{thm}
    \begin{proof}
      By \corref{cor:0afgs01} it is enough to show that $\sF_\Bor$ satisfies descent for Nis-Artin atlases.
      Let $p : \sU \twoheadrightarrow \sX$ be a Nis-Artin atlas.
      For every $(T,t) \in \Lis_\sX$, the base change $\sU_T = T \fibprod_\sX \sU \to T$ is a Nis-Artin atlas of algebraic spaces, hence generates a covering in the Nisnevich topology.
      Thus by Nisnevich descent for $\sF$ we have homotopy limit diagrams
      \[
        \RGamma(T, \sF) \to \RGamma(\sU_T, \sF)\rightrightarrows \RGamma(\sU_T \times_{T} \sU_T, \sF) \rightrightrightarrows \RGamma(\sU_T \times_{T} \sU_T \times_{T} \sU_T, \sF) \rightrightrightrightarrows \cdots.
      \]
      Passing to the limit over $(T,t)$ and using a cofinality argument yields that
      \[
        \RGamma(\sX, \sF_\Bor) \to \RGamma(\sU, \sF)\rightrightarrows \RGamma(\sU \times_{\sX} \sU, \sF) \rightrightrightarrows \RGamma(\sU \times_{\sX} \sU \times_{\sX} \sU, \sF) \rightrightrightrightarrows \cdots
      \]
      is also a homotopy limit diagram.
      The claim follows.
    \end{proof}

    \begin{cor}\label{cor:Chowdhury}
      Let $\sX$ be a Nis-Artin stack.
      Then there is a canonical equivalence of \inftyCats
      \[
        \D_\Bor(\sX) \to \D_\mrm{Kan}(\sX).
      \]
    \end{cor}

    \begin{rem}
      In \cite{Chowdhury}, Chowdhury takes $\SH^*_\mrm{Kan}$ as his definition of the stable motivic homotopy category on the class of Nis-Artin stacks.
      Thus \corref{cor:Chowdhury} shows that the lisse extension recovers his construction when the latter is defined.
    \end{rem}

  \subsection{Lisse-extended motivic homotopy types}
  \label{ssec:lim/mhtp}

    Let $\sX$ be a smooth Artin stack over an algebraic space $S$.
    Recall the lisse-extended stable motivic homotopy type $M_S^\Bor(\sX) \in \SH(S)$ defined in \notatref{notat:as0fg1}.
    We also have the unstable variant
    \[
      \L \h_S^\Bor(\sX) = \colim_{(T,t)\in\Lis_{\sX}} \L \h_S(T)
    \]
    in $\MotSpc(S)$, so that $M_S^\Bor(\sX) = \Sigma^\infty_+ \L \h_S^\Bor(\sX)$.
    Note that this definition still makes sense for $\sX$ singular.
    In this subsection we give some alternative descriptions of $\L\h_S^\Bor(\sX)$ and $M_S^\Bor(\sX)$ (\thmref{thm:bosy1}) in the case of Nis-Artin stacks (see \defnref{defn:Chowdhury}).

    \begin{constr}\label{constr:0pagf01}
      Let $\sX$ be an Artin stack locally of finite presentation over $S$.
      Let $\h_S(\sX)$ denote the $\Sm$-fibred animum over $S$ represented by $\sX$, i.e., the presheaf of anima
      \[T \in \Sm_{/S} \mapsto \Maps_{/S}(T, \sX).\]
      We then write
      \[ M_S^{\mrm{Yon}}(\sX) := \Sigma^\infty_+ \L \h_S(\sX) \in \SH(S) \]
      for the infinite $\bT$-suspension of its motivic localization.
    \end{constr}

    \begin{constr}\label{constr:oabb1}
      Consider the assignment $T \in \Sm_{/S} \mapsto M_S(T) \in \SH(S)$.
      Let $\sX \mapsto M^{\mrm{Kan}}_S(\sX)$ denote its left Kan extension along the inclusion from algebraic spaces locally of finite presentation over $S$ to Artin stacks locally of finite presentation over $S$.
      For example, if $X \twoheadrightarrow \sX$ is a Nis-Artin atlas, then by \propref{prop:pubn-u-1} we may write
      \[
        M_S^\mrm{Kan}(\sX) \simeq \abs{M_S(\Cech(X/\sX)_\bullet)}
        = \colim_{[n]\in\bDelta^\op} M_S(\Cech(X/\sX)_n)
      \]
      where $\Cech(X/\sX)_\bullet$ is the \v{C}ech nerve and $\abs{-}$ denotes geometric realization of simplicial objects.
      We similarly write $\L\h_S^{\mrm{Kan}}(\sX) \in \MotSpc(S)$ for the unstable analogue, defined by left Kan extending $T \in \Sm_{/S} \mapsto \L\h_S(T) \in \MotSpc(S)$.
    \end{constr}

    \begin{thm}\label{thm:bosy1}
      Let $\sX$ be an Artin stack locally of finite presentation over an algebraic space $S$.
      If $\sX$ is a Nis-Artin stack, then there are canonical isomorphisms
      \[
        \L\h_S^\Bor(\sX) \simeq \L\h_S^{\mrm{Kan}}(\sX) \simeq \L\h_S(\sX)
      \]
      in $\MotSpc(S)$, hence in particular
      \[
        M^\Bor_S(\sX) \simeq M^{\mrm{Kan}}_S(\sX) \simeq M^{\mrm{Yon}}_S(\sX)
      \]
      in $\SH(S)$.
    \end{thm}
    \begin{proof}
      The canonical isomorphism $\L\h^\Bor_S(\sX) \simeq \L\h^{\mrm{Kan}}_S(\sX)$ follows from \thmref{thm:Bor=Kan} (applied to $T \mapsto \L\h_S(T)$, regarded as a Nisnevich sheaf $\Sm_{/S}^\op \to \MotSpc(S)^\op$).
      Consider the canonical morphism
      \[
        \h_S^{\mrm{Kan}}(\sX) \to \h_S(\sX)
      \]
      where $\h_S^{\mrm{Kan}}(-)$ is defined by left Kan extension as in \constrref{constr:oabb1}.
      Choosing a Nis-Artin atlas $p : X \twoheadrightarrow \sX$, we may identify this with the canonical map of $\Sm$-fibred anima
      \[
        \abs{\h_S(\Cech(X/\sX)_\bullet)} \to \h_S(\sX)
      \]
      which is a Nisnevich-local equivalence by definition of Nis-Artin atlases and \cite[Cor.~6.2.3.5]{LurieHTT} (compare \cite[Vol.~I, Chap.~2, Lem.~2.3.8]{GaitsgoryRozenblyum} which is the étale analogue).
      Since motivic localization commutes with colimits, we get a canonical isomorphism
      \[
        \L\h_S^{\mrm{Kan}}(\sX)
        \simeq \abs{\L\h_S(\Cech(X/\sX)_\bullet)}
        \to \L\h_S(\sX)
      \]
      as claimed.
    \end{proof}

    \begin{rem}\label{rem:CDH comparison}
      Let $S$ be the spectrum of a field $k$ and consider the ``linearization'' functor
      \[
        \MotSpc(S) \to \on{\mbf{DM}}^{\mrm{eff}}(S)
      \]
      to the \inftyCat of effective Voevodsky motives over $S$ (with integral coefficients), see e.g. \cite[Ex.~5.2.17]{CisinskiDegliseBook} or \cite[\S 5.3]{EHKSY}.
      The image of $\L\h_S(\sX) \in \MotSpc(k)$ by this functor was defined as the \emph{Nisnevich motive} of $\sX$ in \cite{ChoudhuryDeshmukhHogadi}.
      Thus \thmref{thm:bosy1} shows that the lisse-extended motivic homotopy type $\L\h_S^\Bor(\sX)$ recovers the Nisnevich motive of \cite{ChoudhuryDeshmukhHogadi}.
      Therefore, it also recovers the étale-local construction of \cite{Choudhury} and, in the smooth proper case, the Chow motive constructed by To\"en \cite{ToenMotive} in view of \cite[Thm.~6.4]{Choudhury}.
      When $S$ is not the spectrum of a field, $\L\h_S^\Bor(\sX) \in \MotSpc(S)$ still gives rise to a relative (effective) motive over $S$ in the sense of \cite[Chap.~11]{CisinskiDegliseBook} or \cite[Chap.~9]{Spitzweck}, say.
    \end{rem}

    \begin{rem}\label{rem:ysbobql}
      \thmref{thm:bosy1} generalizes a computation of the motivic animum $\L\h_S(\sX) \in \MotSpc(S)$ obtained in \cite[Thm.~1.2]{ChoudhuryDeshmukhHogadi}.
      Indeed, whenever $p : X \twoheadrightarrow \sX$ is a Nis-Artin atlas, \thmref{thm:bosy1} yields an isomorphism
      \[ \L\h_S(\sX) \simeq \abs{\L\h_S(\Cech(X/\sX)_\bullet)} \]
      in $\MotSpc(S)$.
      In \cite[Thm.~1.2]{ChoudhuryDeshmukhHogadi}, this was proven when $S$ is the spectrum of a field $k$ and $\sX$ admits a representable Nisnevich cover $[U/\GL_n] \twoheadrightarrow \sX$, for the Nis-Artin atlas $U \twoheadrightarrow [U/\GL_n] \twoheadrightarrow \sX$.
    \end{rem}

  \subsection{Exhaustive stacks}
  \label{ssec:lim/exh}

    In \cite{HoskinsPepinLehalleur}, Hoskins and Pepin Lehalleur adapt the algebraic approximation to the Borel construction of Totaro and Morel--Voevodsky to a certain class of \emph{exhaustive} algebraic stacks.
    In this subsection, we record a generalization of the material of \ssecref{ssec:Borel homotopy} to this setting.

    \begin{defn}
      Let $\sX$ be an algebraic stack.
      A \emph{system of approximations} $(\sX_i, \sV_i, \sW_i)_i$ for $\sX$ consists of:
      \begin{defnlist}
        \item
        An exhaustive filtration of $\sX$ by an increasing sequence of quasi-compact opens $\sX_i \sub \sX$.

        \item
        For every $i$, a vector bundle $\sV_i$ over $\sX_i$ and a closed substack $\sW_i \sub \sV_i$.

        \item
        For every $i$, a monomorphism of vector bundles $\phi_i : \sV_i \to \sV_{i+1} \cap {\sX_i}$ over $\sX_i$.
      \end{defnlist}
      This data is subject to the following conditions:
      \begin{defnlist}[label=(\alph*), ref=(\alph*)]
        \item
        For every $i$, the open complement $U_i = \sV_i \setminus \sW_i$ is representable by a quasi-separated algebraic space.

        \item
        For every $i$, the morphism $\phi_i$ sends $U_i \sub \sV_i$ to $U_{i+1} \cap \sX_i \sub \sV_{i+1} \cap \sX_i$.

        \item
        Given any integer $n\ge 0$, there exists an $i\gg 0$ such that $\codim_{\sV_i}(\sW_i) > n$.
      \end{defnlist}
    \end{defn}

    We have the following generalization of \thmref{thm:Borel M}, with the same proof.
    When $S$ is the spectrum of a field, this can alternatively be deduced by combining the comparison of \thmref{thm:bosy1} with \cite[Prop.~5.4]{ChoudhuryDeshmukhHogadi}.

    \begin{thm}\label{thm:0gb0111}
      Let $\sX$ be a smooth algebraic stack over an algebraic space $S$.
      Given any system of approximations $(\sX_i, \sV_i, \sW_i)_i$ for $\sX$, the canonical morphism in $\SH(S)$
      \[
        \colim_i M_S(U_i) \to
        M_S^\Bor(\sX)
      \]
      is invertible, where $U_i = \sV_i \setminus \sW_i$ and the transition morphisms in the colimit are induced by the composites
      \[
        U_i \xrightarrow{\phi_i} U_{i+1} \cap \sX_i \hook U_{i+1}.
      \]
    \end{thm}

    \begin{rem}\label{rem:UO1b2hov}
      Let $S$ be the spectrum of a field $k$.
      In \cite[Def.~2.15]{HoskinsPepinLehalleur}, $\sX$ is said to be \emph{exhaustive} if it admits a system of approximations as above, where $U_i$ are representable by separated schemes of finite type over $k$.
      For $\sX$ exhaustive, the right-hand side of \thmref{thm:0gb0111}, or rather its image by the ``linearization'' functor (see \remref{rem:CDH comparison})
      \[
        \MotSpc(S) \to \on{\mbf{DM}}^{\mrm{eff}}(S),
      \]
      is taken to be the definition of the Voevodsky motive of $\sX$ in \cite[Def.~2.17]{HoskinsPepinLehalleur}.
      Combining Theorems~\ref{thm:0gb0111} and \ref{thm:bosy1}, we find that for exhaustive stacks, all existing constructions of the motive or motivic homotopy type agree.
      For example, this applies to moduli stacks of vector bundles on curves (see \cite[Thm.~3.2]{HoskinsPepinLehalleur}).
      Compare the étale-local comparison of \cite[Prop.~A.7]{HoskinsPepinLehalleur}.
    \end{rem}

\changelocaltocdepth{1}
\appendix

\section{Linearly scalloped stacks}
\label{sec:linearly}

  In this appendix we explain how our results can be generalized from \emph{nicely} to \emph{linearly} scalloped stacks, to include for example arbitrary quotients by reductive groups in characteristic zero.

  In the nicely scalloped case our proofs and constructions are relatively clean thanks in large part to the work of \cite{AlperHallRydh} and \cite{AlperHallHalpernLeistnerRydh}, in particular the local structure theorem for nicely scalloped stacks (\thmref{thm:scallop}).
  As we will now see, it is possible to reach essentially the same results, at the cost of somewhat more complicated definitions and arguments, by relying instead on Hoyois's work on equivariant motivic homotopy theory \cite{HoyoisEquivariant}.
  This has the advantage that it will include the case of linearly scalloped stacks, while also reducing much of our reliance on the not yet published works \cite{AlperHallRydh,AlperHallHalpernLeistnerRydh}.

  We have chosen to write this separately from the main text in order to keep the exposition as readable as possible.
  After giving some background on the definition of linearly scalloped stacks, we will go through sections~\ref{sec:H}--\ref{sec:ex} and explain all modifications necessary for the linearly scalloped case.
  For simplicity, we will work with classical stacks throughout (but it is easy to generalize to derived stacks, using the notion of quasi-projectivity developed in \cite{resdag}).

  Note that \secref{sec:fix} only deals with quotients by algebraic tori, which are nicely scalloped, so there is nothing to change there.
  Of course there is nothing to change in \secref{sec:lim} either.

  \ssec{Linearly scalloped stacks}

    \begin{defn}\label{defn:linearly fund}
      Let $\sX$ be a quasi-compact quasi-separated algebraic stack.
      \begin{defnlist}
        \item
        We say that $\sX$ is \emph{linearly \fund} if it admits an \emph{affine} morphism $\sX \to BG$ for some linearly reductive embeddable group scheme $G$ over an affine scheme $S$.

        \item
        We say that $\sX$ is \emph{linearly \qfund} if it admits a \emph{quasi-projective} morphism $\sX \to BG$ for some linearly reductive embeddable group scheme $G$ over an affine scheme $S$.

        \item
        We say that $\sX$ is \emph{linearly scalloped} if it has separated diagonal and admits a scallop decomposition $(\sU_i, \sV_i, u_i)_i$ where $\sV_i$ are linearly \qfund and $u_i$ are quasi-projective.
      \end{defnlist}
    \end{defn}

    \begin{rem}\label{rem:ysfg913o}
      Any linearly \qfund stack $\sX$ satisfies the resolution property.
      This follows from the fact that $BG$ has the resolution property for $G$ linearly reductive and embeddable (\examref{exam:0g11up0}), and the fact that $\sX$ is quasi-projective over such a $BG$ and in particular admits a family of line bundles which is relatively ample over $BG$.
    \end{rem}
    
    \begin{warn}
      Every nicely \fund (resp. \qfund) stack is of course linearly \fund (resp. \qfund).
      Nicely scalloped stacks with affine diagonal are also linearly scalloped by \thmref{thm:scallop}\itemref{item:scallop/cover2}.
      Note that without the affine diagonal condition, this is not clear because given a scallop decomposition $(\sU_i, \sV_i, u_i)_i$ as in \defnref{defn:scallop}, the $u_i$'s may not be quasi-projective in general.
    \end{warn}

    \begin{exam}
      Let $G$ be a linearly reductive group scheme over an affine scheme $S$ and $X$ a $G$-quasi-projective scheme.
      Then the quotient stack $\sX = [X/G]$ is linearly scalloped.
      Indeed, while $G$ need not be embeddable, it is always locally so by \cite[Cor.~13.2]{AlperHallRydh}.
    \end{exam}

    \begin{exam}\label{exam:a0gpfsd71}
      Let $\sX$ be a qcqs algebraic stack with affine diagonal and finitely many characteristics.
      If all stabilizers of $\sX$ are linearly reductive, then $\sX$ is linearly scalloped.
      This follows from \cite[Thm.~1.13]{AlperHallHalpernLeistnerRydh}.
    \end{exam}

    \begin{rem}
      The analogue of \thmref{thm:scallop} is not available to us in the linearly scalloped case.
      For example, a linearly scalloped stack need not have all stabilizers linearly reductive.
    \end{rem}

    \begin{rem}
      Linearly \fund stacks are also linearly fundamental in the sense of \cite[Defn.~2.7]{AlperHallRydh}, but not conversely.
      See \cite[Rem.~2.9, App.~A.1]{AlperHallRydh}.
    \end{rem}

  \ssec{The unstable category}

    We now describe the modifications necessary throughout the text, starting with \secref{sec:H}, if we want to replace all instances of ``nice group scheme'' by ``linearly reductive group scheme'', ``\fund'' by ``linearly \fund'', ``scalloped'' by ``linearly scalloped'', and so on.

    \sssec{The site $\Sm_{/\sX}$.}

      Given a linearly scalloped stack $\sX$, $\Sm_{/\sX}$ will denote the \inftyCat of stacks $\sX'$ which are smooth and \emph{quasi-projective} over $\sX$.
      Note that $\sX'$ is then linearly scalloped (a scallop decomposition for $\sX$ gives one for $\sX'$ by base change).

    \sssec{The \inftyCat of motivic spaces.}

      We define $\MotSpc(\sX)$ as in \defnref{defn:H}, except that a presheaf $\sF$ of \anis on $\Sm_{/\sX}$ is \emph{homotopy invariant} if for every $\sX' \in \Sm_{/\sX}$ and every vector bundle \emph{torsor} $\pi : \sV \to \sX'$, the map of \anis
      \[ \pi^* : \RGamma(\sX', \sF) \to \RGamma(\sV, \sF) \]
      is invertible.

      If $\sX$ happens to be nicely scalloped, then this is equivalent to the \inftyCat $\MotSpc(\sX)$ we already defined in \ssecref{ssec:H/constr}.
      There are two differences to be reconciled:
      \begin{enumerate}
        \item 
        The homotopy invariance condition is \textit{a priori} stronger in the linearly scalloped setting.
        However, the stronger condition involving vector bundle torsors is in fact automatic in the nicely scalloped case.
        Indeed, any nicely scalloped $\sX$ admits a Nisnevich cover $u : \sU \twoheadrightarrow \sX$ with $\sU$ nicely \fund (\thmref{thm:scallop}\itemref{item:scallop/cover}).
        Then if $\pi : \sV \to \sX$ is a vector bundle torsor, the base change $\sV \fibprod_\sX \sU \to \sU$ admits a section, and hence is a vector bundle.

        \item
        In the nicely scalloped case, we defined $\Sm_{/\sX}$ to consist only of smooth stacks that are \emph{representable}, but not necessarily quasi-projective, over $\sX$.
        However, when $\sX$ is nicely scalloped both variants give rise to the same \inftyCat $\MotSpc(\sX)$ (up to equivalence).
        In fact, we claim that the inclusions
        \[
          \Sm^\mrm{qaff}_{/\sX} \hook
          \Sm^\mrm{qproj}_{/\sX} \hook
          \Sm^{\mrm{repr}}_{/\sX}
        \]
        of the sites of smooth \emph{quasi-affine}, smooth \emph{quasi-projective}, and smooth \emph{representable} stacks over $\sX$, both induce equivalences on $\A^1$-invariant Nisnevich sheaves (by left Kan extension).
        To see this recall that the assignment $\sX \mapsto \Sm^?_{/\sX}$ is a Nisnevich (or even étale) sheaf for $? \in \{\mrm{qaff}, \mrm{repr}\}$, which easily implies that the corresponding versions of $\sX \mapsto \MotSpc(\sX)$ are both Nisnevich sheaves.
        Therefore to show that the induced functor is an equivalence, it is enough by \thmref{thm:scallop}\itemref{item:scallop/cover} to consider the case where $\sX$ is nicely \qfund.
        This easily follows from the fact that every $\sX' \in \Sm^\mrm{repr}_{/\sX}$ is isomorphic (up to Nisnevich localization) to a colimit of objects in $\Sm^\mrm{qaff}_{/\sX}$ by \thmref{thm:sumihiro}.
      \end{enumerate}

    \sssec{Generation}

      The proof of the direct analogue of \propref{prop:gen} requires no modifications.
      However, we will require the sharper statement that $\MotSpc(\sX)$ is generated under sifted colimits by $\L\h_\sX(\sX')$ with $\sX'$ linearly \emph{\fund} (and not just linearly \qfund).
      To prove this, it is enough to show that for every linearly \qfund $\sX' \in \Sm_{/\sX}$, there is a linearly \fund $\sX'' \in \Sm_{/\sX}$ such that $\L\h_\sX(\sX') \simeq \L\h_\sX(\sX'')$.
      This follows from the following variant of \cite[Prop.~2.20]{HoyoisEquivariant}, which will be a substitute for \thmref{thm:locally affine}.

      \begin{prop}[Jouanolou device]\label{prop:jou}
        Let $f : \sX \to \sY$ be a quasi-projective morphism with $\sY$ \qfund.
        Then there exists an affine bundle $\pi : \sV \to \sX$ such that the composite $\sV \to \sX \to \sY$ is affine.
      \end{prop}
      \begin{proof}
        Write $\sY = [Y/G]$, where $G$ is a linearly reductive embeddable group scheme over an affine scheme $S$, and $Y$ is a $G$-quasi-projective scheme.
        By \remref{rem:ysfg913o}, $\sY$ has the resolution property.
        By \cite[Prop.~2.20]{HoyoisEquivariant} we therefore have the result in the special case $\sX = \P_{\sY}(\sE)$ for $\sE$ a finite locally free sheaf over $\sY$.
        
        For $f$ projective, choose an embedding into a projective bundle $\P_\sY(\sE)$ where $\sE$ is a finite locally free sheaf on $\sY$.
        By the previous case there exists an affine bundle $\pi_0 : \sV_0 \to \P_{\sY}(\sE)$ such that $\sV_0 \to \sY$ is affine.
        Then its base change $\pi : \sV \to \sX$ along the closed immersion $\sX \to \P_{\sY}(\sE)$ is an affine bundle, and $\sV \hook \sV_0$ is a closed immersion (hence in particular affine).

        Finally for $f$ quasi-projective, since $\sY$ has the resolution property, we may choose a factorization through an \emph{affine} open immersion and a projective morphism.
        Then we may repeat the same argument to reduce to the projective case.
      \end{proof}

      Similarly, \propref{prop:gen}\itemref{item:0h103n} becomes: for every quotient $\sX = [X/G]$, where $G$ is a linearly reductive group scheme over an affine scheme $S$ and $X$ is a $G$-quasi-projective scheme over $S$, $\MotSpc(\sX)$ is generated under sifted colimits by $\L\h_\sX([U/G])$ with $U$ an affine $G$-scheme, smooth over $X$.
      For this we can again apply \propref{prop:jou} if $G$ is embeddable.
      Otherwise, let $S' \twoheadrightarrow S$ be a Nisnevich cover with $S'$ an affine scheme such that $G' = G \fibprod_S S'$ is embeddable (such exists by \cite[Cor.~13.2]{AlperHallRydh}).
      Then for every $\sY = [Y/G] \in \Sm_{/\sX}$ (where $Y$ is $G$-quasi-projective), $\L\h_{\sX}(\sY)$ is by Nisnevich descent the geometric realization of a simplicial diagram of objects of the form $\L\h_{\sX}(\sY')$ where $\sY' = [Y'/G'] \simeq [Y'/G]$ is quasi-projective over $B_{S'}(G') \simeq BG \fibprod_S S' \simeq [S'/G]$ (where $G$ acts trivially on $S'$).
      We conclude by the embeddable case.

    \sssec{Comparison with Hoyois}

      In \remref{rem:Hoyois}, we now get the comparison
      \[ \MotSpc([X/G]) \simeq \MotSpc^G(X) \]
      for all $G$-quasi-projective schemes $X$ over $S$, with $G$ a linearly reductive group scheme over an affine scheme $S$, where the right-hand side is Hoyois's construction \cite{HoyoisEquivariant}.
      This follows from the variant of \propref{prop:gen}\itemref{item:0h103n} proven just above and \cite[Prop.~3.16(1)]{HoyoisEquivariant}.

    \sssec{Functoriality}

      In Propositions~\ref{prop:f_sharp} and \ref{prop:smooth bc}, replace ``smooth representable'' by ``smooth quasi-projective.''
      In \propref{prop:Nis sep}, assume the Nisnevich cover is by quasi-projective morphisms.
      For later use let us record the following (which is a reformulation of homotopy invariance for vector bundle torsors):

      \begin{prop}\label{prop:?v,hy9y1}
        Let $\sX$ be a linearly scalloped stack.
        Then for every vector bundle torsor $\pi : \sV \to \sX$, the inverse image functor
        \[ \pi^* : \MotSpc(\sX) \to \MotSpc(\sV) \]
        is fully faithful (in particular, conservative).
      \end{prop}

    \sssec{Exactness of $i_*$}

      We have the analogue of \thmref{thm:lifting closed} in the special case where $f_0$ is affine.
      In view of \propref{prop:gen}\itemref{item:0h103n} (the form proven above), this will imply \propref{prop:i_* colimits}.
      This is a small (but much easier) variant on the proof of \cite[Prop.~6.1]{AlperHallHalpernLeistnerRydh}; for simplicity we sketch the argument in the case where $\sZ = \sX_\cl$ and $i$ is the canonical inclusion, which is all we will need for \thmref{thm:derinv}.

      Since $\sX$ is linearly \fund, there exists an affine morphism $\sX \to BG$ where $G$ is a linearly reductive group scheme over an affine scheme $S$.
      Write $\sX$ as the colimit of its $n$-truncations $\tau_{\le n}(\sX)$ (see e.g. \cite[Prop.~6.1]{AlperHallHalpernLeistnerRydh}).
      Since the relative cotangent complex $\sL_{\sZ'/\sZ}$ is projective (since $f_0 : \sZ' \to \sZ$ is smooth and $BG$ is cohomologically affine), there is no obstruction to lifting $f_0$ to smooth (resp. étale) morphisms $f_n : \sX'_{\le n} \to \tau_{\le n}(\sX)$ for all $n$.
      We can also extend the affine morphism $p_0 : \sZ' = \tau_{\le 0}(\sX') \to BG$ to $p_n : \sX'_{\le n} \to BG$, since the obstruction to this lifting again vanishes because the cotangent complex of $BG$ is of Tor-amplitude $[-1, 0]$ (with homological grading) and $BG$ is cohomologically affine.
      Finally, let $\sX'$ be the colimit of $\sX'_{\le n}$ over $n$, i.e.,
      \[ \sX' = \colim_n \sX'_{\le n} \simeq \Spec_{BG}\big( \lim_n p_{n,*}\sO_{\sX'_{\le n}} \big). \]
      Then there is an affine morphism $\sX' \to BG$ by construction (so $\sX'$ is linearly \fund), and we conclude by passing to the colimit over $n$ of the squares
      \[\begin{tikzcd}
        \sZ' \ar{d}{f_0}\ar{r}
        & \sX'_{\le n} \ar{d}{f_n}
        \\
        \sZ = \sX_\cl \ar{r}
        & \tau_{\le n}(\sX).
      \end{tikzcd}\]

    \sssec{Localization}

      No modifications necessary in \thmref{thm:loc unstable}.

  \ssec{The stable category}

    Throughout \secref{sec:SH}, ``representable morphism'' should be changed to ``quasi-projective morphism''.
    In particular, we get the analogue of \thmref{thm:funct SH} except that in part~\itemref{item:funct SH/smooth} we only look at smooth quasi-projective morphisms.
    Moreover, any reference to the ``Nisnevich topology'' or ``Nisnevich descent'', etc. refers now to the topology generated by \emph{quasi-projective} Nisnevich covers (instead of representable ones).

    The necessary modifications in this section are as follows:

    \sssec{Proof of \lemref{lem:basic stab}}

      The proof in the (quasi-)affine case requires no modifications.
      For the general (quasi-projective) case, we use the following claim: if $\pi : \sV \to \sX$ is a vector bundle torsor such that the claim holds for $f_\sV = f \circ \pi$, then it holds for $f$ itself.
      This follows by the same argument with \propref{prop:?v,hy9y1} substituted for Nisnevich separation.
      Now by \propref{prop:jou} we are reduced to the affine case.

    \sssec{Smooth base change, \thmref{thm:funct SH}\itemref{item:funct SH/smooth}\itemref{item:funct SH/smooth/bc} (see \sssecref{sssec:sf0g61oyl1})}

      We have the following variant of \lemref{lem:a0s7dft}: if $\pi : \sV \to \sX$ is a vector bundle torsor such that $(p\circ \pi)_\sharp$ satisfies base change against $g^*$, then so does $p_\sharp$.
      The proof is the same except we use the isomorphism $p_\sharp \simeq (p\circ \pi)_\sharp \pi^*$ (\propref{prop:?v,hy9y1}) and similarly for $q$, instead of the ``descent'' isomorphisms.

      In the proof of Case 1, we have appealed to \thmref{thm:sumihiro}\itemref{item:sumihiro/scallop}.
      This is not available in the linearly scalloped case but we can use \propref{prop:jou} instead: take $\pi : \sV \to \sX$ a vector bundle torsor such that $p \circ \pi : \sV \to \sY$ is affine.
      Combining this with the statement just mentioned proves Case 1.

      Cases 2 and 3 then go through without modification.

    \sssec{Smooth projection formula, \thmref{thm:funct SH}\itemref{item:funct SH/smooth}\itemref{item:funct SH/smooth/bc} (see \sssecref{sssec:a0sgu0p1})}

      Instead of appealing to \lemref{lem:a0s7dft} we argue as in the above proof of smooth base change, using the isomorphism $p_\sharp \simeq (p\circ \pi)_\sharp \pi^*$ for $\pi : \sV \to \sX$ a vector bundle torsor with $p \circ \pi$ affine (\propref{prop:jou}).

  \ssec{Proper base change}

    No modifications necessary in \secref{sec:prop}.
    Note that \thmref{thm:Chow} holds also for $\sY$ linearly scalloped, since linearly scalloped stacks are also of global type in the sense of \cite{RydhApprox}.

  \ssec{The \texorpdfstring{$!$}{!}-operations}

    In \secref{sec:shriek}, the $!$-operations are defined by the same procedure for representable morphisms of finite type which are \emph{Nisnevich-locally compactifiable} in the sense that there exists a commutative square as in \remref{rem:repr loc compact} with $f_0$ compactifiable and $u,v$ quasi-projective Nisnevich covers.
    For example, any quasi-projective morphism of linearly scalloped stacks $f : \sX \to \sY$ is compactifiable (by definition).

    Throughout \ssecref{ssec:purity} we replace the assumption that ``$f : \sX \to \sY$ is compactifiable or $\sX$ and $\sY$ have affine diagonal'' with the assumption that $f$ is Nisnevich-locally compactifiable in the above sense.
    Compare \remref{rem:)&gh1}.

  \ssec{The Euler and Gysin transformations}

    In \secref{sec:gys} we replace the \emph{representably smoothable} assumptions (\defnref{defn:smoothable}) by \emph{quasi-projectively smoothable}, meaning that $f$ factors through a closed immersion and a \emph{quasi-projective} smooth morphism.
    Note that if $\sY$ has the resolution property, then any quasi-projective morphism $f : \sX \to \sY$ is quasi-projectively smoothable.

    Since quasi-projectively smoothable morphisms are compactifiable, we can then drop the affine diagonal assumptions on $\sX$ and $\sY$ (thanks to the modifications to \ssecref{ssec:purity} described just above).

  \ssec{Cohomology and Borel--Moore homology}

    In \secref{sec:coh} one should make the obvious modifications.
    For example, whenever a $!$-operation appears (e.g. in the definition of Borel--Moore homology), we should assume that $f$ is Nisnevich-locally compactifiable.
    When the word ``representably smoothable'' appears, it should be replaced with ``quasi-projectively smoothable''.
    On the other hand all affine diagonal assumptions can be dropped (e.g. in \defnref{defn:fund smooth} and \propref{prop:poincare}).

  \ssec{Examples}

    In \secref{sec:ex} we make the following modifications.

    \sssec{Homotopy invariant K-theory}

      In \ssecref{ssec:KH}:
      \begin{enumerate}
        \item 
        We claim that Nisnevich descent for $\KB$ (\thmref{thm:Thomason}) still holds on linearly scalloped stacks.
        For this we just need to verify that \thmref{thm:perfect} goes through.
        The only modification necessary is in the proof of \lemref{lem:spdufg0p1}, where we have to replace ``quasi-affine morphism $f : \sX \to BG$'' by ``quasi-projective''.
        Then the functor $f^* : \Dqc(BG) \to \Dqc(\sX)$ does not typically generate under colimits anymore.
        Nevertheless, if we choose a quasi-compact immersion $i : \sX \hook \P_{BG}(\sE)$, then we claim that the functors $f^*(-) \otimes i^*(\sO(k))$ do jointly generate under colimits (as $k\in\bZ$ varies).
        Indeed, since $i^*$ generates under colimits, it suffices to show that $g^*(-) \otimes \sO(k) : \Dqc(BG) \to \Dqc(\P_{BG}(\sE))$ generate under colimits, where $g : \P_{BG}(\sE) \to BG$ is the projection.
        This follows from \cite[Thm.~3.3]{KhanKblow}.

        \item
        Remarks~\ref{rem:0afsgdh013} and \ref{rem:KH cdh} generalize our \corref{cor:intro/KH} and \corref{cor:intro/KH cdh} to the linearly scalloped case.
        Note that this extends the cdh descent theorem of \cite{HoyoisKrishna} to the case of stacks with linearly reductive but possibly \emph{infinite} stabilizer groups, at least assuming affine diagonals (see \examref{exam:a0gpfsd71}).

        \item
        In \thmref{thm:KGL}\itemref{item:0asfh08} we should replace ``representable'' by ``quasi-projective''.
      \end{enumerate}

    \sssec{Algebraic cobordism}
      
      In \ssecref{ssec:coh/MGL}, ``smooth representable'' should be replaced by ``smooth quasi-projective'' in \constrref{constr:07sg113} and \propref{prop:dogo011}.

    \sssec{Motivic cohomology}

      In \ssecref{ssec:coh/Z}:
      \begin{enumerate}
        \item 
        Add the assumptions in \defnref{defn:07g1ufdl} that $f$ is quasi-projectively smoothable and $g$ is quasi-projective.
        (The former is superfluous if $\sX$ admits the resolution property, e.g. if it is linearly \qfund.)
      
        \item
        In \constrref{constr:afsdo0by}, change ``representable'' to ``quasi-projective''.
      \end{enumerate}



\bibliographystyle{halphanum}

Institute of Mathematics, Academia Sinica, Taipei 10617, Taiwan

School of Mathematics, Tata Institute of Fundamental Research, Homi Bhabha Road, Colaba, Mumbai 400005, India

\end{document}